\documentclass{amsart}
\renewcommand\baselinestretch{1.2}

\usepackage{helvet}
\usepackage{amsmath, amsfonts, amssymb} \usepackage[english]{babel}
\usepackage[cmyk]{xcolor}      \usepackage{graphicx}   \usepackage{tikz}
\usetikzlibrary{arrows.meta,calc,positioning}
\usepackage{url}        \usepackage{bm}         \usepackage{multirow}
\usepackage{array}
\usepackage{booktabs}
\usepackage{enumitem}
\usepackage{epsfig}
\usepackage{algorithm}
\usepackage{algorithmic}
\usepackage{latexsym, mathrsfs, subfigure, stmaryrd}

\usepackage[numbers,sort&compress]{natbib}
\usepackage[colorlinks=true,linkcolor=blue!55!black,citecolor=blue!55!black,urlcolor=blue!55!black]{hyperref} 

\usepackage{appendix}

\theoremstyle{plain}
\newtheorem{theorem}{Theorem}[section]

\newtheorem{lemma}[theorem]{Lemma}
\newtheorem{proposition}[theorem]{Proposition}

\newtheorem*{remark}{Remark}
\newtheorem{definition}[theorem]{Definition}
\newtheorem{assumption}[theorem]{Assumption}

\numberwithin{equation}{section}

\usepackage{geometry}
\newcommand{\red}{\color{red}}

\newcommand\la{\langle}
\newcommand\ra{\rangle}

\newcommand{\mc}[1]{\mathcal{#1}}
\newcommand{\mb}[1]{\mathbb{#1}}
\newcommand{\ms}[1]{\mathscr{#1}}
\newcommand{\mr}[1]{\mathrm{#1}}
\newcommand{\mf}[1]{\mathbf{#1}}

\newcommand{\wh}[1]{\widehat{#1}}

\begin{document}
\title[Spectral Convergence of RFM in Multiple Dimensions]{Spectral Convergence of Random Feature Method in Multiple Dimensions}

\author[P. Ming]{Pingbing Ming}
\author[H. Yu]{Hao Yu}
\address{SKLMS, Institute of Computational Mathematics and Scientific/Engineering Computing, Academy of Mathematics and Systems Science, Chinese Academy of Sciences, Beijing 100190, China} \address{School of Mathematical Sciences, University of Chinese Academy of Sciences, Beijing 100049, China}
\email{mpb@lsec.cc.ac.cn} \email{yuhao@amss.ac.cn}
\date{}
\begin{abstract}
We first prove spectral convergence of the random feature method (RFM) for multidimensional targets in Sobolev, Gevrey, ultra-analytic, and bandlimited classes.
The analysis establishes general high-probability approximation estimates in the interpolation scale generated by a kernel integral operator. On a single event determined only by the sampled features, one random space approximates every target in a prescribed source ball; moreover, for each target, a single coefficient vector defines an approximant that attains spectral accuracy simultaneously in all admissible error norms. For both regularity-adapted frequency distributions and uniform distributions on growing frequency windows, the resulting rates range from super-exponential to algebraic, depending on the regularity of the target. Second, we establish abstract error estimates for strong- and weak-form RFM discretizations, thereby converting the preceding approximation bounds into convergence estimates for multidimensional second-order elliptic boundary value and eigenvalue problems. Finally, for random feature matrices (RFMtxs), we prove super-exponential singular-value decay with Fourier features and exponential decay with $\tanh$ features, together with corresponding condition-number lower bounds. The analysis identifies a common mechanism: the same spectral approximation that yields high accuracy also drives severe ill-conditioning.
\end{abstract}

\maketitle

\section{Introduction} \label{Section: Introduction}

Recently, learning-based methods have emerged as an active paradigm for solving partial differential equations (PDEs). Inherently meshfree, these methods enjoy the flexibility to handle complex geometries and high-dimensional problems, facilitating the integration of experimental data. However, current learning-based PDE solvers still face significant challenges, such as high computational costs and a lack of strategies to consistently improve accuracy.

A parallel line of work analyzes learning-based and data-driven discretizations through the classical numerical-analysis concepts of approximation, stability, and conditioning. Random features (RFs) provide a particularly transparent setting for this program. Their hidden parameters are sampled from a prescribed distribution and only the output coefficients are optimized, so training reduces to a linear least-squares problem. Introduced as scalable approximations of kernel machines \cite{rahimi2007random}, RFs now form a useful interface between kernel approximation, randomized numerical linear algebra, and shallow neural networks. These developments lead to a unified question: once the trial space and sampling mechanism are specified, what are the approximation rate, the stability of the residual map, and the conditioning of the resulting algebraic system?

Against this background, recent efforts have sought to bridge traditional numerical solvers with machine learning through extreme learning machines, the random feature method (RFM), and related randomized neural networks. The core idea is to approximate the solution by a neural network with prescribed inner-layer weights, thereby reducing training to a linear least-squares problem for the outer-layer weights. Representative examples include RFM discretizations for stationary and time-dependent problems \cite{ChenChiEYang2022RFM,ChenELuo2023TimeRFM} and dimension-robust solvers for Kolmogorov equations \cite{Gonon2023BlackScholes}. A hallmark of these approaches is the spectral or near-spectral accuracy observed even on complicated geometries \cite{ChenChiEYang2022RFM,DongLi2021LocalELM,SunDongWang2024LRNNDG}.

A critical issue, however, is the severe ill-conditioning of the random feature matrix (RFMtx) in high-accuracy computations \cite{ChenESun2024Optimization}.
Moreover, a multidimensional convergence theory must answer three questions simultaneously: how the sampling law should reflect the regularity of the target, whether one sampled space works uniformly over an entire target class, and whether the same reconstruction controls the several derivative norms required by PDE stability estimates. Standard fixed-target or fixed-norm bounds do not provide this combination, nor do they explain spectral approximation and rapid singular-value decay within a common framework.

Our analysis is built on concentration of the empirical feature operator around its population counterpart, an approach developed for kernel quadrature and RF regression \cite{bach2017equivalence,rudi2017generalization}. The central new ingredient is an abstract estimate on the interpolation scale of the kernel integral operator that preserves the quantifiers needed for simultaneous approximation. Combined with an effective-dimension bound, this estimate yields a single high-probability event, independent of the target, on which every target admits one reconstruction that works simultaneously across all admissible error norms. The resulting approximation theory also supplies the decisive low-dimensional approximation mechanism in our analysis of singular-value decay.

\subsection{Our contributions}

The main results can be summarized as follows.

\begin{enumerate}[leftmargin=*]
\item \emph{Simultaneous, target-uniform spectral approximation.}
Theorem~\ref{Theorem: abstract error estimate for regression in interpolation spaces} supplies the abstract approximation principle which, combined with the effective-dimension estimates in Section~\ref{Section: Estimate of d(lambda) and the convergence rate, interpolation}, yields the following twofold uniformity.
With probability at least $1-\delta$, a single sampled space contains a spectrally accurate approximant for every target in a source ball, and each target admits a single coefficient vector, independent of the error norm, whose associated approximant attains spectral accuracy simultaneously in all admissible error norms. This error-norm simultaneity is enabled by formulating the ridge problem in $\mc{H}^{\theta}$; when $\theta>0$, this regression norm is stronger than $L^2=\mc{H}^{0}$. In particular, if $p\le\theta$, the same $\mc{H}^{\theta}$-ridge reconstruction and the same target-independent event control the entire admissible weaker-norm scale. With $\bar p=\max\{p,2\theta-1\}$, the rate is governed by $s-\bar p$ and saturates below $p=2\theta-1$. If $p\ge\theta$, the stronger error norm retains the unsaturated rate $\lambda^{(s-p)/(2(1-\theta))}$, but the sampling condition depends on $p$ through $\gamma=1-p$. This distinction identifies when one feature sample controls a full norm scale and when stronger derivative estimates require a stronger sampling condition, which is essential for the subsequent PDE analysis.

\item \emph{Multidimensional spectral rates for random Fourier features.}
Theorems~\ref{Theorem: interpolation improved convergence rate of RFM} and \ref{Theorem: optimal growing bandwidth leverage rates} establish multidimensional rates for both regularity-adapted reference measures and uniform measures on explicitly growing frequency windows. The latter is a particularly simple sampling strategy widely used in practice. The resulting rates and the corresponding sampling strategies are summarized in Table~\ref{Table: comparison of reference measures and sample size rates}. Importantly, the error estimates are not restricted to the conventional $L^2$ norm; they hold simultaneously in the general Sobolev norms $W^{t,p}(\Omega)$, $1\le p\le\infty$.

\end{enumerate}

\begin{table}[!htbp]
\centering
\small
\caption{Sample-size convergence rates and reference frequency measures in Theorems~\ref{Theorem: interpolation improved convergence rate of RFM} and \ref{Theorem: optimal growing bandwidth leverage rates}. For fixed $\delta$, set $L_N:=N/\log(N/\delta)$ and $J:=\lfloor c_0L_N^{1/d}\rfloor$, where $c_0=c_0(d)>0$ is sufficiently small. The uniform reference measure is $\mr d\tau_{S_J}(w)=(2S_J)^{-d}1_{Q_{S_J}}(w)\mr dw$; the last column records only the bandwidth $S_J$.}
\label{Table: comparison of reference measures and sample size rates}
\renewcommand{\arraystretch}{1.35}
\setlength{\tabcolsep}{4pt}
\begin{tabular}{@{}
  >{\raggedright\arraybackslash}m{0.21\textwidth}
  >{\centering\arraybackslash}m{0.23\textwidth}
  >{\centering\arraybackslash}m{0.31\textwidth}
  >{\centering\arraybackslash}m{0.18\textwidth}@{}}
\toprule
Target class
& Convergence rate
& Reference measure\newline in Theorem~\ref{Theorem: interpolation improved convergence rate of RFM}
& Bandwidth\newline in Theorem~\ref{Theorem: optimal growing bandwidth leverage rates} \\
\midrule
\emph{Sobolev.}
& $L_N^{-(s-t)/d}$
& $\mr d\tau(w)\propto(1+|w|^2)^{-\bar s}\mr dw$
& $S_J=J/(4R_*)$ \\
\addlinespace
\shortstack[l]{\emph{Stretched exponential}\\$(s\ge1)$.}
& $\exp(-cL_N^{1/(sd)})$
& $\mr d\tau(w)\propto\exp(-2\bar\kappa|w|^{1/s})\mr dw$
& $S_J=J/(4R_*)$ \\
\addlinespace
\shortstack[l]{\emph{Super-exponential}\\$(s>1)$.}
& $\exp(-cL_N^{1/d}\log L_N)$
& $\mr d\tau(w)\propto\exp(-2\bar\kappa|w|^s)\mr dw$
& $\displaystyle S_J=\frac{(J\log J)^{1/s}}{4R_*}$ \\
\addlinespace
\emph{Bandlimited.}
& $\exp(-cL_N^{1/d}\log L_N)$
& $\mr d\tau_S(w)\propto 1_{Q_S}(w)\mr dw$
& $S_J=S$ \\
\bottomrule
\end{tabular}
\end{table}

The parameter $M$ in Theorem~\ref{Theorem: interpolation improved convergence rate of RFM} is an auxiliary complexity scale. For fixed $\delta$ and sufficiently large $N$, the sampling condition permits $M\asymp L_N$, where $L_N:=N/\log(N/\delta)$, which gives the sample-size rates in Table~\ref{Table: comparison of reference measures and sample size rates}. The Sobolev case is algebraic, whereas the Gevrey case has the stretched-exponential rate $\exp(-cL_N^{1/(sd)})$, with the exponential endpoint $s=d=1$. The ultra-analytic and bandlimited cases have rate $\exp(-cL_N^{1/d}\log L_N)$, which is super-exponential in the linear resolution scale $L_N^{1/d}$.

\begin{enumerate}[leftmargin=*,start=3]

\item \emph{Consequences for PDE and eigenvalue solvers.}
Combining the approximation estimates with elliptic regularity, a strong-form stability estimate, C\'ea's lemma, and compact-operator spectral approximation, we obtain error estimates for strong- and weak-form RFM discretizations of elliptic boundary value problems and for the associated eigenspaces and eigenvalues.

\item \emph{Rapid singular-value decay and ill-conditioning of random feature matrices.}
For collocation matrices generated by Fourier and $\tanh$ features, we prove, respectively, super-exponential and exponential decay of high-index singular values, together with corresponding lower bounds on their condition numbers. These results give a quantitative explanation for the severe ill-conditioning observed in high-accuracy RFM computations.
\end{enumerate}

Compared with our one-dimensional predecessor \cite{MingYu2026Spectral}, the present paper is not merely a tensor-product extension. The earlier work uses a direct one-dimensional construction to derive expectation bounds in prescribed $W^{\sigma,p}$ norms and relates one-dimensional spectral approximation to matrix ill-conditioning. Here we instead develop an operator-theoretic framework on multidimensional bounded domains, based on kernel interpolation scales, effective dimension, and empirical-operator concentration. This framework yields target-independent high-probability events and norm-independent reconstructions, covers both regularity-adapted sampling and uniform sampling on growing frequency windows, transfers the estimates to multidimensional elliptic boundary value and eigenvalue problems, and proves multidimensional singular-value decay for Fourier and $\tanh$ features. Thus the spatial setting, the probabilistic quantifiers, and the simultaneous norm control are all strengthened.

A principal distinction of the first two contributions is their twofold uniformity. For each regularity class and feature representation $\varrho$, the high-probability event depends only on the sampled features, so the single space $\operatorname{span}\{\phi_\varrho(\cdot,v_j^\varrho):1\le j\le N\}$ works for every target in the class. For each target $u$, a single coefficient vector $\alpha^\varrho(u)$ defines an approximant $u_N^\varrho$ that simultaneously attains the corresponding spectral convergence rate in every admissible error norm. Thus the random event is uniform over the target class, while the reconstruction is uniform over the error norms. To the best of our knowledge, this combination of target-uniform high-probability control, norm-independent reconstruction, and simultaneous spectral convergence throughout an admissible Sobolev/interpolation scale has not previously been established for RF approximation.

\subsection{Related work}

For learning-based PDE solvers, random-weight approaches include RFM discretizations for stationary and time-dependent problems \cite{ChenChiEYang2022RFM,ChenELuo2023TimeRFM} and dimension-robust methods for Kolmogorov equations \cite{Gonon2023BlackScholes}. Recent extensions address discontinuous interface problems \cite{SongChiYangChengChen2026Discontinuity}, structure-adaptive approximation of high-dimensional elliptic equations \cite{LinghuDongWang2026StructureAdaptive}, and nonlinear evolution equations \cite{ZhouFuWangFeng2026DiscreteTimeRFM}. Our analysis is complementary to these algorithmic developments: it isolates the approximation, PDE stability, and conditioning mechanisms needed to explain the accuracy and numerical behavior of the resulting linear least-squares discretizations.

\textbf{Theoretical foundations of RFM approximation.} RFMs belong to the broader class of randomized neural architectures, including extreme learning machines and deep randomized networks \cite{huang2006extreme,gallicchio2020deep}. Their approximation theory includes universal and activation-specific results \cite{huang2006universal,sun2019approximation}, uniform random-feature approximation \cite{rahimi2007random,Rahimi2008UniformAO}, and inverse approximation theory \cite{Weinan2020TowardsAM}. At infinite width, these models are related to Gaussian processes \cite{neal1996priors,williams1996computing}.

Within the kernel framework, Rudi and Rosasco \cite{rudi2017generalization} show that RFs can attain kernel-learning rates with substantially fewer features than data, while Bach \cite{bach2017equivalence} obtains kernel-quadrature bounds uniform over an RKHS unit ball in a prescribed $L^2$ error norm.
Leverage-score and adaptive sampling further reduce feature complexity \cite{Rudi2018LeverageSampling,Kammonen2020AdaptiveRFF,Liu2020SurrogateRFF, Chen2021FastLeverageKRR}. RF architectures have also been analyzed for high-dimensional PDEs and dynamical systems \cite{Gonon2023BlackScholes,GononGrigoryevaOrtega2023, deryck2025approximation,neufeld2025full}, and extended to operator learning between infinite-dimensional spaces \cite{nelsen2021random,nelsen2024operator}.

Spectral approximation of Gevrey targets is classical for $hp$-finite elements \cite{gui1986h1Dimension,guo1986hpversion,melenk2004hp, feischl2020exponential}.
Related neural-network results establish geometric or exponential rates for smooth, analytic, and generalized bandlimited functions \cite{mhaskar1996neural,weinan2018exponential,montanelli2019deep, DeRyck2021app_tanh}. For randomized models, spectral convergence in a fixed $L^2$-type best-approximation problem was proved by Fabiani \cite{Fabiani2025RPNN}, while our one-dimensional predecessor gives spectral RFM estimates in prescribed $W^{\sigma,p}$ norms, in expectation \cite{MingYu2026Spectral}. Fu and Wang construct a sample-dependent operator that is uniform over a Sobolev ball and simultaneous in integer $H^m$ norms, but with algebraic rates \cite{FuWang2026OptimalSobolev}. In contrast, our leverage-score analysis combines target-uniform high-probability control, one norm-independent reconstruction, and spectral convergence across the full admissible norm scale.

Simultaneous approximation of a function and its derivatives is classical \cite{MeirSharma1966,DaiXu2011} and has been used to construct common approximants in Lebesgue and Sobolev norms \cite{FeffermanHajdukRobinson2022}. Such control is natural in PDE discretization, where the same trial function enters interior, boundary, and derivative-dependent estimates. The point specific to the present RF result is the conjunction of a target-uniform random event, a single reconstruction for each target, an entire admissible norm scale, and spectral rather than algebraic convergence.

\textbf{Conditioning of RFMtxs.} Conditioning has been studied in probabilistic high-dimensional regimes \cite{ChenSchaefferWard2022,ChenSchaeffer2024}; for structured Fourier matrices, exponential ill-conditioning is known even for contiguous submatrices \cite{Barnett2022Fourier}. Numerical work has documented the resulting difficulties in high-accuracy RFM computations \cite{ChenESun2024Optimization} and motivated preconditioning strategies \cite{TanChen2026Preconditioned}. Our analysis instead derives multidimensional decay rates directly from low-dimensional approximation of bounded smooth features, thereby placing spectral convergence and severe ill-conditioning within the same approximation-theoretic mechanism.

The remainder of the paper is organized as follows. Section~\ref{Section: Approximation error by interpolation} develops the abstract interpolation-space estimate. Section~\ref{Section: Estimate of d(lambda) and the convergence rate, interpolation} specializes it to random Fourier features and relates Fourier decay to spatial regularity. Section~\ref{Section: RFM PDE solvers} transfers the estimates to elliptic boundary value and eigenvalue problems, and Section~\ref{Section: Exponential ill conditionality of random feature matrices} studies singular values and condition numbers. Section~\ref{Section: Conclusion} summarizes the main results. The appendices contain the proofs and technical estimates.

\subsection{Notation}
We denote by $\mb{N}$ the set of nonnegative integers and by $\mb{N}_{+}$ the set of positive integers. For a finite set $\mc{I}$, $|\mc{I}|$ denotes its cardinality. For a vector $\boldsymbol{\alpha}$, $|\boldsymbol{\alpha}|_{p}$ denotes its $\ell^{p}$-norm, and $|\boldsymbol{\alpha}|:=|\boldsymbol{\alpha}|_{2}$. Throughout the paper, $d$ denotes the ambient dimension. Unless stated otherwise, $\Omega\subset\mb R^d$ denotes a bounded Lipschitz domain.

Let $(\mc{X},\rho)$ be a measure space. For $1\le p\le\infty$, $L^{p}(\mc{X},\rho)$ denotes the corresponding Lebesgue space. When $\rho$ is the Lebesgue measure on $\Omega$, we write $L^{p}(\Omega)$. In the complex case, the inner product on $L^{2}(\mc{X},\rho)$ is $\langle f,g\rangle_{\rho}=\int_{\mc{X}}f(x)\overline{g(x)}\mr{d}\rho(x)$. $H^{s}(\Omega)$ and $W^{s,p}(\Omega)$ denote the standard Sobolev spaces\cite{Adams2003Sobolev}.

Hilbert-space statements are understood over $\mb{R}$ or $\mb{C}$ as appropriate. In the complex case, inner products are linear in the first argument and conjugate-linear in the second. For a Hilbert space $\mc{H}$, $\langle\cdot,\cdot\rangle_{\mc{H}}$ and $\|\cdot\|_{\mc{H}}$ denote its inner product and norm.
$\|\cdot\|$ denotes the operator norm when the domain and range are clear. Let $\mc{H}_1$, $\mc{H}_2$ be two Hilbert spaces, and $v \in \mc{H}_1, u \in \mc{H}_2$. The rank-one operator $u \otimes v : \mc{H}_1 \to \mc{H}_2$ is defined as $(u \otimes v)x = \langle x, v \rangle_{\mc{H}_1} u$ for $x \in \mc{H}_1$.
For self-adjoint operators $A$ and $B$ on the same Hilbert space, $A\preceq B$ means that $B-A$ is positive semidefinite, and $A\succeq B$ means $B\preceq A$.
For a positive self-adjoint operator $A$, fractional powers $A^{\alpha}$ are defined by spectral calculus; negative powers are understood on the positive spectral subspace. We write $\operatorname{tr}(A)$ for the trace of a trace-class operator $A$.

For a multi-index $\eta=(\eta_{1},\ldots,\eta_{d})\in\mb{N}^{d}$, set $|\eta|:=\eta_1+\cdots+\eta_d$, $\eta!:=\eta_1!\cdots\eta_d!$, $\partial^{\eta}:=\partial_{1}^{\eta_{1}}\cdots\partial_{d}^{\eta_{d}}$, and $w^{\eta}:=w_{1}^{\eta_{1}}\cdots w_{d}^{\eta_{d}}$ for $w\in \mb{R}^{d}$. The Fourier transform of $f$ in the sense of distributions is denoted by $\hat{f}$. In particular, for any $f \in L^1(\mb{R}^d)$, the Fourier transform $\hat{f}$ is defined as
\begin{equation*}
\hat{f}(\xi) = \frac{1}{(2\pi)^{d/2}} \int_{\mb{R}^d} f(x) e^{-i x \cdot \xi} \, \mr{d} x.
\end{equation*}
The notation $a \lesssim b$ means $a \le C b$ for some constant $C$ independent of $b$. An unsubscripted $\lesssim$ implies that $C$ is universal, while subscripts indicate its dependence on specific parameters. We write $a\simeq b$ if both $a\lesssim b$ and $b\lesssim a$ hold. Finally, $a\propto b$ denotes equality up to a positive normalization constant.

\section{Approximation in kernel interpolation spaces} \label{Section: Approximation error by interpolation}
This section proves an abstract approximation estimate for RFs in the interpolation scale generated by the kernel integral operator. The estimate accommodates independent choices of source regularity $s$, regression norm $\theta$, and error norm $p$ within their admissible ranges, thereby providing a unified route from kernel effective-dimension bounds to approximation rates across a broad class of target norms.

\subsection{RKHS, interpolation spaces and feature representation} \label{subsection: RKHS, interpolation spaces and feature representation}
Let $(\mc{X},\rho)$ be a measure space. A kernel $k:\mc{X}\times\mc{X}\to\mb{C}$ is called Hermitian positive definite if $k(x,y)=\overline{k(y,x)}$ and, for all $n\in\mb{N}_{+}$ and $\{x_i\}_{i=1}^{n}\subset\mc{X}$, the matrix $(k(x_i,x_j))_{1\le i,j\le n}$ is positive semidefinite. Given such a kernel $k$, there is a unique Hilbert space $\mc{H}_{k}$ such that $k(\cdot,x)\in\mc{H}_{k}$ for all $x\in\mc{X}$, and $f(x)= \la f, k(\cdot,x)\ra_{\mc{H}_{k}}$ for all $f\in\mc{H}_{k}$ and $x\in\mc{X}$. The second property is known as the reproducing property and $\mc{H}_{k}$ is referred to as the RKHS~\cite{Berlinet2004RKHS,PaulsenRaghupathi2016} associated with $k$. We abbreviate $\mc{H}_{k}$ as $\mc{H}$ when no confusion arises.

Assume that $k$ is measurable and satisfies the trace condition $\int_{\mc{X}}k(x,x)\mr{d}\rho(x)<\infty$. We also assume that the canonical embedding $I_{\rho}:\mc{H}_{k}\to L^{2}(\mc{X},\rho)$, $I_{\rho}f=[f]_{\rho}$, is injective, so that functions in $\mc{H}_{k}$ are identified unambiguously with their $L^{2}(\rho)$ equivalence classes. The associated integral operator $\Sigma:L^{2}(\mc{X},\rho)\to L^{2}(\mc{X},\rho)$ is defined by
\[
(\Sigma f)(x)=\int_{\mc{X}}k(x,y)f(y)\mr{d}\rho(y).
\]
Under these assumptions, $\Sigma$ is self-adjoint, positive, and trace class~\cite{Simon2005trace}. Let $\{(\lambda_j,e_j)\}_{j\ge1}$ be the positive eigenpairs of $\Sigma$, with the eigenvalues arranged in decreasing order and the eigenfunctions orthonormal in $L^{2}(\mc{X},\rho)$. Then, the spectral decomposition $\Sigma = \sum_{j=1}^{\infty}\lambda_{j}e_{j}\otimes e_{j}$ holds. Without loss of generality, we assume that $\{e_j\}_{j=1}^{\infty}$ forms a complete basis of $L^2(\mathcal{X}, \rho)$, otherwise we work on the closed positive spectral subspace of $\Sigma$, namely $\overline{\operatorname{span}}\{e_j:\lambda_j>0\}$.
All spectral powers below are taken on this positive spectral subspace.

Following \cite{Steinwart2012Mercer,Long2025optimal}, for $a\in\mb R$ define
\[
\begin{aligned}
\mathcal D(\Sigma^a)
&:=\left\{f=\sum_{j=1}^{\infty}f_je_j\in L^2(\mc X,\rho):
\sum_{j=1}^{\infty}\lambda_j^{2a}|f_j|^2<\infty\right\},\\
\Sigma^af&:=\sum_{j=1}^{\infty}\lambda_j^af_je_j,
\qquad f\in\mathcal D(\Sigma^a).
\end{aligned}
\]
Thus $\Sigma^a$ is bounded on $L^2$ for $a\ge0$ and is generally unbounded for $a<0$. For every $a\in\mb R$, let $\mc H^a$ be the completion of $\operatorname{span}\{e_j:j\ge1\}$ under
\begin{equation*} \label{eq: interpolation between RKHS and L2 via spectral decomposition}
\left\|\sum_{j=1}^{\infty}f_je_j\right\|_{\mc H^a}^2
:=\sum_{j=1}^{\infty}\lambda_j^{-a}|f_j|^2.
\end{equation*}
For $a\ge0$, this space is identified with $\{f\in L^2:\sum_j\lambda_j^{-a}|f_j|^2<\infty\}$, and the displayed norm equals $\|\Sigma^{-a/2}f\|_{L^2}$. For $a<0$, $\mc H^a$ is equivalently the completion of $L^2$ in this norm, or the anti-dual of $\mc H^{-a}$ with $L^2$ as pivot. By \cite[Theorem 2.11]{Steinwart2012Mercer}, $\Sigma^{\frac{1}{2}}$ is an isometry from $L^{2}(\mc{X}, \rho)$ to $\mc{H}$ and $\mc{H}^{1}=\mc{H}$. Additionally, $\mc{H}^{0}=L^{2}(\mc{X}, \rho)$ and $\mc{H}^{a}\hookrightarrow\mc{H}^{b}$ for $a\ge b$.
The following interpolation identity is used later to connect the spectral scale with Sobolev spaces.
\begin{proposition}[{\cite[Theorem 4.6]{Steinwart2012Mercer}}] \label{Proposition: interpolation between RKHS and L2}
For $a\in(0,1)$, $\mc{H}^{a}=(L^{2}(\mc{X}, \rho),\mc{H})_{a,2}$.
\end{proposition}

Let $(\mc{V},\tau)$ be a measurable parameter space equipped with a probability measure $\tau$. Consider a parametric feature function $\phi\in L^{2}(\mc{X}\times\mc{V},\rho\otimes\tau;\mb{C})$, where $\mc{X}$ and $\mc{V}$ denote the input and weight domains, respectively. We assume that $k$ admits an RF representation of the form
\begin{equation} \label{eq: kernel itself as expectation section2}
k(x,y)=\int_{\mc{V}}\phi(x,v)\overline{\phi(y,v)}\mr{d}\tau(v).
\end{equation}
The corresponding feature operator $\mc{T}:L^{2}(\mc{V},\mr{d}\tau)\to L^{2}(\mc{X},\mr{d}\rho)$ and its adjoint operator $\mc{T}^{*}$ are
\[
(\mc{T}g)(x):=\int_{\mc{V}}g(v)\phi(x,v)\mr{d}\tau(v), \quad (\mc{T}^{*}f)(v)=\int_{\mc{X}}f(x)\overline{\phi(x,v)}\mr{d}\rho(x).
\]
It follows from Fubini's theorem that $\Sigma=\mc{T}\mc{T}^{*}$, or equivalently,
\begin{equation} \label{eq: kernel integral operator as expectation section2}
\Sigma=\int_{\mc{V}}\phi(\cdot,v)\otimes\phi(\cdot,v)\mr{d}\tau(v).
\end{equation}
Furthermore, the RKHS admits the feature-space characterization~\cite[Theorem~11.3]{PaulsenRaghupathi2016}
\begin{equation}
\label{eq: RKHS feature-space characterization}
\mc{H}_{k}=\operatorname{Ran}(\mc{T}),\qquad
\|f\|_{\mc{H}_{k}}=\inf\left\{\|g\|_{L^{2}(\mc{V},\mr{d}\tau)}:\mc{T}g=f\right\}.
\end{equation}

\subsection{Ridge approximation in \texorpdfstring{$\mc{H}^{\theta}$}{H\string^theta}} \label{subsection: ridge approximation in interpolation norm}
We now formulate the ridge approximation problem in the interpolation space introduced above. Specifically, let $N\in\mb{N}_{+}$ and draw $v_1,\ldots,v_N$ independently from $q\mr{d}\tau$, where $q$ is a probability density with respect to $\tau$ satisfying the support condition specified below. Define the RF operator $\Phi: \mb{C}^{N} \to L^{2}(\mc{X},\mr{d}\rho)$ by
\begin{equation*}
\begin{aligned}
(\Phi \boldsymbol{\beta})(x) = \sum_{j=1}^{N} \beta_{j} q(v_{j})^{-\frac{1}{2}} \phi(x,v_{j}) .
\end{aligned}
\end{equation*}
Assume the target function $u \in \mc{H}^{s}$ with $s\ge0$, which is the source condition. For $\theta\le s$, we consider the minimization problem
\begin{equation} \label{ridge regression for bounding approximation error, interpolation}
\begin{aligned}
\boldsymbol{\beta}^{*} = \mathop{\arg\min}\limits_{\boldsymbol{\beta} \in \mb{C}^{N}}\left\|u-\Phi \boldsymbol{\beta}\right\|_{\mc{H}^{\theta}}^{2}+\lambda N|\boldsymbol{\beta}|^{2},
\end{aligned}
\end{equation}
where $\lambda >0$ is a regularization parameter. Then, for $p\le s$, we measure the error of the RF solution as
\begin{equation*}
\mc{E}(N,s,\theta,p) := \sup_{\|u\|_{\mc{H}^{s}}\le1} \left\|u-\Phi \boldsymbol{\beta}^{*}\right\|_{\mc{H}^{p}} .
\end{equation*}
We aim to prove upper bounds for $\mc{E}(N,s,\theta,p)$ and $|\boldsymbol{\beta}^{*}|$ in this part. As for Problem (\ref{ridge regression for bounding approximation error, interpolation}), although practically implementing $\mc{H}^{\theta}$-norm is more difficult than $L^{2}$-norm ($\theta=0$) in general cases, it demonstrates the possibility of solving problems by RFs under a norm stronger than $L^{2}$-norm. In particular, with the relation clarified in Proposition \ref{Proposition: relation between fractional Sobolev and interpolation space of RKHS on bounded domains}, it has implications for solving PDEs in Sobolev spaces.

Inspired by \cite{Long2025optimal}, we refined the concepts of maximum RF dimension and the effective dimension of the kernel integral operator. For $\theta<1$, $\lambda>0$, and $\gamma>0$, denote
\begin{equation} \label{def: r(x) used in concentration}
r(x) = \left(x^{1-\theta} + \lambda\right)^{-\frac{\gamma}{2(1-\theta)}} x^{\frac{\gamma-1}{2}} ,\quad \text{for $x\ge0$}.
\end{equation}
Set
\[
\ell_{\lambda}(v;\theta,\gamma) :=\|r(\Sigma)\phi(\cdot,v)\|_{L^{2}(\mc{X},\rho)}^{2}.
\]
Let $q$ be a probability density with respect to $\tau$ such that $q>0$ $\tau$-almost everywhere on $\{\ell_{\lambda}>0\}$. The quotient $\ell_{\lambda}/q$ is understood $\tau$-almost everywhere and is set to zero on $\{\ell_{\lambda}=q=0\}$. We define
\begin{equation*}
\begin{aligned}
& d_{\max}(q,\lambda;\theta,\gamma)
:=\operatorname*{\tau\text{-}ess\,sup}_{v\in\mc V}
\frac{\ell_{\lambda}(v;\theta,\gamma)}{q(v)}, \\
& d(\lambda;\theta,\gamma) := \operatorname{tr}\left(\Sigma r^{2}(\Sigma)\right) .
\end{aligned}
\end{equation*}
By \eqref{eq: kernel integral operator as expectation section2},
\begin{equation*}
\begin{aligned}
d(\lambda;\theta,\gamma)
&=\int_{\mc V}\ell_{\lambda}(v;\theta,\gamma)\mr d\tau(v)\\
&=\int_{\mc V}\frac{\ell_{\lambda}(v;\theta,\gamma)}{q(v)}q(v)\mr d\tau(v)
\le d_{\max}(q,\lambda;\theta,\gamma).
\end{aligned}
\end{equation*}
Equality is attained by the normalized leverage-score density defined below.
\begin{definition}[Leverage score sampling] \label{Definition: Leverage score sampling}
Assume $0<d(\lambda;\theta,\gamma)<\infty$. The normalized leverage-score density with respect to $\tau$ is
\begin{equation} \label{eq: leverage score sampling density}
q^{*}_{\lambda}(v;\theta,\gamma)=\frac{\ell_{\lambda}(v;\theta,\gamma)}{d(\lambda;\theta,\gamma)}.
\end{equation}
Sampling features from $q^{*}_{\lambda}\mr{d}\tau$ is referred to as leverage-score sampling \cite{bach2017equivalence,Rudi2018LeverageSampling,Kammonen2020AdaptiveRFF,Li2020RFMethods,Liu2020SurrogateRFF,Chen2021FastLeverageKRR,Liu2021DistributedRF}.
\end{definition}

For a fixed penalty level $\lambda>0$, we call $\lambda$ admissible for $(N,\delta,q,\theta,\gamma)$ if
\begin{equation} \label{ineq: N > d_max(lambda) ln(d(lambda)/delta), interpolation}
N\ge 3d_{\max}(q,\lambda;\theta,\gamma)\max\left\{\ln(14d(\lambda;\theta,\gamma)/\delta),1\right\}.
\end{equation}
This admissibility condition is the sampling threshold used in the empirical-operator concentration estimate below. Conversely, for $N\in\mb{N}_{+}$, the associated critical penalty level is defined by
\[
\varsigma_{N}(\delta,q,\theta,\gamma) = \inf\left\{\lambda>0: \text{\eqref{ineq: N > d_max(lambda) ln(d(lambda)/delta), interpolation} holds} \right\},
\]
analogously to the construction in \cite{Long2025optimal}. For fixed $q,\theta,\gamma$, the functions $d_{\max}(q,\lambda;\theta,\gamma)$ and $d(\lambda;\theta,\gamma)$ are nonincreasing in $\lambda$, so every penalty level strictly above $\varsigma_N(\delta,q,\theta,\gamma)$ is admissible.

\subsection{Main abstract estimate and its interpretation} \label{subsection: main abstract interpolation estimate}
The following theorem is the central estimate of this section. It controls both the approximation error and the Euclidean norm of the associated ridge-regression coefficient vector. It applies to the real or complex Hilbert-space setting described above. The high-probability event in the theorem depends only on the sampled features and is therefore uniform over the unit ball of $\mc{H}^{s}$.
\begin{theorem} \label{Theorem: abstract error estimate for regression in interpolation spaces}
Let $N\in\mb{N}_{+}$, $0\le\theta\le s\le1$, $p\le s$, $\max(\theta,p)<1$, and $0<\delta<1$. Assume $\sum_{j=1}^{\infty}\lambda_{j}^{1-\max(\theta,p)} <\infty$.
(1) If $p\le \theta$, denote $\bar{p} = \max(p,2\theta-1)$. For any $\lambda>0$, with probability at least $1-\delta$,
\begin{equation*}
\begin{aligned}
& \mc{E}(N,s,\theta,p) \le 16 \max(\lambda,\varsigma_{N})^{\frac{s-\bar{p}}{2(1-\theta)}} \left\|\Sigma\right\|^{\frac{\bar{p}-p}{2}}, \\
& |\boldsymbol{\beta}^{*}| \le 16 N^{-\frac{1}{2}} \lambda^{\frac{s-1}{2(1-\theta)}} \max(1,\varsigma_{N}/\lambda)^{\frac{s-\theta}{2(1-\theta)}},
\end{aligned}
\end{equation*}
where $\varsigma_{N} = \varsigma_{N}(\delta,q,\theta,1-\theta)$. (2) If $p\ge \theta$, for any $\lambda>0$ satisfying \eqref{ineq: N > d_max(lambda) ln(d(lambda)/delta), interpolation} with $\gamma=1-p$, with probability at least $1-\delta$,
\begin{equation*}
\begin{aligned}
& \mc{E}(N,s,\theta,p) \le 16 \lambda^{\frac{s-p}{2(1-\theta)}} , \quad \text{and }\quad |\boldsymbol{\beta}^{*}| \le 16 N^{-\frac{1}{2}} \lambda^{\frac{s-1}{2(1-\theta)}} .
\end{aligned}
\end{equation*}
In particular, this holds for every $\lambda>\varsigma_{N}(\delta,q,\theta,1-p)$.
\end{theorem}

The proof is given in Appendix~\ref{Section: proofs approximation interpolation}. It first reduces the ridge error to operator norms involving the empirical resolvent $(\widetilde{\Sigma}+\lambda I)^{-1}$, and then bounds these norms on a high-probability concentration event.

\begin{remark}
Theorem \ref{Theorem: abstract error estimate for regression in interpolation spaces} has the following interpretation. If $p\le\theta$, the regression norm is stronger than the error norm, and the same sampling event with $\varsigma_{N}(\delta,q,\theta,1-\theta)$ controls the error simultaneously for all admissible $p$ in this range. In particular, for $2\theta-1\le p\le\theta$, choosing $\lambda\lesssim\varsigma_{N}$ yields the spectral rate predicted by the source smoothness. The threshold $\bar{p}=\max(p,2\theta-1)$ records a saturation in weaker norms: when $p<2\theta-1$, lowering the error norm does not further improve the rate because the coefficient estimate is limited by the concentration of $\tilde{\Sigma}$. If $p\ge\theta$, the error norm is stronger than the regression norm, and the admissible penalty depends on $1-p$; thus larger $p$ requires a correspondingly stronger regularization condition. Finally, the coefficient bound shows that $|\boldsymbol{\beta}^{*}|$ increases as $\lambda\to0^{+}$, while $s=1$ and $\lambda\gtrsim\varsigma_{N}$ give the scale $|\boldsymbol{\beta}^{*}|\lesssim N^{-\frac{1}{2}}$.
\end{remark}

\section{Uniform approximation of regularity classes by random Fourier features} \label{Section: Estimate of d(lambda) and the convergence rate, interpolation}

\subsection{Two equivalent forms of random Fourier features}
\label{subsection: two equivalent forms of random Fourier features}
We introduce two real implementations of random Fourier features~\cite{rahimi2007random} generated by the same frequency measure $\mr{d}\tau(w)$. They are distinguished at the level of the finite-dimensional trial space, but they induce the same population kernel.

\paragraph{\textbf{Randomly shifted cosine features.}}
Consider the parameter space $\mc{V}_{\operatorname{ph}}=\mb{R}^{d}\times[0,\pi]$ equipped with the probability measure $\mr{d}\mu_{\tau}^{\operatorname{ph}}(w,b)=\pi^{-1}\mr{d}\tau(w)\mr{d}b$. For $v=(w,b)\in \mc{V}_{\operatorname{ph}}$, the feature function is defined by $\phi_{\operatorname{ph}}(x,(w,b))=\sqrt{2}\cos(w^{\top}x+b)$. The corresponding finite expansion is
\begin{equation} \label{Eq: linear combination of random Fourier features, cos(wx+b)}
u_{N}^{\operatorname{ph}}(x) = \sum_{j=1}^{N} \alpha_{j}\sqrt{2}\cos \left(w_{j}^{\top} x + b_{j}\right).
\end{equation}
According to \eqref{eq: kernel itself as expectation section2}, the corresponding kernel is
\begin{equation} \label{eq: translation invariant kernel for random Fourier feature}
\begin{aligned}
k(x,y) & = \frac{2}{\pi}\int_{\mb{R}^{d}}\int_{0}^{\pi} \cos(w^{\top}x+b)\cos(w^{\top}y+b)\mr{d}b \mr{d}\tau(w) \\
& = \frac{1}{\pi}\int_{\mb{R}^{d}}\int_{0}^{\pi} \cos(w^{\top}(x-y))\mr{d}b \mr{d}\tau(w) \\
& = \int_{\mb{R}^{d}} \cos(w^{\top}(x-y))\mr{d}\tau(w) .
\end{aligned}
\end{equation}

\paragraph{\textbf{Cosine-sine features from the complex exponential representation.}}
Alternatively, consider $\mc V_{\operatorname{cx}}=\mb R^d$ equipped with $\mr d\mu_{\tau}^{\operatorname{cx}}(w)=\mr d\tau(w)$ and use the complex Fourier feature $\phi_{\operatorname{cx}}(x,w)=e^{\mathrm{i}w^{\top}x}$. According to \eqref{eq: kernel itself as expectation section2}, if $\tau(A)=\tau(-A)$ for every Borel set $A\subset\mb{R}^{d}$, this feature induces
\[
k_{\operatorname{cx}}(x,y) =\int_{\mb{R}^{d}}e^{\mathrm{i}w^{\top}(x-y)}\mr{d}\tau(w) = \int_{\mb{R}^{d}}\cos(w^{\top}(x-y))\mr{d}\tau(w).
\]
This kernel is identical to the one in \eqref{eq: translation invariant kernel for random Fourier feature}, and therefore induces the same real RKHS and associated integral operator $\Sigma$. Writing $\gamma_j=a_j-\mathrm{i}b_j$ gives the equivalent real-valued expansion
\begin{equation} \label{Eq: linear combination of random Fourier features, cosine-sine}
u_N^{\operatorname{cs}}(x)
:=\operatorname{Re}\sum_{j=1}^{N}\gamma_j e^{\mathrm{i}w_j^{\top}x}
=\sum_{j=1}^{N}\left[a_j\cos(w_j^{\top}x)+b_j\sin(w_j^{\top}x)\right].
\end{equation}
Thus each sampled frequency contributes the two real components $\psi_{\operatorname{cs}}(x,w)=\left(\cos(w^{\top}x),\sin(w^{\top}x)\right)$. We use $\varrho\in\{\operatorname{ph},\operatorname{cx}\}$ to index the two representations and write $(\mc V_{\varrho},\mu_{\tau}^{\varrho},\phi_{\varrho})$ for the corresponding parameter space, sampling measure, and feature map.

To apply the spectral framework above on the whole space $L^2(\Omega)$, we need the associated kernel integral operators to be injective. All frequency measures considered below have symmetrized support with nonempty interior. The following lemma therefore implies that every such operator is injective and that its positive spectral subspace $(\ker\Sigma)^\perp$ coincides with $L^2(\Omega)$.

\begin{lemma}[Injectivity of translation-invariant kernel integral operators] \label{Lemma: injectivity of translation-invariant kernel integral operators}
Let $\Omega\subset\mb R^d$ be a bounded domain, let $\tau$ be a finite nonnegative Borel measure on $\mb R^d$, and let $\Sigma$ be the integral operator on $L^2(\Omega)$ induced by the kernel \eqref{eq: translation invariant kernel for random Fourier feature}. Denote the symmetrization of $\tau$ by $\tau_s(A):=[\tau(A)+\tau(-A)]/2$ for Borel sets $A\subset\mb R^d$. If $\operatorname{supp}\tau_s$ has nonempty interior, then $\ker\Sigma=\{0\}$; hence every eigenvalue of $\Sigma$ is strictly positive.
\end{lemma}

\begin{remark}
The support assumption on $\tau_s$ is essential. Without it, injectivity of $\Sigma$ may fail. For example, if $\tau$ is a finite discrete measure, then the associated Fourier kernel has finite rank, and hence the corresponding integral operator has a nontrivial null space on the infinite dimensional space $L^{2}(\Omega)$.
\end{remark}
The proof is given in Appendix~\ref{Section: proofs uniform approximation random Fourier features}.

\subsection{Approximation with regularity-adapted reference measures}

We now derive convergence rates uniformly over several target classes. For $\kappa,a>0$, define
\[
\mc F_{\kappa,a}(\Omega):=\{u:\|u\|_{\kappa,a}<\infty\}, \qquad \|u\|_{\kappa,a}:= \inf_{U|_{\Omega}=u}\left\|e^{\kappa|\cdot|^{a}}\widehat
U\right\|_{L^{2}(\mb{R}^{d})},
\]
where the infimum is taken over all $U\in L^{2}(\mb{R}^{d})$ whose restriction to $\Omega$ is $u$ and whose weighted Fourier norm is finite. For $S>0$, let $\mc B_S(\Omega)$ be the space of restrictions to $\Omega$ of functions in $L^2(\mb R^d)$ bandlimited to $[-S,S]^d$, equipped with the norm
\[
\|u\|_{\mc B_{S}(\Omega)} :=\inf\{\|U\|_{L^{2}(\mb{R}^{d})}:U|_{\Omega}=u,\ \operatorname{supp}\widehat U\subset[-S,S]^{d}\}.
\]
For any normed space $X$, denote its closed unit ball by $\mb B_X:=\{u\in X:\|u\|_X\le1\}$.

Theorem~\ref{Theorem: interpolation improved convergence rate of RFM} is the first result in this paper that establishes spectral convergence for random Fourier feature approximation. It combines regularity-adapted reference measures with optimal leverage-score sampling and controls both the approximation error and the coefficient norm. The theorem applies separately to the two feature representations introduced in Subsection~\ref{subsection: two equivalent forms of random Fourier features}. We use $\varrho\in\{\operatorname{ph},\operatorname{cx}\}$ to denote the phase and complex representations, respectively, and write $u_N^{\varrho}$ and $\alpha^\varrho(u)$ for the corresponding approximant and coefficient vector, where $N$ is the number of sampled frequencies. For real-valued targets, the complex representation is realized by the cosine--sine expansion \eqref{Eq: linear combination of random Fourier features, cosine-sine}.

For $\rho>0$, write $Q_\rho:=(-\rho,\rho)^d$. Suppose that $\Omega$ is contained, after a translation, in $Q_R$, and let $\widetilde\Omega$ be a translated copy of $Q_R$ containing $\Omega$. For each case in Theorem~\ref{Theorem: interpolation improved convergence rate of RFM}, choose $\mr d\tau$ and $\lambda$ as specified there. For each representation $\varrho$, let $q^*_{\lambda,\varrho}$ be the leverage-score density with respect to $\mu_\tau^\varrho$ for the kernel integral operator on $\Omega$ in case (1) and on $\widetilde\Omega$ in cases (2)--(4). Draw $\{v_j^\varrho\}_{j=1}^N$ independently from $q^*_{\lambda,\varrho}\mr d\mu_\tau^\varrho$. For a density $q$ on $\mc V_{\varrho}$ and points $v_j\in\mc V_{\varrho}$, define $|\alpha|_{\ell^{2}(q)}:= (\sum_{j=1}^{N}q(v_j)|\alpha_j|^2)^{1/2}$ for $\alpha\in\mb C^N$.

\begin{theorem}[Regularity-adapted reference measures] \label{Theorem: interpolation improved convergence rate of RFM}
Let $\Omega$, $R$, and $\widetilde\Omega$ be as above. Let $\nu\ge0$ and $0<\delta<1$. Let $M>0$ and $N\in\mb N_{+}$ satisfy $M\ge e\delta/14$ and $N\ge 3M\ln(14M/\delta)$. In (2)--(4), assume additionally that $M$ is sufficiently large in terms of the fixed parameters. In (2) and (3), let $0<\kappa\le\bar\kappa$.

For each case and representation $\varrho$, with probability at least $1-\delta$, every admissible target $u$ admits coefficients $\alpha^{\varrho}(u)$ such that $u_N^{\varrho}$ satisfies the stated bounds for all admissible error indices. The event depends only on the corresponding sampled features.

\emph{(1) Sobolev ball.} Let $\nu+d/2<s\le\bar{s}$ and $\mr{d}\tau\simeq (1+|w|^{2})^{-\bar{s}}\mr{d}w$. There is a constant $c_1=c_1(d,\bar s,\nu,\Omega)>0$.
Set $\lambda=c_{1}M^{-2(\bar{s}-\nu)/d}$. Then, for all $u\in\mb B_{H^s(\Omega)}$ and $\max\{0,2\nu-\bar{s}\}\le t\le \nu$,
\[
\begin{aligned}
&\|u-u_N^{\varrho}\|_{H^t(\Omega)}
\lesssim_{d,\bar{s},s,\nu,t,\Omega}M^{-(s-t)/d},\\
&|\alpha^{\varrho}(u)|_{\ell^{2}(q^*_{\lambda,\varrho})}
\lesssim_{d,\bar{s},s,\nu,\Omega}
N^{-1/2}M^{(\bar{s}-s)/d}.
\end{aligned}
\]

\emph{(2) Stretched-exponential Fourier ball.} Let $s\ge1$ and $\mr{d}\tau\propto e^{-2\bar\kappa|w|^{1/s}}\mr{d}w$. There exist constants $c_{\lambda},a_{\lambda},a_{\mathrm e}>0$ and $a_{\mathrm c}\ge0$, depending only on $\kappa,\bar\kappa,d,s,R$. Set $\lambda=c_{\lambda}\exp(-a_{\lambda}M^{1/(sd)})$. Then, for all $u\in\mb B_{\mc F_{\kappa,1/s}(\Omega)}$, $t\ge0$, and $1\le p\le\infty$,
\[
\begin{aligned}
&\|u-u_N^{\varrho}\|_{W^{t,p}(\Omega)}
\lesssim_{\kappa,\bar\kappa,d,s,t,p,R}
\exp(-a_{\mathrm e}M^{1/(sd)}),\\
&|\alpha^{\varrho}(u)|_{\ell^{2}(q^*_{\lambda,\varrho})}
\lesssim_{\kappa,\bar\kappa,d,s,R}
N^{-1/2}\exp(a_{\mathrm c}M^{1/(sd)}).
\end{aligned}
\]

\emph{(3) Super-exponential Fourier ball.} Let $s>1$ and $\mr{d}\tau\propto e^{-2\bar\kappa|w|^{s}}\mr{d}w$. There exist constants $c_{\lambda},a_{\lambda},a_{\mathrm e}>0$ and $a_{\mathrm c}\ge0$, depending only on $\kappa,\bar\kappa,d,s,R$. Set $\lambda=c_{\lambda}\exp(-a_{\lambda}M^{1/d}\ln M)$. Then, for all $u\in\mb B_{\mc F_{\kappa,s}(\Omega)}$, $t\ge0$, and $1\le p\le\infty$,
\[
\begin{aligned}
&\|u-u_N^{\varrho}\|_{W^{t,p}(\Omega)}
\lesssim_{\kappa,\bar\kappa,d,s,t,p,R}
\exp(-a_{\mathrm e}M^{1/d}\ln M),\\
&|\alpha^{\varrho}(u)|_{\ell^{2}(q^*_{\lambda,\varrho})}
\lesssim_{\kappa,\bar\kappa,d,s,R}
N^{-1/2}\exp(a_{\mathrm c}M^{1/d}\ln M).
\end{aligned}
\]

\emph{(4) Bandlimited ball.} Let $S>0$ and $\mr{d}\tau=(2S)^{-d}1_{(-S,S)^{d}}(w)\mr{d}w$. There exist positive constants $c_{\lambda},a_{\lambda}$, and $a_{\mathrm e}$, depending only on $d,S,R$. Set $\lambda=c_{\lambda}\exp(-a_{\lambda}M^{1/d}\ln M)$. Then, for all $u\in\mb B_{\mc B_{S}(\Omega)}$, $t\ge0$, and $1\le p\le\infty$,
\[
\begin{aligned}
&\|u-u_N^{\varrho}\|_{W^{t,p}(\Omega)}
\lesssim_{d,t,p,S,R}\exp(-a_{\mathrm e}M^{1/d}\ln M),\\
&|\alpha^{\varrho}(u)|_{\ell^{2}(q^*_{\lambda,\varrho})}
\lesssim_{d,S,R}N^{-1/2}.
\end{aligned}
\]
All constants are independent of $M,N,\delta$, and $u$; they may depend on the fixed domain and the displayed regularity and kernel parameters.
\end{theorem}

One admissible explicit choice of the rate constants in cases (2)--(4) is listed in Table~\ref{Table: explicit rate constants in regularity adapted theorem}.
Put $\sigma:=\kappa/\bar\kappa$ and $\xi_s:=1-1/s$. In case (2), let $b_2$ denote the right-hand side of \eqref{eq: def of kappa_s in subexponential and exponential convergence rate, interpolation} with $\theta=\sigma/2$ and the decay parameter there replaced by $\bar\kappa$. For case (3), set
\[
b_3:=\frac{2^{-1/d}}
{e\bigl(\xi_s^{-1}(1-\sigma/2)^{-1}+3/2\bigr)}.
\]
\begin{table}[!htbp]
\centering
\small
\caption{Explicit admissible rate constants in Theorem~\ref{Theorem: interpolation improved convergence rate of RFM}.}
\label{Table: explicit rate constants in regularity adapted theorem}
\renewcommand{\arraystretch}{1.35}
\begin{tabular}{cccc}
\toprule
Case & $a_\lambda$ & $a_{\mathrm e}$ & $a_{\mathrm c}$ \\
\midrule
(2) & $b_2/2$
& $\dfrac{3\sigma}{4(2-\sigma)}a_\lambda$
& $\dfrac{1-\sigma}{2-\sigma}a_\lambda$ \\
(3) & $\dfrac12(1-\sigma/2)\xi_s b_3$
& $\dfrac{3\sigma}{4(2-\sigma)}a_\lambda$
& $\dfrac{1-\sigma}{2-\sigma}a_\lambda$ \\
(4) & $\dfrac{2^{1-1/d}}{28e}$
& $3a_\lambda/4$
& $0$ \\
\bottomrule
\end{tabular}
\end{table}

In case (1), $s$ specifies the Sobolev regularity of the target class, $\bar{s}$ determines the polynomial decay of the reference frequency measure and hence the kernel smoothness scale, and $\nu$ is the highest Sobolev order controlled by the approximation estimate. In cases (2) and (3), $\kappa$ and $\bar\kappa$ play the corresponding target and reference roles: $\kappa$ quantifies the Fourier decay of the target class, whereas $\bar\kappa$ determines the reference frequency distribution and hence the kernel interpolation scale.

The parameter $M$ is an auxiliary complexity scale. For fixed $\delta$ and sufficiently large $N$, the sampling condition permits $M\asymp L_N$, where $L_N:=N/\log(N/\delta)$, which gives the sample-size rates reported in Table~\ref{Table: comparison of reference measures and sample size rates}. Case (1) gives the algebraic rate $L_N^{-(s-t)/d}$, whereas case (2) gives the stretched-exponential rate $\exp(-cL_N^{1/(sd)})$, with the exponential endpoint $s=d=1$. Cases (3) and (4) yield $\exp(-cL_N^{1/d}\log L_N)$, which is super-exponential in the linear resolution scale $L_N^{1/d}$.

A key strength of Theorem~\ref{Theorem: interpolation improved convergence rate of RFM} is its twofold uniformity. With probability at least $1-\delta$, the single sampled space $\operatorname{span}\{\phi_{\varrho}(\cdot,v_j^{\varrho}):1\le j\le N\}$ contains an approximant attaining the asserted convergence rate for every target $u$ in the prescribed class. For each such $u$, one can choose a single coefficient vector $\alpha^{\varrho}(u)$ so that the corresponding approximant $u_N^{\varrho}$ satisfies all admissible error estimates simultaneously, indexed by $t$ in case (1) and by $(t,p)$ in cases (2)--(4). Thus the high-probability event is uniform over the target class, while the selected approximant is uniform over the error norms. Both uniformities follow from the abstract error estimate in Theorem~\ref{Theorem: abstract error estimate for regression in interpolation spaces}: the high-probability event is independent of the target, and the resulting error bounds hold simultaneously throughout the admissible interpolation scale.

The proof is given in Appendix~\ref{Section: proof uniform approximation regularity balls}.

\subsection{Approximation with growing-bandwidth uniform reference measures}

This subsection establishes spectral convergence for a simple, broadly applicable reference-measure design: a uniform distribution on a frequency cube whose support expands at an explicit rate. The bandwidth growth law adapts this readily implementable construction to different target regularities, and the features are sampled from the associated leverage-score distribution. All results below apply separately to each fixed representation $\varrho\in\{\operatorname{ph},\operatorname{cx}\}$. Since $\tau_S$ is symmetric, the two representations induce the same kernel integral operator and hence the same effective dimension, while their leverage-score densities and sampled features are representation dependent. We therefore use $q^*_{\lambda,\varrho}$ and $\mu_{\tau_S}^{\varrho}$ for the corresponding density and parameter measure. For real-valued targets in the complex representation, taking the real part gives the cosine--sine realization \eqref{Eq: linear combination of random Fourier features, cosine-sine}; the error bounds are preserved by contractivity, and $|\gamma_j|^2=a_j^2+b_j^2$ preserves the weighted coefficient norm.

Figure~\ref{Figure: fixed and growing reference frequency measures} summarizes the fixed and growing reference measures used below.
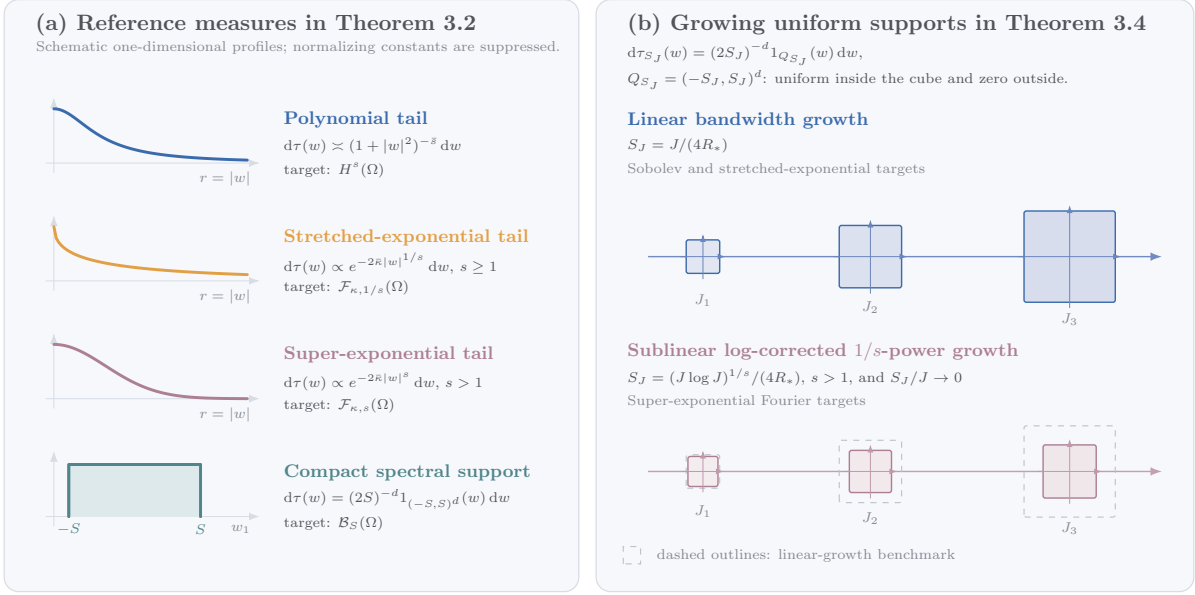
\begin{figure}[!htbp]
\centering
\begingroup
\definecolor{ink}{HTML}{1F2937}
\definecolor{muted}{HTML}{667085}
\definecolor{grid}{HTML}{D7DEE8}
\definecolor{panel}{HTML}{F8FAFC}
\definecolor{sobolev}{HTML}{2563EB}
\definecolor{stretch}{HTML}{E58A18}
\definecolor{superexp}{HTML}{A44A9D}
\definecolor{band}{HTML}{198C67}
\tikzset{ axis/.style={draw=grid, line width=.45pt, -{Latex[length=1.8mm]}}, profile/.style={line width=1.45pt, line cap=round, line join=round}, panelbox/.style={draw=grid, fill=panel, rounded corners=2.5mm, line width=.6pt}, smalllabel/.style={font=\scriptsize, text=muted}, lawtitle/.style={font=\small\bfseries, text=ink}, formula/.style={font=\scriptsize, text=ink}, }
\newcommand{\freqsquare}[5]{\draw[draw=#4, fill=#4!#5, line width=.7pt, rounded corners=.4mm]
(#1-#3/2,#2-#3/2) rectangle (#1+#3/2,#2+#3/2); \draw[draw=#4!75, line width=.45pt, -{Latex[length=1.35mm]}] (#1-#3/2-.10,#2) -- (#1+#3/2+.10,#2); \draw[draw=#4!75, line width=.45pt, -{Latex[length=1.35mm]}] (#1,#2-#3/2-.10) -- (#1,#2+#3/2+.10); }
\newcommand{\linearbenchmark}[3]{\draw[draw=muted!48, dashed, line width=.55pt, rounded corners=.4mm]
(#1-#3/2,#2-#3/2) rectangle (#1+#3/2,#2+#3/2); }
\resizebox{\textwidth}{!}{\begin{tikzpicture}[x=1cm,y=1cm,font=\small,text=ink]
\path[use as bounding box] (-5pt,-5pt) rectangle ([xshift=5pt,yshift=5pt]20.6,10.25);

\path[panelbox] (0,0) rectangle (9.95,10.25);
\node[anchor=west,font=\large\bfseries] at (.42,9.82)
  {(a) Reference measures in Theorem 3.2};
\node[anchor=west,smalllabel] at (.43,9.42)
  {Schematic one-dimensional profiles; normalizing constants are suppressed.};

\draw[axis] (.72,7.42) -- (4.42,7.42) node[below left=1pt,smalllabel] {$r=|w|$};
\draw[axis] (.86,7.25) -- (.86,8.55);
\draw[profile,sobolev,domain=0:3.35,samples=80]
  plot ({.86+\x},{7.42+.94/(1+1.55*\x*\x)});
\node[anchor=west,lawtitle,text=sobolev] at (4.72,8.18) {Polynomial tail};
\node[anchor=west,formula] at (4.72,7.72)
  {$\mathrm d\tau(w)\asymp(1+|w|^2)^{-\bar s}\,\mathrm dw$};
\node[anchor=west,formula] at (4.72,7.30)
  {target: $H^s(\Omega)$};

\draw[axis] (.72,5.38) -- (4.42,5.38) node[below left=1pt,smalllabel] {$r=|w|$};
\draw[axis] (.86,5.21) -- (.86,6.51);
\draw[profile,stretch,domain=0:3.35,samples=80]
  plot ({.86+\x},{5.38+.94*exp(-1.18*sqrt(\x))});
\node[anchor=west,lawtitle,text=stretch] at (4.72,6.14) {Stretched-exponential tail};
\node[anchor=west,formula] at (4.72,5.68)
  {$\mathrm d\tau(w)\propto e^{-2\bar\kappa|w|^{1/s}}\,\mathrm dw$, $s\ge1$};
\node[anchor=west,formula] at (4.72,5.26)
  {target: $\mathcal F_{\kappa,1/s}(\Omega)$};

\draw[axis] (.72,3.34) -- (4.42,3.34) node[below left=1pt,smalllabel] {$r=|w|$};
\draw[axis] (.86,3.17) -- (.86,4.47);
\draw[profile,superexp,domain=0:3.35,samples=80]
  plot ({.86+\x},{3.34+.94*exp(-.63*\x*\x)});
\node[anchor=west,lawtitle,text=superexp] at (4.72,4.10) {Super-exponential tail};
\node[anchor=west,formula] at (4.72,3.64)
  {$\mathrm d\tau(w)\propto e^{-2\bar\kappa|w|^s}\,\mathrm dw$, $s>1$};
\node[anchor=west,formula] at (4.72,3.22)
  {target: $\mathcal F_{\kappa,s}(\Omega)$};

\draw[axis] (.72,1.30) -- (4.42,1.30) node[below left=1pt,smalllabel] {$w_1$};
\draw[axis] (.86,1.13) -- (.86,2.43);
\draw[profile,band] (1.12,1.30) -- (1.12,2.20) -- (3.40,2.20) -- (3.40,1.30);
\fill[band!14] (1.12,1.30) rectangle (3.40,2.20);
\draw[profile,band] (1.12,1.30) -- (1.12,2.20) -- (3.40,2.20) -- (3.40,1.30);
\node[smalllabel,text=band] at (1.12,1.08) {$-S$};
\node[smalllabel,text=band] at (3.40,1.08) {$S$};
\node[anchor=west,lawtitle,text=band] at (4.72,2.06) {Compact spectral support};
\node[anchor=west,formula] at (4.72,1.60)
  {$\mathrm d\tau(w)=(2S)^{-d}1_{(-S,S)^d}(w)\,\mathrm dw$};
\node[anchor=west,formula] at (4.72,1.18)
  {target: $\mathcal B_S(\Omega)$};

\begin{scope}[xshift=-.4cm]
\path[panelbox] (10.65,0) rectangle (21.00,10.25);
\node[anchor=west,font=\large\bfseries] at (11.07,9.82)
  {(b) Growing uniform supports in Theorem 3.4};
\node[anchor=west,formula] at (11.07,9.31)
  {$\displaystyle \mathrm d\tau_{S_J}(w)=(2S_J)^{-d}1_{Q_{S_J}}(w)\,\mathrm dw$,};
\node[anchor=west,formula] at (11.07,8.87)
  {$Q_{S_J}=(-S_J,S_J)^d$: uniform inside the cube and zero outside.};

\node[anchor=west,lawtitle,text=sobolev] at (11.07,8.16)
  {Linear bandwidth growth};
\node[anchor=west,formula] at (11.07,7.72)
  {$S_J=J/(4R_*)$};
\node[anchor=west,smalllabel] at (11.07,7.31)
  {Sobolev and stretched-exponential targets};
\draw[-{Latex[length=2mm]},draw=sobolev!70,line width=.7pt]
  (11.55,5.80) -- (20.45,5.80);
\freqsquare{12.50}{5.80}{.58}{sobolev}{12}
\freqsquare{15.40}{5.80}{1.08}{sobolev}{14}
\freqsquare{18.85}{5.80}{1.58}{sobolev}{16}
\node[smalllabel] at (12.50,5.05) {$J_1$};
\node[smalllabel] at (15.40,4.93) {$J_2$};
\node[smalllabel] at (18.85,4.72) {$J_3$};

\node[anchor=west,lawtitle,text=superexp] at (11.07,4.15)
  {Sublinear log-corrected $1/s$-power growth};
\node[anchor=west,formula] at (11.07,3.71)
  {$S_J=(J\log J)^{1/s}/(4R_*)$, $s>1$, and $S_J/J\to0$};
\node[anchor=west,smalllabel] at (11.07,3.30)
  {Super-exponential Fourier targets};
\draw[-{Latex[length=2mm]},draw=superexp!70,line width=.7pt]
  (11.55,2.08) -- (20.45,2.08);
\linearbenchmark{12.50}{2.08}{.58}
\linearbenchmark{15.40}{2.08}{1.08}
\linearbenchmark{18.85}{2.08}{1.58}
\freqsquare{12.50}{2.08}{.52}{superexp}{11}
\freqsquare{15.40}{2.08}{.73}{superexp}{13}
\freqsquare{18.85}{2.08}{.92}{superexp}{15}
\node[smalllabel] at (12.50,1.40) {$J_1$};
\node[smalllabel] at (15.40,1.27) {$J_2$};
\node[smalllabel] at (18.85,1.10) {$J_3$};
\draw[draw=muted!48, dashed, line width=.55pt]
  (11.12,.48) rectangle (11.42,.78);
\node[anchor=west,smalllabel] at (11.58,.63)
  {dashed outlines: linear-growth benchmark};
\end{scope}

\end{tikzpicture}}
\endgroup \caption{Reference frequency measures underlying Theorems \ref{Theorem: interpolation improved convergence rate of RFM} and \ref{Theorem: optimal growing bandwidth leverage rates}. Panel (a) compares schematic one-dimensional profiles of the four reference measures. Panel (b) illustrates the two bandwidth growth laws for the cube-supported uniform reference measure.}
\label{Figure: fixed and growing reference frequency measures}
\end{figure}

\begin{proposition}[Leverage approximation of a growing-bandwidth target]
\label{Proposition: leverage approximation growing bandwidth target}
Suppose that $\Omega$ is contained, after a translation, in $Q_R$. Let $0<a,\delta<1$ and $N\in\mb N_+$. Suppose that $U\in L^2(\mb R^d)$ and $\operatorname{supp}\widehat U\subset\overline{Q_S}$. Let $J\ge\max\{2SR,2\}$ be an integer. For the constant $C_{a,d}\ge1$ supplied by Lemma~ \ref{Lemma: growing bandwidth PSWF effective dimension}, set
\begin{equation}
\label{eq: lambda growing bandwidth PSWF}
\lambda_{J,S}
=C_{a,d}\left(\frac{\pi}{S}\right)^{da}
J^{-1}(1+SR)^a\left(\frac {SR}{J}\right)^{2aJ},
\end{equation}
and assume
\begin{equation}
\label{eq: sample condition growing bandwidth leverage}
N\ge6J^d\log(28J^d/\delta).
\end{equation}
Fix $\varrho\in\{\operatorname{ph},\operatorname{cx}\}$ and draw $\{v_j^\varrho\}_{j=1}^N$ independently from $q^*_{\lambda_{J,S},\varrho}(\,\cdot\,;1-a,a) \mr{d}\mu_{\tau_S}^{\varrho}$. Then, with probability at least $1-\delta$, Problem~ \ref{ridge regression for bounding approximation error, interpolation} has an RF solution $u_{N,S}^{\varrho}$ satisfying, for all $t\ge0$ and $1\le p\le\infty$,
\begin{equation}
\label{eq: finite bandwidth leverage error explicit}
\begin{aligned}
\|U-u_{N,S}^{\varrho}\|_{W^{t,p}(\Omega)}
&\le
C_{d,t,p,a,R}(1+dS^2)^{(2t+d+1)/4}
J^{-1/2}(1+SR)^{a/2}
(SR/J)^{aJ}
\|U\|_{L^2(\mb R^d)}.
\end{aligned}
\end{equation}
The coefficients satisfy
\begin{equation}
\label{eq: growing bandwidth leverage coefficient bound}
|\alpha^{\varrho}|_{\ell^2(q^*_{\lambda_{J,S},\varrho}
(\,\cdot\,;1-a,a))}
\le16N^{-1/2}(S/\pi)^{d/2}
\|U\|_{L^2(\mb R^d)}.
\end{equation}
\end{proposition}

The proof is given in Appendix~\ref{Subsection: proof growing bandwidth leverage approximation}.

\begin{theorem}[Growing uniform reference measures]
\label{Theorem: optimal growing bandwidth leverage rates}
Let $\Omega$ and $R$ be as in Proposition \ref{Proposition: leverage approximation growing bandwidth target}, let $0<\delta<1$ and $N\in\mb N_+$, and put $R_*=\max\{R,1\}$. Let $J$ be a sufficiently large integer satisfying \eqref{eq: sample condition growing bandwidth leverage}. For each case and each fixed $\varrho\in\{\operatorname{ph},\operatorname{cx}\}$, sample as in Proposition~\ref{Proposition: leverage approximation growing bandwidth target} with $a=1/2$ and $S=S_J$. Then, with probability at least $1-\delta$, every admissible target admits an approximation $u_{N,J}^{\varrho}$ satisfying the corresponding bounds simultaneously for all admissible $t$ and $p$.
The constants are independent of $J,N,\delta$, and the target.

\emph{(1) Sobolev ball.} Let $s>0$, $1\le p\le\infty$, and $u\in\mb B_{W^{s,p}(\Omega)}$. Take $S_J=J/(4R_*)$. Then, for all $0\le t\le s$,
\begin{equation}
\label{eq: optimal Sobolev growing leverage rate}
\|u-u_{N,J}^{\varrho}\|_{W^{t,p}(\Omega)}
\le CJ^{-(s-t)}.
\end{equation}

\emph{(2) Stretched-exponential Fourier ball.} Let $s\ge1$, $\kappa_0>0$, and $u\in\mb B_{\mc F_{\kappa_0,1/s}(\Omega)}$. Take $S_J=J/(4R_*)$. Then, for all $t\ge0$ and $1\le p\le\infty$,
\begin{equation}
\label{eq: optimal stretched exponential growing leverage rate}
\|u-u_{N,J}^{\varrho}\|_{W^{t,p}(\Omega)}
\le C\exp\left(-cJ^{1/s}\right).
\end{equation}
The endpoint $s=1$ is the analytic case.

\emph{(3) Super-exponential Fourier ball.} Let $s>1$ and $\kappa_0>0$. Assume $u\in\mb B_{\mc F_{\kappa_0,s}(\Omega)}$, and take $S_J=(J\log J)^{1/s}/(4R_*)$.
Then, for all $t\ge0$ and $1\le p\le\infty$,
\begin{equation}
\label{eq: optimal superexponential growing leverage rate}
\|u-u_{N,J}^{\varrho}\|_{W^{t,p}(\Omega)}
\le C\exp\left(-cJ\log J\right).
\end{equation}

The fixed-bandwidth bandlimited case is already covered by Theorem \ref{Theorem: interpolation improved convergence rate of RFM}(4).

For all sufficiently large $N$, choose $J=\left\lfloor c_0\bigl(N/\log(N/\delta)\bigr)^{1/d}\right\rfloor$, where $c_0=c_0(d)>0$ is sufficiently small, and put $L_{N,\delta}:=N/\log(N/\delta)$. Then \eqref{eq: sample condition growing bandwidth leverage} holds, and the three rates are, respectively,
\begin{equation}
\label{eq: optimal rates in sample number N}
L_{N,\delta}^{-(s-t)/d},\qquad
\exp(-cL_{N,\delta}^{1/(sd)}),\qquad
\exp(-cL_{N,\delta}^{1/d}\log L_{N,\delta}).
\end{equation}
Thus \eqref{eq: optimal rates in sample number N} agrees with the sample-size orders in Theorem \ref{Theorem: interpolation improved convergence rate of RFM}; the conversion uses $J^d\asymp N/\log(N/\delta)$.
\end{theorem}

The proof is given in Appendix~\ref{Subsection: proof growing uniform reference measures}.

Besides matching the regularity-dependent rates of Theorem~\ref{Theorem: interpolation improved convergence rate of RFM}, Theorem~\ref{Theorem: optimal growing bandwidth leverage rates} preserves its twofold uniformity: the sampling event is target-uniform, while for each target one reconstruction attains the estimates simultaneously in all admissible $W^{t,p}$ norms. It also replaces the regularity-adapted reference measures by a particularly simple one: the uniform measure on $Q_{S_J}$, with the features drawn from its associated leverage-score distribution. For both Sobolev and stretched-exponential Fourier classes, the regularity-independent bandwidth $S_J=J/(4R_*)$, where $J\asymp L_{N,\delta}^{1/d}$, yields the corresponding algebraic and spectral rates without prior knowledge of the target regularity; equivalently, the frequency-cube side length grows as $N^{1/d}$ up to logarithmic factors. For the higher, super-exponential Fourier regularity, the faster rate is obtained by the slower bandwidth growth $S_J=(J\log J)^{1/s}/(4R_*)=o(J)$, which concentrates the reference measure on a narrower frequency window.

\subsection{Spatial interpretation of the regularity conditions}
\label{subsection: spatial interpretation of the regularity conditions}

Theorems~\ref{Theorem: interpolation improved convergence rate of RFM} and \ref{Theorem: optimal growing bandwidth leverage rates} are formulated in Sobolev spaces and in the weighted Fourier classes $\mc F_{\kappa,a}(\Omega)$. The latter classes make the approximation rates transparent, whereas the following conditions relate them to standard smoothness assumptions in the physical variable $x$.

For completeness, we recall the Gevrey classes used below~\cite{rodino1993linear}. For $s\ge1$, a function $f\in C^{\infty}(\Omega)$ belongs to $G^{s}(\Omega)$ if, for every compact subset $K\subset\Omega$, there exist positive constants $M,C$, independent of $\beta$ and $x\in K$, such that
\[
|\partial^{\beta}f(x)|\le M C^{|\beta|_{1}}(\beta!)^{s} \qquad (\beta\in\mb N_{0}^{d},\ x\in K).
\]
For $s>1$, let $G_{0}^{s}(\Omega)$ denote the subspace of functions in $G^{s}(\Omega)$ with compact support in $\Omega$. The classes are nested: $G^{s}(\Omega)\subset G^{t}(\Omega)$ whenever $1\le s\le t$, and $G^{1}(\Omega)$ is the space of analytic functions on $\Omega$. Moreover, both inclusions
\[
G^{1}(\Omega)\subset\bigcap_{s>1}G^{s}(\Omega), \qquad \bigcup_{s\ge1}G^{s}(\Omega)\subset C^{\infty}(\Omega)
\]
are strict.

\begin{assumption}[Spatial regularity conditions]
\label{Assumption: regularities}
\leavevmode
\begin{enumerate}
\item[(a)] \emph{Sobolev regularity.}
$u\in H^s(\Omega)$ with $s>d/2+\nu$.
\item[(b)] \emph{Gevrey regularity.}
For some $s>1$, there exists a bounded open set $\Omega'$ with $\overline\Omega\Subset\Omega'$ such that $u$ extends to a function in $G^s(\Omega')$.
\item[(c)] \emph{Analytic regularity.}
For some $\rho>0$, $u$ admits an analytic continuation to $\{\zeta\in\mb C^d:|\operatorname{Im}\zeta|<\rho\}$. Moreover, $u(\cdot+\mathrm i y)\in L^2(\mb R^d)$ for $|y|<\rho$, and
\[
\sup_{|y|\le\kappa} \|u(\cdot+\mathrm i y)\|_{L^2(\mb R^d)}<\infty \qquad\text{for every }0<\kappa<\rho.
\]
\item[(d)] \emph{Super-exponential Fourier decay.}
For some $s>1$ and $\kappa>0$, $u$ has an extension $U$ to $\mb R^d$ satisfying $e^{\kappa|\cdot|^s}\widehat U\in L^2(\mb R^d)$.
\item[(e)] \emph{Bandlimited regularity.}
For some $S>0$, $u$ has an extension $U\in L^2(\mb R^d)$ satisfying $\operatorname{supp}\widehat U\subset[-S,S]^d$.
\end{enumerate}
\end{assumption}

\begin{lemma}
\label{Lemma: Fourier decay from Gevrey and analytic regularity}
Under Assumption~\ref{Assumption: regularities}, the following statements hold.
\begin{enumerate}
\item[(i)] If (b) holds, then there exists $\kappa>0$ such that $\|u\|_{\kappa,1/s}<\infty$.
\item[(ii)] If (c) holds, then $\|u\|_{\kappa,1}<\infty$ for every $0<\kappa<\rho$.
\end{enumerate}
\end{lemma}

The proof is given in Appendix~\ref{Section: proofs uniform approximation random Fourier features}. More generally, if $u\in\mc F_{\kappa,a}(\Omega)$, then Fourier multiplication and Plancherel's identity give
\[
\|\partial^\beta u\|_{L^2(\Omega)} \le C_{\kappa,a}^{|\beta|+1}(|\beta|!)^{1/a} \|u\|_{\kappa,a} \qquad (\beta\in\mb N_0^d).
\]
Thus $a=1/s<1$ yields the derivative growth of a Gevrey class of order $s>1$, the endpoint $a=1$ corresponds to analytic regularity, and $a>1$ imposes an ultra-analytic derivative bound stronger than analyticity. Bandlimited functions satisfy every exponential Fourier weight and obey the Bernstein estimate $\|\partial^\beta U\|_2\le(\sqrt d\,S)^{|\beta|}\|U\|_2$.

\paragraph{Gelfand--Shilov interpretation of super-exponential decay.}
To make the physical-space content of Assumption~\ref{Assumption: regularities}(d) precise, we recall the Roumieu Gelfand--Shilov classes. For $\alpha,\beta>0$, the space $S_\alpha^\beta(\mb R^d)$ consists of all $f\in C^\infty(\mb R^d)$ for which there exist $C,h>0$ such that
\begin{equation}
\label{eq: Roumieu Gelfand Shilov definition}
\sup_{x\in\mb R^d}|x^\mu\partial^\nu f(x)|
\le C h^{|\mu|+|\nu|}(\mu!)^\alpha(\nu!)^\beta,
\qquad \mu,\nu\in\mb N_0^d.
\end{equation}
With this convention, the lower index controls decay in the physical variable and the upper index controls derivative growth; the Fourier transform interchanges the two indices. The Roumieu space is nontrivial exactly when $\alpha+\beta\ge1$. On the critical line $\alpha+\beta=1$, the existential quantifier in $h$ is essential: the corresponding Beurling space, defined by requiring \eqref{eq: Roumieu Gelfand Shilov definition} for every $h>0$ with a constant depending on $h$, is trivial~\cite[Chapter~6]{NicolaRodino2010}.

The reciprocal indices associated with Assumption~\ref{Assumption: regularities}(d) admit the following equivalent characterizations; see, for example, \cite[Definition~6.1.1, Theorem~6.1.6, and Proposition~6.1.7]{NicolaRodino2010}.

\begin{proposition}[Gelfand--Shilov characterization]
\label{Proposition: Gelfand Shilov characterization}
Let $s>1$, set $q=s/(s-1)$, $\alpha=1-1/s=1/q$, and $\beta=1/s$, and let $f\in L^2(\mb R^d)$. Then the following assertions are equivalent.
\begin{enumerate}
\item[(i)] $f\in S_{1-1/s}^{\,1/s}(\mb R^d)$.
\item[(ii)] There exist $a,\kappa>0$ such that
\[
e^{a|x|^{s/(s-1)}}f\in L^2(\mb R^d), \qquad e^{\kappa|\xi|^s}\widehat f\in L^2(\mb R^d).
\]
\item[(iii)] The function $f$ has a $C^\infty$ representative, and there exist $C,A,a>0$ such that
\[
|\partial^\nu f(x)| \le C A^{|\nu|}(\nu!)^{1/s}e^{-a|x|^{s/(s-1)}}, \qquad x\in\mb R^d,\quad \nu\in\mb N_0^d.
\]
\item[(iv)] The function $f$ has its canonical Schwartz representative, and there exist $C,a,b>0$ such that
\[
|f(x)|\le C e^{-a|x|^{s/(s-1)}}, \qquad |\widehat f(\xi)|\le C e^{-b|\xi|^s}, \qquad x,\xi\in\mb R^d.
\]
\end{enumerate}
The constants in the different assertions need not coincide. In particular, when a pointwise estimate is converted into a weighted $L^2$ estimate, the weight exponent must in general be chosen strictly smaller than the exponent in the pointwise bound.
\end{proposition}

Applied to the extension $U$ in Assumption~\ref{Assumption: regularities}(d), Proposition~\ref{Proposition: Gelfand Shilov characterization} shows that the Fourier condition in (d), supplemented by $e^{a|x|^{s/(s-1)}}U\in L^2(\mb R^d)$ for some $a>0$, is equivalent to $U\in S_{1-1/s}^{1/s}(\mb R^d)$. Thus condition (d) supplies the frequency-side requirement of the Gelfand--Shilov characterization, while the additional weighted condition records the matching physical-space decay.

Together with Lemma~\ref{Lemma: Fourier decay from Gevrey and analytic regularity}, this characterization connects Assumption~\ref{Assumption: regularities} to the rate regimes in Theorems~\ref{Theorem: interpolation improved convergence rate of RFM} and \ref{Theorem: optimal growing bandwidth leverage rates}.
Condition (a) gives the Sobolev case of the former theorem, conditions (b) and (c) yield its stretched-exponential case, and conditions (d) and (e) give the super-exponential and bandlimited cases. The latter theorem follows the same correspondence, with $W^{s,p}(\Omega)$ in its Sobolev case. At the Fourier level, increasingly rapid decay, from polynomial through stretched-exponential, exponential, and super-exponential decay to compact spectral support, yields the corresponding hierarchy of approximation rates.

\section{RFM discretizations of PDE and eigenvalue problems}
\label{Section: RFM PDE solvers}
This section establishes abstract error estimates for RFM discretizations. These bounds are independent of the particular RF realization and rely only on the stability of the underlying PDEs and the approximation capacity of the trial space. Throughout, $V_{\mathrm{RF}}$ denotes a random linear trial space generated by finite RF expansions. Typical examples are the randomly shifted cosine space in \eqref{Eq: linear combination of random Fourier features, cos(wx+b)} and the equivalent cosine--sine space in \eqref{Eq: linear combination of random Fourier features, cosine-sine}; more generally, $V_{\mathrm{RF}}$ is the range of the feature operator $\Phi$ introduced in Section~\ref{subsection: ridge approximation in interpolation norm}. Here, $\mathbf{n}$ denotes the unit outward normal vector on $\partial\Omega$.

\subsection{Strong-form RFM}
Consider the boundary value problem of a general second order linear differential operator
\begin{equation} \label{elliptic equation with variable coefficients}
\left\{\begin{aligned}
Lu &:= \sum_{k,l=1}^{d} a_{kl}\partial_{kl}u+\sum_{k=1}^{d} b_{k}\partial_{k}u+cu=f, && x\in\Omega, \\
Bu &:= g_{1}\frac{\partial u}{\partial \mf n}+g_{2}u=g, && x\in\partial\Omega .
\end{aligned}\right.
\end{equation}
where $f \in L^{2}(\Omega)$ and $ g \in L^{2}(\partial \Omega)$.
{\red The operator $B$ accommodates general boundary conditions. If a Dirichlet condition is imposed on $\partial \Omega_{D} \subset \partial \Omega$ and a Neumann or Robin condition is imposed on $\partial \Omega_{N} = \partial \Omega \setminus \partial \Omega_{D}$, then $g_{1} = \tilde{g}_{1}\mathbf{1}_{\partial \Omega_{N}}$ and $g_{2} = \mathbf{1}_{\partial \Omega_{D}} + \tilde{g}_{2}\mathbf{1}_{\partial \Omega_{N}}$. } We impose the following assumptions on the coefficients of the operators $L$ and $B$.

\begin{assumption} \label{Assumption: Linf norm of the coefficients of linear differential operators L, B}
Let $A = (a_{kl})_{d\times d}$ be a symmetric matrix and $b = (b_{1}, \dots, b_{d})^{\top}$. There exist positive constants $\{\Lambda_{i}\}_{i =1}^{5}$ such that for a.e. $x \in \Omega$, $\|A(x)\|_{2} \leq \Lambda_{1}$, $|b(x)| \leq \Lambda_{2}$, $|c(x)| \leq \Lambda_{3}$, and $\|g_{1}\|_{L^{\infty}(\partial \Omega)} \le \Lambda_{4}$, $\|g_{2}\|_{L^{\infty}(\partial \Omega)} \le \Lambda_{5}$.
\end{assumption}

The strong-form loss function associated with Problem~\eqref{elliptic equation with variable coefficients} is defined as
\begin{equation}
\begin{aligned} \label{eq: strong-form loss for solving: elliptic equation with variable coefficients}
\mathcal{L}(u)= & \left\|L u-f\right\|_{L^{2}(\Omega)}^{2} + \gamma \left\|B u-g\right\|_{H^{s}(\partial \Omega)}^{2} ,
\end{aligned}
\end{equation}
where $\gamma > 0$ is the boundary penalty parameter, $s \ge 0$, and we assume $g \in H^s(\partial\Omega)$. We evaluate the boundary residual in the general $H^s(\partial\Omega)$ norm, as recent studies \cite{Muller2022Notes,Doyoon2025Trace,Zhou2026SSBEPINN} demonstrate that such stronger trace-space control may enhance the stability and accuracy of the PDE approximations.
The RFM solution is then obtained by \begin{equation*}
u_{\mathrm{RF}} \in \mathop{\arg\min}_{u \in V_{\mathrm{RF}}} \mathcal{L}(u).
\end{equation*}

Under certain conditions, the error of $u_{\mathrm{RF}}$ can be bounded by the approximation error of $u$ by $V_{\mathrm{RF}}$, as established in the following theorem.
\begin{theorem}[Strong-form RFM]
\label{Theorem: abstract error estimate for Strong-form RFM}
Assume that $\partial\Omega$ is smooth, $L$ is properly elliptic, and the coefficients of $L$ and $B$ are in $C^{\infty}(\bar{\Omega})$. Let $B$ be a normal boundary operator of order $l$ covering $L$, where $l=0$ for a Dirichlet boundary condition and $l=1$ for a Neumann or Robin boundary condition. Let $0\le s\le 3/2-l$, assume that $V_{\mathrm{RF}}\subset H^{2}(\Omega)$, and let $u\in H^{2}(\Omega)$ be a solution to Problem~\eqref{elliptic equation with
variable coefficients}. Then there exists a constant $C > 0$, independent of $u$ and $u_{\mathrm{RF}}$, such that $$ \inf_{v \in V}
\|u-u_{\mathrm{RF}}-v\|_{H^{s + l + 1/2}(\Omega)} \leq C \inf_{w \in V_{\mathrm{RF}}} \|u-w\|_{H^{2}(\Omega)}, $$ where $V = \{v \in C^{\infty}(\bar{\Omega}) :
L v = 0 \text{ in } \Omega, B v = 0 \text{ on } \partial\Omega\}$ is the null space. In particular, if the problem admits a unique solution, the estimate simplifies to $ \|u-u_{\mathrm{RF}}\|_{H^{s+l+1/2}(\Omega)} \leq C \inf_{w \in V_{\mathrm{RF}}} \|u-w\|_{H^{2}(\Omega)}. $
\end{theorem}

This result accommodates various boundary conditions and problems with a nontrivial null space, such as the Laplace equation with a Neumann boundary condition.
In certain cases, the regularity assumptions on the coefficients and the domain boundary can be further relaxed~\cite{Geymonat1965,Grisvard2011}; this issue is beyond the scope of the present paper.

The proof is given in Appendix~\ref{Section: proof strong-form RFM}.

\subsection{Weak-form RFM}
Let $V$ be a Hilbert space with norm $\|\cdot\|_V$. Consider the variational problem: find $u^* \in V$ such that
\begin{equation}\label{eq: weak variational problem}
a(u^*,v)=F(v),\qquad \forall v\in V,
\end{equation}
where $a:V\times V\to\mathbb R$ is a bilinear form and $F:V\to\mathbb R$ is a bounded linear functional. Assume that $a$ is continuous and coercive, i.e., there exist constants $M,\alpha>0$ such that
\begin{equation}\label{continuity and coercivity of a}
|a(w,v)|\le M\|w\|_V\|v\|_V,\qquad a(v,v)\ge \alpha\|v\|_V^2,\qquad \forall w,v\in V.
\end{equation}
By the Lax--Milgram theorem, Problem \eqref{eq: weak variational problem} admits a unique solution. Let $V_{\mathrm{RF}}\subset V$ be the RF trial space. The weak-form RFM solution $u_{\mathrm{RF}}\in V_{\mathrm{RF}}$ is defined by
\begin{equation*}\label{weak form RFM}
a(u_{\mathrm{RF}},v)=F(v), \qquad \forall v\in V_{\mathrm{RF}}.
\end{equation*}
A direct application of C\'ea's lemma~\cite{Cea1964Approximation} gives

\begin{theorem}[Weak-form RFM] \label{thm: abstract error estimate for weak form RFM}
Under Assumption \eqref{continuity and coercivity of a}, it holds that
\begin{equation*} \label{eq: abstract weak form RFM error estimate}
\|u^{*}-u_{\mathrm{RF}}\|_{V}
\le \frac{M}{\alpha}\inf_{w\in V_{\mathrm{RF}}}\|u^{*}-w\|_{V}.
\end{equation*}
\end{theorem}

We now give a concrete example. Consider the Neumann boundary value problem
\begin{equation} \label{eqs: elliptic equation with Neumann BC for weak form}
\begin{cases}
- \nabla \cdot\left(A(x)\nabla u\right) + c u = f, &  x \in \Omega, \\
A(x)\nabla u\cdot \mf{n} = g, &  x \in \partial \Omega,
\end{cases}
\end{equation}
where  $f\in L^{2}(\Omega)$ and $g\in H^{-1/2}(\partial\Omega)$. We make the following assumption on $A$ and $c$.
\begin{assumption} \label{assumption: coefficients for Neumann weak form}
Let $A\in L^{\infty}(\Omega;\mb{R}^{d\times d})$ be symmetric and let $c\in L^{\infty}(\Omega)$. There exist positive constants $\Lambda_{1}$, $\underline{a}$, $\Lambda_{3}$ and $\underline{c}$ such that $\underline{a}|\xi|^{2}\le \xi^{\top}A(x)\xi \le \Lambda_{1}|\xi|^{2}$ and $\underline{c}\le c(x)\le \Lambda_{3}$ for a.e. $x\in\Omega$ and all $\xi\in\mb{R}^{d}$.
\end{assumption}

The weak formulation of Problem \eqref{eqs: elliptic equation with Neumann BC for weak form} fits the abstract form \eqref{eq: weak variational problem} with $V=H^{1}(\Omega)$. More precisely,
\begin{equation} \label{eq: bilinear form a and linear form F}
\begin{aligned}
a(u,v) &= \int_{\Omega}\left(A(x)\nabla u\cdot \nabla v + c u v\right)\mr{d}x, \\
F(v) &= \int_{\Omega} f v \mr{d}x + \langle g, v\rangle_{\partial\Omega}.
\end{aligned}
\end{equation}
Here $\langle g,v\rangle_{\partial\Omega}$ denotes the duality pairing between $H^{-1/2}(\partial\Omega)$ and $H^{1/2}(\partial\Omega)$ through the trace of $v$; if $g\in L^{2}(\partial\Omega)$, it reduces to $\int_{\partial\Omega}gv\,\mr{d}s$. Under Assumption \ref{assumption: coefficients for Neumann weak form}, the bilinear form $a$ is continuous and coercive on $H^{1}(\Omega)$ with $M=\max\{\Lambda_{1},\Lambda_{3}\}$ and $\alpha=\min\{\underline{a},\underline{c}\}$.

\subsection{RFM for Eigenvalue Problems}
We now consider the RFM approximation for elliptic eigenvalue problems. Let $a(\cdot,\cdot)$ be a symmetric, continuous and coercive bilinear form as in the weak formulation, and let $m(\cdot,\cdot)$ be a symmetric, continuous and nonnegative bilinear form on $V\times V$. Assume that the solution operator $T:V\to V$ defined by $a(Tf,v)=m(f,v)$ for all $v\in V$ is compact. The continuous eigenvalue problem is to find $(\lambda,u)\in\mathbb R\times V$, $u\ne0$, such that
\begin{equation} \label{eq: abstract variational eigenvalue problem}
a(u,v)=\lambda m(u,v), \qquad \forall v\in V.
\end{equation}
The RFM eigenvalue approximation is to find $(\lambda_{\mathrm{RF}},u_{\mathrm{RF}})\in\mathbb R\times V_{\mathrm{RF}}$, $u_{\mathrm{RF}}\ne0$, such that
\begin{equation} \label{eq: abstract RF eigenvalue problem}
a(u_{\mathrm{RF}},v_{\mathrm{RF}})=\lambda_{\mathrm{RF}}
m(u_{\mathrm{RF}},v_{\mathrm{RF}}), \quad
\forall v_{\mathrm{RF}}\in V_{\mathrm{RF}}.
\end{equation}
Let $\lambda$ be an eigenvalue of \eqref{eq: abstract variational eigenvalue problem} with multiplicity $q$, and let $E$ be the corresponding eigenspace. Define the best approximation error of $E$ in $V_{\mathrm{RF}}$ by
\begin{equation*}
\eta_{\mathrm{RF}}(E):=\sup_{u\in E,\  \|u\|_{V}=1}
\inf_{w_{\mathrm{RF}}\in V_{\mathrm{RF}}}\|u-w_{\mathrm{RF}}\|_{V}.
\end{equation*}
The following theorem is a direct consequence of \cite[Theorems 9.12 and 9.13]{Boffi2010FEAeigenvalue}.
\begin{theorem}[RFM for eigenvalue problems]\label{Theorem: abstract error estimate for RFM eigenvalue problems}
Under the above assumptions, for $V_{\mathrm{RF}}$ sufficiently rich, the discrete problem \eqref{eq: abstract RF eigenvalue problem} has exactly $q$ eigenvalues $\{\lambda_{\mathrm{RF},j}\}_{j=1}^{q}$ converging to $\lambda$, counted with multiplicity. Let $E_{\mathrm{RF}}$ be the space spanned by the corresponding discrete eigenfunctions. Then, there exists a constant $C>0$, independent of $V_{\mathrm{RF}}$, such that
\begin{equation*}
\begin{aligned}
& \sup_{u\in E,\  \|u\|_V=1}\inf_{v\in E_{\mathrm{RF}}}\|u-v\|_V
\le C\eta_{\mathrm{RF}}(E), \\
& \max_{1\le j\le q}|\lambda-\lambda_{\mathrm{RF},j}|
\le C\eta_{\mathrm{RF}}^2(E).
\end{aligned}
\end{equation*}
\end{theorem}
Theorem \ref{Theorem: abstract error estimate for RFM eigenvalue problems} indicates that the convergence rate of the eigenspace matches the best approximation error, whereas the eigenvalues converge at twice this rate.

We now give a concrete example. Consider the elliptic eigenvalue problem with a Neumann boundary condition
\begin{equation*}
\begin{cases}
- \nabla \cdot\left(A(x)\nabla u\right) + c u = \lambda u, &  x \in \Omega, \\
A(x)\nabla u\cdot \mf{n} = 0, &  x \in \partial \Omega .
\end{cases}
\end{equation*}
Under Assumption \ref{assumption: coefficients for Neumann weak form}, this problem fits the abstract formulation \eqref{eq: abstract variational eigenvalue problem} with $V=H^{1}(\Omega)$, $a$ defined as in \eqref{eq: bilinear form a and linear form F}, and $m(u,v)=\la u,v\ra_{L^{2}(\Omega)}$. Since $a$ is symmetric, continuous and coercive on $H^{1}(\Omega)$, and the embedding $H^{1}(\Omega)\hookrightarrow L^{2}(\Omega)$ is compact, the associated solution operator is compact.

\section{Singular values and exponential ill-conditioning of RFMtxs} \label{Section: Exponential ill conditionality of random feature matrices}
In this part, we analyze the exponential ill-conditioning of RFMtxs. In practical calculations, collocation points $\{x_{i}\}_{i = 1}^{n_{1}} \subset \Omega$, $\{x_{i}\}_{i = n_{1}+1}^{n_{1}+n_{2}} \subset \partial\Omega$, are selected for numerical integration to approximate the loss $\mc{L}(\boldsymbol{\alpha})$.
Denote the total number of collocation points by $n=n_{1}+n_{2}$. Without loss of generality, in this part we set $\gamma=1$ and $n>2N$. Assume $\Omega\subset(-1,1)^{d}$. Otherwise, a scaling argument can be applied. Define the operator $\widetilde{L}$ by $\widetilde{L} = L$ in $\Omega$ and $\widetilde{L} = B$ on $\partial\Omega$. The vector $\boldsymbol{\alpha}$ of trainable parameters can be obtained by solving the least squares problem $\mf{\Psi\boldsymbol{\alpha}} = \mf{F}$, where $\mf{\Psi} = \left(\Psi_{i j}\right) \in \mb{R}^{n \times 2 N}$ is the RFMtx with $\Psi_{i j} = \widetilde{L}\psi_{j}(x_{i})$, and $\mf{F} = (f(x_{1}), \ldots, f(x_{n_{1}}), g(x_{n_{1}+1}), \ldots, g(x_{n_{1}+n_{2}}))^{\top}$. Each column of $\mf{\Psi}$ corresponds to a feature and each row corresponds to a collocation point. For Fourier features, $\psi_{j}=\cos(k_{j}^{\top}\cdot)$ and $\psi_{j+N}=\sin(k_{j}^{\top}\cdot)$ for $1\le j\le N$. For $\tanh$ features, $\psi_{j} = \tanh(k_{j}^{\top}\cdot+v_{j})$ for $1\le j\le 2N$.

Let $\{\sigma_m(\mf{\Psi})\}_{m = 1}^{2N}$ be the singular values of $\mf{\Psi}$ in descending order. The following theorem estimates the decay rate of $\sigma_{m}$ and shows that the condition number of $\mf{\Psi}$ may be extremely large.
\begin{theorem}[Fast decay of singular values]
\label{Theorem: fast decay of singular values of RFMtx}
Let $M>0$ be sufficiently large in terms of $d$ and $S$, and let $m\in\mb N$ satisfy $3M\ln(15M)+1\le m\le2N$.

(1) If $\mf\Psi$ is generated by Fourier features with $k_j\in[-S,S]^d$, $1\le j\le N$, then
\begin{equation}
\label{ineq: Fourier fast decay of singular values}
\sigma_m
\lesssim_{\Omega,d,\{\Lambda_i\}_{i=1}^{5},S}
\sqrt{nN}\exp\bigl(-a_{\mathrm F}M^{1/d}\ln M\bigr),
\end{equation}
where one may take $a_{\mathrm F}:=3\cdot 2^{1-1/d}/(112e)$.

(2) If $\mf\Psi$ is generated by $\tanh$ features with $k_j\in[-S,S]^d$, $1\le j\le2N$, then
\begin{equation}
\label{ineq: tanh fast decay of singular values}
\sigma_m
\lesssim_{\Omega,d,\{\Lambda_i\}_{i=1}^{5},S}
\sqrt{nN}\exp\bigl(-a_{\mathrm T}M^{1/d}\bigr),
\end{equation}
where $a_{\mathrm T}=a_{\mathrm T}(d,S)>0$ can be chosen explicitly in terms of $d$ and $S$.
\end{theorem}

\begin{theorem}[Lower bounds for condition number]
\label{Theorem: exponential ill conditionality of random feature matrices}
Suppose that there is a constant $\underline a>0$ such that $-\xi^{\top}A(x)\xi\ge \underline a|\xi|^2$ and $c(x)\ge0$ for all $x\in\Omega$ and $\xi\in\mb R^d$.
Let $M>0$ be sufficiently large in terms of $d$ and $S$, and assume $2N-1\ge3M\ln(15M)$. Let $a_{\mathrm F}$ and $a_{\mathrm T}$ be as in Theorem~\ref{Theorem: fast decay of singular values of RFMtx}. Use the convention $\sigma_1/\sigma_{2N}=\infty$ when $\sigma_{2N}=0$.

(1) \textbf{Fourier features.} Let $k_j\in[-S,S]^d$. Define $C,G\ge0$ by $C^2:=n_1^{-1}\sum_{i=1}^{n_1}c^2(x_i)$ and $G^2:=n_2^{-1}\sum_{i=n_1+1}^{n}g_2^2(x_i)$. Then
\[
\frac{\sigma_1}{\sigma_{2N}} \gtrsim_{\Omega,d,\{\Lambda_i\}_{i=1}^5,S} \frac{\sqrt{n_1}(\underline a+C)+\sqrt{n_2}G}{\sqrt{nN}} \exp\bigl(a_{\mathrm
F}M^{1/d}\ln M\bigr).
\]

(2) \textbf{Tanh features.} Let $k_j\in[-S,S]^d$ and $|v_j|\le dS$. Assume $b\equiv0$ and that, for some $r_0>0$ independent of $M$ and $N$, the interior collocation points satisfy
\begin{equation}
\label{assump: uniform directional spread of interior collocation points}
\inf_{\theta\in\mathbb S^{d-1}}
\left(
\max_{1\le i\le n_1}\theta^{\top}x_i
-\min_{1\le i\le n_1}\theta^{\top}x_i
\right)\ge2r_0.
\end{equation}
Then, setting $\widetilde a_{\mathrm T}:=\min\{a_{\mathrm T},d\ln2/(4e)\}$, we have
\begin{equation}
\label{ineq: tanh condition number lower bound}
\frac{\sigma_1}{\sigma_{2N}}
\gtrsim_{\Omega,d,\{\Lambda_i\}_{i=1}^5,S,r_0}
\frac{\underline a}{\sqrt{nN}}
\exp\bigl(\widetilde a_{\mathrm T}M^{1/d}\bigr).
\end{equation}
\end{theorem}

\begin{remark}
Assumption~\eqref{assump: uniform directional spread of interior collocation points} is a quantitative full-dimensionality condition on the interior point cloud. It follows, in particular, if there exists $x_\circ\in\mb R^d$ such that
\[
\overline B(x_\circ,r_0) \subset\operatorname{conv}\{x_i:1\le i\le n_1\}.
\]
Indeed, maximizing and minimizing a linear functional over the convex hull and then over the enclosed ball give a directional width of at least $2r_0$.
\end{remark}

The auxiliary estimates and the proofs of Theorems~\ref{Theorem: fast decay of singular values of RFMtx} and \ref{Theorem: exponential ill conditionality of random feature matrices} are given in Appendix~\ref{Section: proofs for singular value estimates and condition number lower bounds}. Their common key step is to approximate all feature columns simultaneously in a common low-dimensional trial space and then invoke the min--max characterization of singular values; the spectral-accuracy estimates of Section~\ref{Section: Estimate of d(lambda) and the convergence rate, interpolation} make the resulting residual exponentially or super-exponentially small. Accordingly, within the regularity hierarchy covered by Theorem~\ref{Theorem: interpolation improved convergence rate of RFM}, greater smoothness of the activation, as expressed by stronger Fourier regularity of the associated weighted feature family, leads to faster singular-value decay. This reveals an intrinsic tradeoff: the spectral approximation power of RFM is accompanied by severe exponential ill-conditioning, which presents a major obstacle to stable high-precision solution of the resulting least-squares systems.

Numerical experiments further suggest that singular values may decay more slowly as the dimension $d$ increases, with rank deficiency most pronounced in one dimension and partially mitigated in higher dimensions \cite{ChenESun2024Optimization}. This trend is qualitatively consistent with our bounds, whose decay exponents scale with $M^{1/d}$; for fixed $M$, this scale decreases as $d$ grows.

Localization by a partition of unity or a compatible domain decomposition can mitigate rapid singular-value decay: \cite[Theorem~3.7]{MingYu2026Spectral} gives two-sided global-to-local singular-value bounds in one dimension, while numerical evidence supports the benefit of localization and local feature filtering \cite{ChenESun2024Optimization,vanBeekDoleanMoseley2026Filtering}. Complementary remedies include overlapping Schwarz preconditioners and randomized sketching-based right preconditioners, both supported by analysis and numerical experiments \cite{ShangHeinleinMishraWang2025,TanChen2026Preconditioned}.

\section{Conclusion}\label{Section: Conclusion}

We establish a multidimensional approximation theory for RFM in the interpolation scale of the associated kernel integral operator and apply it to general second-order elliptic PDEs and eigenvalue problems. We prove high-probability convergence rates ranging from super-exponential to algebraic, depending on the regularity of the target, under both regularity-adapted sampling and uniform sampling on growing frequency windows. On a single target-independent event, one sampled space approximates an entire source ball and, for each target, one norm-independent coefficient vector defines an approximant that attains spectral accuracy simultaneously in all admissible error norms. The abstract RFM solver estimates then transfer these approximation bounds to convergence estimates for strong- and weak-form RFM discretizations. Additionally, we demonstrate super-exponential singular-value decay for RFMtxs generated by Fourier features and exponential decay for those generated by $\tanh$ features, together with corresponding condition-number lower bounds. The analysis identifies spectral approximation as the common mechanism behind high accuracy and severe ill-conditioning, and extends the one-dimensional Fourier analysis of \cite{MingYu2026Spectral} to a multidimensional, operator-theoretic framework.

A standard RF expansion employs basis functions of the form $h_j(x)=\eta(w_j^{\top}x+b_j)$. While the explicit approximation results in Section~\ref{Section: Estimate of d(lambda) and the convergence rate, interpolation} focus on trigonometric functions, the abstract estimate in Section~\ref{Section: Approximation error by interpolation} applies to any feature representation for which the kernel interpolation spaces, source conditions, and effective dimensions can be controlled. Extending spectral convergence to other activation families therefore reduces to establishing these three ingredients. The singular-value analysis in Section~\ref{Section: Exponential ill conditionality of random feature matrices} treats each feature column as a target to be approximated in a common low-dimensional space. Hence, the same mechanism extends to bounded feature families whose derivatives required by the PDE operator admit uniform spectral approximation; the result for bounded $\tanh$ features provides one concrete analytic example.

\begingroup
\renewcommand{\baselinestretch}{1.0}\small\selectfont
\setlength{\bibsep}{0pt plus 0.25pt}
\setlength{\emergencystretch}{2em}
\sloppy \hbadness=2000
\bibliographystyle{sn-mathphys-num}
\bibliography{reference}

@inproceedings{rahimi2007random,
  title={Random features for large-scale kernel machines},
  author={Rahimi, Ali and Recht, Benjamin},
  booktitle={Advances in Neural Information Processing Systems},
  volume={20},
  pages={1177--1184},
  year={2007}
}

@article{Ogawa1988Pseudo,
author = {Ogawa, Hidemitsu},
title = {An Operator Pseudo-Inversion Lemma},
journal = {SIAM J. Appl. Math.},
volume = {48},
number = {6},
pages = {1527-1531},
year = {1988},
doi = {10.1137/0148095},
}

@inproceedings{rudi2017generalization,
  title={Generalization properties of learning with random features},
  author={Rudi, Alessandro and Rosasco, Lorenzo},
  booktitle={Advances in Neural Information Processing Systems},
  volume={30},
  pages={3215--3225},
  year={2017}
}

@article{Minsker2017extensions,
title = {On some extensions of {B}ernstein's inequality for self-adjoint operators},
author = {Stanislav Minsker},
journal = {Stat. Probab. Lett.},
volume = {127},
pages = {111-119},
year = {2017},
doi = {10.1016/j.spl.2017.03.020},
}

@article{Vyacheslav1999ExtensionsBesov,
author = {Rychkov, Vyacheslav S.},
title = {On Restrictions and Extensions of the {B}esov and {T}riebel--{L}izorkin Spaces with Respect to {L}ipschitz Domains},
journal = {J. Lond. Math. Soc.},
volume = {60},
number = {1},
pages = {237-257},
year = {1999},
doi = {10.1112/S0024610799007723}
}

@book{rodino1993linear,
  title={Linear partial differential operators in Gevrey spaces},
  author={Rodino, Luigi},
  year={1993},
  publisher={World Scientific},
  address={Singapore},
  doi={10.1142/1550}
}

@book{NicolaRodino2010,
  author    = {Nicola, Fabio and Rodino, Luigi},
  title     = {Global Pseudo-Differential Calculus on Euclidean Spaces},
  series    = {Pseudo-Differential Operators: Theory and Applications},
  volume    = {4},
  publisher = {Birkh{\"a}user},
  address   = {Basel},
  year      = {2010},
  doi       = {10.1007/978-3-7643-8512-5}
}

@book{Reed1978Methods,
  title={II: {F}ourier Analysis, Self-Adjointness},
  author={Reed, M.  and  Simon, B.},
  series={Methods of Modern Mathematical Physics},
  volume={2},
  year={1975},
  publisher={Academic Press},
  address={New York}
}

@book{Simon2005trace,
  title={Trace ideals and their applications},
  author={Simon, Barry},
  number={120},
  year={2005},
  publisher={American Mathematical Society},
  address={Providence, RI},
  doi={10.1090/surv/120}
}

@article{widom1963asymptotic1,
    AUTHOR = {Widom, Harold},
     TITLE = {Asymptotic behavior of the eigenvalues of certain integral
              equations},
   JOURNAL = {Trans. Amer. Math. Soc.},
    VOLUME = {109},
      YEAR = {1963},
     PAGES = {278--295},
       DOI = {10.1090/S0002-9947-1963-0155161-0},
}

@article{widom1964asymptotic2,
    AUTHOR = {Widom, Harold},
     TITLE = {Asymptotic behavior of the eigenvalues of certain integral
              equations. {II}},
   JOURNAL = {Arch. Rational Mech. Anal.},
    VOLUME = {17},
      YEAR = {1964},
     PAGES = {215--229},
       DOI = {10.1007/BF00282438},
}

@book{Berlinet2004RKHS,
    AUTHOR = {Berlinet, Alain and Thomas-Agnan, Christine},
     TITLE = {Reproducing kernel {H}ilbert spaces in probability and statistics},
 PUBLISHER = {Springer},
   ADDRESS = {New York, NY},
      YEAR = {2004},
       DOI = {10.1007/978-1-4419-9096-9},
}

@article{bach2017equivalence,
  title={On the equivalence between kernel quadrature rules and random feature expansions},
  author={Bach, Francis},
  journal={J. Mach. Learn. Res.},
  volume={18},
  number={21},
  pages={1--38},
  year={2017}
}

@book{Artin1964gamma,
    AUTHOR = {E. Artin},
     TITLE = {The {G}amma function},
    SERIES = {Athena Series: Selected Topics in Mathematics},
 PUBLISHER = {Holt, Rinehart and Winston},
   ADDRESS = {New York},
      YEAR = {1964},
     PAGES = {vii+39},
}

@book{Cordes1987, 
title={Spectral Theory of Linear Differential Operators and Comparison Algebras}, 
author={Cordes, Heinz Otto}, 
year={1987}, 
publisher={Cambridge University Press}, 
address={Cambridge}, 
series={London Mathematical Society Lecture Note Series},
doi={10.1017/CBO9780511662836},
}

@book{Wendland2005Scattered,
    AUTHOR = {Wendland, H.},
     TITLE = {Scattered data approximation},
    SERIES = {Cambridge Monographs on Applied and Computational Mathematics},
    VOLUME = {17},
 PUBLISHER = {Cambridge University Press},
   ADDRESS = {Cambridge},
      YEAR = {2004},
     PAGES = {x+336},
       DOI = {10.1017/CBO9780511617539},
}

@article{Steinwart2012Mercer,
    AUTHOR = {Steinwart, I. and Scovel, C.},
     TITLE = {Mercer's theorem on general domains: on the interaction between measures, kernels, and {RKHS}s},
   JOURNAL = {Constr. Approx.},
    VOLUME = {35},
      YEAR = {2012},
    NUMBER = {3},
     PAGES = {363--417},
       DOI = {10.1007/s00365-012-9153-3},
}

@misc{Long2025optimal,
  author        = {Long, J. and Peng, X. and Wu, L.},
  title         = {Optimal Rates and Saturation for Noiseless Kernel Ridge Regression},
  year          = {2024},
  eprint        = {2402.15718},
  archiveprefix = {arXiv},
  note          = {arXiv:2402.15718v2},
  doi           = {10.48550/arXiv.2402.15718}
}

@article{Peetre1964interpolation,
    AUTHOR = {Peetre, J.},
     TITLE = {On an interpolation theorem of {F}oia\c s{} and {L}ions},
   JOURNAL = {Acta Sci. Math. (Szeged)},
  FJOURNAL = {Acta Universitatis Szegediensis. Acta Scientiarum Mathematicarum},
    VOLUME = {25},
    NUMBER = {3--4},
      YEAR = {1964},
     PAGES = {255--261},
}

@InProceedings{Muller2022Notes,
  title = 	 {Notes on Exact Boundary Values in Residual Minimisation},
  author =       {M\"{u}ller, Johannes and Zeinhofer, Marius},
  booktitle = 	 {Proceedings of Mathematical and Scientific Machine Learning},
  pages = 	 {231--240},
  year = 	 {2022},
  volume = 	 {190},
  series = 	 {Proceedings of Machine Learning Research},
  publisher =    {PMLR},
  address =      {Beijing, China},
}

@misc{Doyoon2025Trace,
  author        = {Kim, Doyoon and Song, Junbin},
  title         = {Trace Regularity PINNs: Enforcing $\mathrm{H}^{\frac{1}{2}}(\partial \Omega)$ for Boundary Data},
  year          = {2025},
  eprint        = {2510.16817},
  archiveprefix = {arXiv},
  note          = {arXiv:2510.16817},
  doi           = {10.48550/arXiv.2510.16817}
}

@misc{Zhou2026SSBEPINN,
  author        = {Zhou, Qixuan and Chen, Chuqi and Luo, Tao and Xiang, Yang},
  title         = {SSBE-PINN: A Sobolev Boundary Scheme Boosting Stability and Accuracy in Elliptic/Parabolic PDE Learning},
  year          = {2025},
  eprint        = {2508.10322},
  archiveprefix = {arXiv},
  note          = {Accepted for publication in Communications in Computational Physics; arXiv:2508.10322v2},
  doi           = {10.48550/arXiv.2508.10322}
}

@book{LionsMagenes1972,
  author    = {Lions, J.-L. and Magenes, E.},
  title     = {Non-Homogeneous Boundary Value Problems and Applications},
  volume    = {1},
  series    = {Grundlehren der mathematischen Wissenschaften},
  number    = {181},
  publisher = {Springer-Verlag},
  address   = {Berlin, Heidelberg},
  year      = {1972},
  pages     = {XVI+360},
  doi       = {10.1007/978-3-642-65161-8}
}

@article{Geymonat1965,
  author  = {Geymonat, Giuseppe},
  title   = {Sui problemi ai limiti per i sistemi lineari ellittici},
  journal = {Ann. Mat. Pura Appl.},
  volume  = {69},
  number  = {1},
  pages   = {207--284},
  year    = {1965},
  doi     = {10.1007/BF02414374},
}

@book{Grisvard2011,
  author    = {Grisvard, Pierre},
  title     = {Elliptic Problems in Nonsmooth Domains},
  series    = {Classics in Applied Mathematics},
  number    = {69},
  publisher = {Society for Industrial and Applied Mathematics},
  address   = {Philadelphia, PA},
  year      = {2011},
  isbn      = {978-1-61197-202-3},
  pages     = {xx+410},
  doi       = {10.1137/1.9781611972030}
}

@article{Cea1964Approximation,
    AUTHOR = {C\'ea, Jean},
     TITLE = {Approximation variationnelle des probl\`emes aux limites},
   JOURNAL = {Ann. Inst. Fourier (Grenoble)},
    VOLUME = {14},
    NUMBER = {2},
      YEAR = {1964},
     PAGES = {345--444},
       DOI = {10.5802/aif.181},
}

@article {Boffi2010FEAeigenvalue,
    AUTHOR = {Boffi, Daniele},
     TITLE = {Finite element approximation of eigenvalue problems},
   JOURNAL = {Acta Numer.},
    VOLUME = {19},
      YEAR = {2010},
     PAGES = {1--120},
       DOI = {10.1017/S0962492910000012},
}

@book {Adams2003Sobolev,
    AUTHOR = {Adams, Robert A. and Fournier, John J. F.},
     TITLE = {Sobolev spaces},
    SERIES = {Pure and Applied Mathematics},
    VOLUME = {140},
   EDITION = {Second},
 PUBLISHER = {Academic Press},
   ADDRESS = {Amsterdam},
      YEAR = {2003},
     PAGES = {xiv+305},
}

@book{PaulsenRaghupathi2016,
  author    = {Paulsen, Vern I. and Raghupathi, Mrinal},
  title     = {An Introduction to the Theory of Reproducing Kernel Hilbert Spaces},
  series    = {Cambridge Studies in Advanced Mathematics},
  volume    = {152},
  publisher = {Cambridge University Press},
  address   = {Cambridge},
  year      = {2016},
  doi       = {10.1017/CBO9781316219232},
}

@book{Hormander2003,
  author    = {H{\"o}rmander, Lars},
  title     = {The Analysis of Linear Partial Differential Operators I: Distribution Theory and Fourier Analysis},
  series    = {Classics in Mathematics},
  edition   = {2},
  publisher = {Springer},
  address   = {Berlin, Heidelberg},
  year      = {2003},
  doi       = {10.1007/978-3-642-61497-2},
}

@inproceedings{Liu2021DistributedRF,
  author    = {Liu, Yong and Liu, Jiankun and Wang, Shuqiang},
  title     = {Effective Distributed Learning with Random Features: Improved Bounds and Algorithms},
  booktitle = {International Conference on Learning Representations},
  year      = {2021}
}

@article{Kammonen2020AdaptiveRFF,
  author  = {Kammonen, Aku and Kiessling, Jonas and Plech{\'a}{\v c}, Petr and Sandberg, Mattias and Szepessy, Anders},
  title   = {Adaptive Random {Fourier} Features with {Metropolis} Sampling},
  journal = {Found. Data Sci.},
  volume  = {2},
  number  = {3},
  pages   = {309--332},
  year    = {2020},
  doi     = {10.3934/fods.2020014}
}

@inproceedings{Chen2021FastLeverageKRR,
  author    = {Chen, Yifan and Yang, Yun},
  title     = {Fast Statistical Leverage Score Approximation in Kernel Ridge Regression},
  booktitle = {Proceedings of The 24th International Conference on Artificial Intelligence and Statistics},
  series    = {Proceedings of Machine Learning Research},
  volume    = {130},
  pages     = {2935--2943},
  year      = {2021},
  address   = {Virtual},
  publisher = {PMLR}
}

@phdthesis{Li2020RFMethods,
  author = {Li, Zhu},
  title  = {On the Properties of Random Feature Methods},
  school = {University of Oxford},
  type   = {{DPhil} thesis},
  year   = {2020},
  doi    = {10.5287/ora-g7rga2w20}
}

@inproceedings{Rudi2018LeverageSampling,
  author    = {Rudi, Alessandro and Calandriello, Daniele and Carratino, Luigi and Rosasco, Lorenzo},
  title     = {On Fast Leverage Score Sampling and Optimal Learning},
  booktitle = {Advances in Neural Information Processing Systems},
  volume    = {31},
  pages     = {5677--5687},
  year      = {2018},
  publisher = {Curran Associates, Inc.},
  address   = {Montr\'eal, Canada}
}

@inproceedings{Liu2020SurrogateRFF,
  author    = {Liu, Fanghui and Huang, Xiaolin and Chen, Yudong and Yang, Jie and Suykens, Johan A. K.},
  title     = {Random {Fourier} Features via Fast Surrogate Leverage Weighted Sampling},
  booktitle = {Proceedings of the AAAI Conference on Artificial Intelligence},
  volume    = {34},
  number    = {04},
  pages     = {4844--4851},
  year      = {2020},
  doi       = {10.1609/aaai.v34i04.5920}
}

@article{Osipov2012ExplicitUpper,
  author  = {Osipov, Andrei},
  title   = {Certain Upper Bounds on the Eigenvalues Associated with Prolate Spheroidal Wave Functions},
  journal = {Appl. Comput. Harmon. Anal.},
  volume  = {35},
  number  = {2},
  pages   = {309--340},
  year    = {2013},
  doi     = {10.1016/j.acha.2013.03.002}
}

@article{Osipov2013CertainInequalities,
  author  = {Osipov, Andrei},
  title   = {Certain Inequalities Involving Prolate Spheroidal Wave Functions and Associated Quantities},
  journal = {Appl. Comput. Harmon. Anal.},
  volume  = {35},
  number  = {3},
  pages   = {359--393},
  year    = {2013},
  doi     = {10.1016/j.acha.2012.10.002}
}

@book{OsipovRokhlinXiao2013PSWF,
  author    = {Osipov, Andrei and Rokhlin, Vladimir and Xiao, Hong},
  title     = {Prolate Spheroidal Wave Functions of Order Zero: Mathematical Tools for Bandlimited Approximation},
  series    = {Applied Mathematical Sciences},
  volume    = {187},
  publisher = {Springer},
  address   = {New York},
  year      = {2013},
  doi       = {10.1007/978-1-4614-8259-8}
}

@misc{MingYu2026Spectral,
  author        = {Ming, Pingbing and Yu, Hao},
  title         = {Spectral Convergence of Random Feature Method in One Dimension},
  year          = {2025},
  eprint        = {2507.07371v2},
  archiveprefix = {arXiv},
  primaryclass  = {math.NA},
  note          = {arXiv:2507.07371v2},
  doi           = {10.48550/arXiv.2507.07371}
}

@misc{FuWang2026OptimalSobolev,
  author        = {Fu, Zhaohui and Wang, Yangshuai},
  title         = {Optimal Sobolev Approximation by Deterministic and Random Shallow Sigmoidal Networks},
  year          = {2026},
  eprint        = {2608.19797v2},
  archiveprefix = {arXiv},
  primaryclass  = {math.NA},
  note          = {arXiv:2608.19797v2},
  doi           = {10.48550/arXiv.2608.19797}
}

@article{MeirSharma1966,
  author  = {Meir, A. and Sharma, A.},
  title   = {Simultaneous Approximation of a Function and Its Derivatives},
  journal = {SIAM J. Numer. Anal.},
  volume  = {3},
  number  = {4},
  pages   = {553--563},
  year    = {1966},
  doi     = {10.1137/0703046}
}

@article{DaiXu2011,
  author  = {Dai, Feng and Xu, Yuan},
  title   = {Polynomial Approximation in {S}obolev Spaces on the Unit Sphere and the Unit Ball},
  journal = {J. Approx. Theory},
  volume  = {163},
  number  = {10},
  pages   = {1400--1418},
  year    = {2011},
  doi     = {10.1016/j.jat.2011.05.001}
}

@article{FeffermanHajdukRobinson2022,
  author  = {Fefferman, Charles L. and Hajduk, Karol W. and Robinson, James C.},
  title   = {Simultaneous Approximation in {L}ebesgue and {S}obolev Norms via Eigenspaces},
  journal = {Proc. Lond. Math. Soc.},
  volume  = {125},
  number  = {4},
  pages   = {759--777},
  year    = {2022},
  doi     = {10.1112/plms.12469}
}

@article{ChenChiEYang2022RFM,
  author  = {Chen, Jingrun and Chi, Xurong and E, Weinan and Yang, Zhouwang},
  title   = {Bridging Traditional and Machine Learning-Based Algorithms for Solving {PDE}s: The Random Feature Method},
  journal = {J. Mach. Learn.},
  volume  = {1},
  number  = {3},
  pages   = {268--298},
  year    = {2022},
  doi     = {10.4208/jml.220726}
}

@article{DongLi2021LocalELM,
  author  = {Dong, Suchuan and Li, Zongwei},
  title   = {Local Extreme Learning Machines and Domain Decomposition for Solving Linear and Nonlinear Partial Differential Equations},
  journal = {Comput. Methods Appl. Mech. Engrg.},
  volume  = {387},
  pages   = {114129},
  year    = {2021},
  doi     = {10.1016/j.cma.2021.114129}
}

@article{SunDongWang2024LRNNDG,
  author  = {Sun, Jingbo and Dong, Suchuan and Wang, Fei},
  title   = {Local Randomized Neural Networks with Discontinuous {Galerkin} Methods for Partial Differential Equations},
  journal = {J. Comput. Appl. Math.},
  volume  = {445},
  pages   = {115830},
  year    = {2024},
  doi     = {10.1016/j.cam.2024.115830}
}

@article{ChenELuo2023TimeRFM,
  author  = {Chen, Jingrun and E, Weinan and Luo, Yixin},
  title   = {The Random Feature Method for Time-Dependent Problems},
  journal = {East Asian J. Appl. Math.},
  volume  = {13},
  number  = {3},
  pages   = {435--463},
  year    = {2023},
  doi     = {10.4208/eajam.2023-065.050423}
}

@article{Gonon2023BlackScholes,
  author  = {Gonon, Lukas},
  title   = {Random Feature Neural Networks Learn {Black-Scholes} Type {PDE}s Without Curse of Dimensionality},
  journal = {J. Mach. Learn. Res.},
  volume  = {24},
  number  = {189},
  pages   = {1--51},
  year    = {2023}
}

@article{GononGrigoryevaOrtega2023,
  author  = {Gonon, Lukas and Grigoryeva, Lyudmila and Ortega, Juan-Pablo},
  title   = {Approximation Bounds for Random Neural Networks and Reservoir Systems},
  journal = {Ann. Appl. Probab.},
  volume  = {33},
  number  = {1},
  pages   = {28--69},
  year    = {2023},
  doi     = {10.1214/22-AAP1806}
}

@article{Fabiani2025RPNN,
  author  = {Fabiani, Gianluca},
  title   = {Random Projection Neural Networks of Best Approximation: Convergence Theory and Practical Applications},
  journal = {SIAM J. Math. Data Sci.},
  volume  = {7},
  number  = {2},
  pages   = {385--409},
  year    = {2025},
  doi     = {10.1137/24M1639890}
}

@article{ChenSchaeffer2024,
  author  = {Chen, Zhijun and Schaeffer, Hayden},
  title   = {Conditioning of Random {Fourier} Feature Matrices: Double Descent and Generalization Error},
  journal = {Inf. Inference},
  volume  = {13},
  number  = {2},
  pages   = {iaad054},
  year    = {2024},
  doi     = {10.1093/imaiai/iaad054}
}

@inproceedings{ChenSchaefferWard2022,
  author    = {Chen, Zhijun and Schaeffer, Hayden and Ward, Rachel},
  title     = {Concentration of Random Feature Matrices in High-Dimensions},
  booktitle = {Proceedings of Mathematical and Scientific Machine Learning},
  series    = {Proceedings of Machine Learning Research},
  volume    = {190},
  pages     = {287--302},
  year      = {2022},
  publisher = {PMLR},
  address   = {Beijing, China}
}

@article{Barnett2022Fourier,
  author  = {Barnett, Alex H.},
  title   = {How Exponentially Ill-Conditioned Are Contiguous Submatrices of the {Fourier} Matrix?},
  journal = {SIAM Rev.},
  volume  = {64},
  number  = {1},
  pages   = {105--131},
  year    = {2022},
  doi     = {10.1137/20M1336837}
}

@article{ChenESun2024Optimization,
  author  = {Chen, Jingrun and E, Weinan and Sun, Yifei},
  title   = {Optimization of Random Feature Method in the High-Precision Regime},
  journal = {Commun. Appl. Math. Comput.},
  volume  = {6},
  number  = {2},
  pages   = {1490--1517},
  year    = {2024},
  doi     = {10.1007/s42967-024-00389-8}
}

@article{ShangHeinleinMishraWang2025,
  author  = {Shang, Yong and Heinlein, Alexander and Mishra, Siddhartha and Wang, Fei},
  title   = {Overlapping {Schwarz} Preconditioners for Randomized Neural Networks with Domain Decomposition},
  journal = {Comput. Methods Appl. Mech. Engrg.},
  volume  = {442},
  pages   = {118011},
  year    = {2025},
  doi     = {10.1016/j.cma.2025.118011}
}

@article{TanChen2026Preconditioned,
  author  = {Tan, Longze and Chen, Jingrun},
  title   = {High-Precision Randomized Preconditioned Iterative Methods for the Random Feature Method},
  journal = {J. Comput. Appl. Math.},
  volume  = {481},
  pages   = {117255},
  year    = {2026},
  doi     = {10.1016/j.cam.2025.117255}
}

@article{vanBeekDoleanMoseley2026Filtering,
  author  = {{van Beek}, Jan Willem and Dolean, Victorita and Moseley, Ben},
  title   = {Local Feature Filtering for Scalable and Well-Conditioned Domain-Decomposed Random Feature Methods},
  journal = {Comput. Methods Appl. Mech. Engrg.},
  volume  = {449},
  pages   = {118583},
  year    = {2026},
  doi     = {10.1016/j.cma.2025.118583}
}

@article{SongChiYangChengChen2026Discontinuity,
  author  = {Song, Wentian and Chi, Xurong and Yang, Zhouwang and Cheng, Wan and Chen, Jingrun},
  title   = {Discontinuity-Capturing Random Feature Method for Interface Problems},
  journal = {Comput. Methods Appl. Mech. Engrg.},
  volume  = {453},
  pages   = {118841},
  year    = {2026},
  doi     = {10.1016/j.cma.2026.118841}
}

@misc{LinghuDongWang2026StructureAdaptive,
  author        = {Linghu, Jiale and Dong, Hao and Wang, Yangshuai},
  title         = {A Structure-Adaptive Random Feature Method for High-Dimensional Elliptic {PDE}s},
  year          = {2026},
  eprint        = {2607.19786},
  archiveprefix = {arXiv},
  primaryclass  = {math.NA},
  note          = {arXiv:2607.19786},
  doi           = {10.48550/arXiv.2607.19786}
}

@misc{ZhouFuWangFeng2026DiscreteTimeRFM,
  author        = {Zhou, Haoran and Fu, Zhaohui and Wang, Yangshuai and Feng, Xinlong},
  title         = {A Discrete-Time Random Feature Method for Nonlinear Evolution Equations with Implicit--Explicit {Runge--Kutta} Time Stepping},
  year          = {2026},
  eprint        = {2604.25502},
  archiveprefix = {arXiv},
  primaryclass  = {math.NA},
  note          = {arXiv:2604.25502},
  doi           = {10.48550/arXiv.2604.25502}
}

@article{huang2006extreme,
  author  = {Huang, Guang-Bin and Zhu, Qin-Yu and Siew, Chee-Kheong},
  title   = {Extreme Learning Machine: Theory and Applications},
  journal = {Neurocomputing},
  volume  = {70},
  number  = {1--3},
  pages   = {489--501},
  year    = {2006},
  doi     = {10.1016/j.neucom.2005.12.126}
}

@incollection{gallicchio2020deep,
  author    = {Gallicchio, Claudio and Scardapane, Simone},
  title     = {Deep Randomized Neural Networks},
  booktitle = {Recent Trends in Learning From Data},
  editor    = {Oneto, Luca and Navarin, Nicol{\`o} and Sperduti, Alessandro and Anguita, Davide},
  series    = {Studies in Computational Intelligence},
  volume    = {896},
  pages     = {43--68},
  publisher = {Springer International Publishing},
  address   = {Cham},
  year      = {2020},
  doi       = {10.1007/978-3-030-43883-8_3}
}

@article{huang2006universal,
  author  = {Huang, Guang-Bin and Chen, Lei and Siew, Chee-Kheong},
  title   = {Universal Approximation Using Incremental Constructive Feedforward Networks with Random Hidden Nodes},
  journal = {IEEE Trans. Neural Netw.},
  volume  = {17},
  number  = {4},
  pages   = {879--892},
  year    = {2006},
  doi     = {10.1109/TNN.2006.875977}
}

@misc{sun2019approximation,
  author        = {Sun, Yitong and Gilbert, Anna and Tewari, Ambuj},
  title         = {On the Approximation Properties of Random {ReLU} Features},
  year          = {2019},
  eprint        = {1810.04374},
  archiveprefix = {arXiv},
  primaryclass  = {stat.ML},
  note          = {arXiv:1810.04374v3},
  doi           = {10.48550/arXiv.1810.04374}
}

@inproceedings{Rahimi2008UniformAO,
  author    = {Rahimi, Ali and Recht, Benjamin},
  title     = {Uniform Approximation of Functions with Random Bases},
  booktitle = {2008 46th Annual Allerton Conference on Communication, Control, and Computing},
  pages     = {555--561},
  year      = {2008},
  doi       = {10.1109/ALLERTON.2008.4797607}
}

@article{Weinan2020TowardsAM,
  author  = {{E}, Weinan and Ma, Chao and Wu, Lei and Wojtowytsch, Stephan},
  title   = {Towards a Mathematical Understanding of Neural Network-Based Machine Learning: What We Know and What We Don't},
  journal = {CSIAM Trans. Appl. Math.},
  volume  = {1},
  number  = {4},
  pages   = {561--615},
  year    = {2020},
  doi     = {10.4208/csiam-am.SO-2020-0002}
}

@incollection{neal1996priors,
  author    = {Neal, Radford M.},
  title     = {Priors for Infinite Networks},
  booktitle = {Bayesian Learning for Neural Networks},
  series    = {Lecture Notes in Statistics},
  volume    = {118},
  pages     = {29--53},
  publisher = {Springer},
  address   = {New York},
  year      = {1996},
  doi       = {10.1007/978-1-4612-0745-0_2}
}

@inproceedings{williams1996computing,
  author    = {Williams, Christopher K. I.},
  title     = {Computing with Infinite Networks},
  booktitle = {Advances in Neural Information Processing Systems},
  volume    = {9},
  pages     = {295--301},
  publisher = {MIT Press},
  address   = {Cambridge, MA},
  year      = {1996}
}

@misc{deryck2025approximation,
  author        = {De Ryck, Tim and Mishra, Siddhartha and Shang, Yong and Wang, Fei},
  title         = {Approximation Theory and Applications of Randomized Neural Networks for Solving High-Dimensional {PDE}s},
  year          = {2025},
  eprint        = {2501.12145},
  archiveprefix = {arXiv},
  primaryclass  = {math.NA},
  note          = {arXiv:2501.12145},
  doi           = {10.48550/arXiv.2501.12145}
}

@article{neufeld2025full,
  author  = {Neufeld, Ariel and Schmocker, Philipp and Wu, Sizhou},
  title   = {Full Error Analysis of the Random Deep Splitting Method for Nonlinear Parabolic {PDE}s and {PIDE}s},
  journal = {Commun. Nonlinear Sci. Numer. Simul.},
  volume  = {143},
  pages   = {108556},
  year    = {2025},
  doi     = {10.1016/j.cnsns.2024.108556}
}

@article{nelsen2021random,
  author  = {Nelsen, Nicholas H. and Stuart, Andrew M.},
  title   = {The Random Feature Model for Input-Output Maps between {Banach} Spaces},
  journal = {SIAM J. Sci. Comput.},
  volume  = {43},
  number  = {5},
  pages   = {A3212--A3243},
  year    = {2021},
  doi     = {10.1137/20M133957X}
}

@article{nelsen2024operator,
  author  = {Nelsen, Nicholas H. and Stuart, Andrew M.},
  title   = {Operator Learning Using Random Features: A Tool for Scientific Computing},
  journal = {SIAM Rev.},
  volume  = {66},
  number  = {3},
  pages   = {535--571},
  year    = {2024},
  doi     = {10.1137/24M1648703}
}

@article{gui1986h1Dimension,
  author  = {Gui, Wenzhuang and Babu{\v{s}}ka, Ivo},
  title   = {The {$h$, $p$ and $h$-$p$} Versions of the Finite Element Method in One Dimension. Part II: The Error Analysis of the {$h$}- and {$h$-$p$} Versions},
  journal = {Numer. Math.},
  volume  = {49},
  number  = {6},
  pages   = {613--657},
  year    = {1986},
  doi     = {10.1007/BF01389734}
}

@article{guo1986hpversion,
  author  = {Guo, Benqi and Babu{\v{s}}ka, Ivo},
  title   = {The {$h$-$p$} Version of the Finite Element Method. Part 1: The Basic Approximation Results},
  journal = {Comput. Mech.},
  volume  = {1},
  number  = {1},
  pages   = {21--41},
  year    = {1986},
  doi     = {10.1007/BF00298636}
}

@book{melenk2004hp,
  author    = {Melenk, Jens M.},
  title     = {{$hp$}-Finite Element Methods for Singular Perturbations},
  series    = {Lecture Notes in Mathematics},
  volume    = {1796},
  publisher = {Springer},
  address   = {Berlin},
  year      = {2002},
  doi       = {10.1007/b84212}
}

@article{feischl2020exponential,
  author  = {Feischl, Michael and Schwab, Christoph},
  title   = {Exponential Convergence in {$H^1$} of {$hp$-FEM} for {Gevrey} Regularity with Isotropic Singularities},
  journal = {Numer. Math.},
  volume  = {144},
  number  = {2},
  pages   = {323--346},
  year    = {2020},
  doi     = {10.1007/s00211-019-01085-z}
}

@article{mhaskar1996neural,
  author  = {Mhaskar, Hrushikesh N.},
  title   = {Neural Networks for Optimal Approximation of Smooth and Analytic Functions},
  journal = {Neural Comput.},
  volume  = {8},
  number  = {1},
  pages   = {164--177},
  year    = {1996},
  doi     = {10.1162/neco.1996.8.1.164}
}

@article{weinan2018exponential,
  author  = {{E}, Weinan and Wang, Qingcan},
  title   = {Exponential Convergence of the Deep Neural Network Approximation for Analytic Functions},
  journal = {Sci. China Math.},
  volume  = {61},
  number  = {10},
  pages   = {1733--1740},
  year    = {2018},
  doi     = {10.1007/s11425-018-9387-x}
}

@article{montanelli2019deep,
  author  = {Montanelli, Hadrien and Yang, Haizhao and Du, Qiang},
  title   = {Deep {ReLU} Networks Overcome the Curse of Dimensionality for Generalized Bandlimited Functions},
  journal = {J. Comput. Math.},
  volume  = {39},
  number  = {6},
  pages   = {801--815},
  year    = {2021},
  doi     = {10.4208/jcm.2007-m2019-0239}
}

@article{DeRyck2021app_tanh,
  author  = {De Ryck, Tim and Lanthaler, Samuel and Mishra, Siddhartha},
  title   = {On the Approximation of Functions by tanh Neural Networks},
  journal = {Neural Netw.},
  volume  = {143},
  pages   = {732--750},
  year    = {2021},
  doi     = {10.1016/j.neunet.2021.08.015}
}
\endgroup

\appendix
\section{Proofs for approximation in interpolation spaces} \label{Section: proofs approximation interpolation}

This appendix proves the abstract estimate in Section~\ref{Section: Approximation error by interpolation}. We first transform the control of $\mc{E}(N,s,\theta,p)$ and $|\boldsymbol{\beta}^{*}|$ into bounding the norms of certain random operators. Denote $\tilde{\Phi}=\Sigma^{-\frac{\theta}{2}} \Phi$ and $\tilde{\Phi}^{*} = \Phi^{*} \Sigma^{-\frac{\theta}{2}}$, where the adjoint operator $\Phi^{*}:L^{2}(\mc{X},\mr{d}\rho)\to\mb{C}^{N}$,
\begin{equation*}
\left(\Phi^{*}f\right)_{j} = q(v_{j})^{-1 / 2} \la f, \phi(\cdot,v_{j})\ra_{L^{2}(\mc{X},\mr{d}\rho)}, \quad 1\le j\le N.
\end{equation*}
Since $\left\|u-\Phi \boldsymbol{\beta}\right\|_{\mc{H}^{\theta}} = \left\|\Sigma^{-\frac{\theta}{2}}(u-\Phi \boldsymbol{\beta})\right\|_{L^{2}}$, (\ref{ridge regression for bounding approximation error, interpolation}) has a unique solution from the usual normal equations and the matrix inversion lemma for operators\cite{Ogawa1988Pseudo}
\begin{equation} \label{solution of normal equations, interpolation}
\boldsymbol{\beta}^{*}=\left(\tilde{\Phi}^{*} \tilde{\Phi}+\lambda N I\right)^{-1} \tilde{\Phi}^{*}\Sigma^{-\frac{\theta}{2}}u
=N^{-1} \tilde{\Phi}^{*}\left(N^{-1} \tilde{\Phi} \tilde{\Phi}^{*}+\lambda I\right)^{-1} \Sigma^{-\frac{\theta}{2}}u.
\end{equation}
Denote the empirical integral operator $\hat{\Sigma}=N^{-1}\Phi\Phi^{*}$, which is characterized by
\begin{equation*}
\hat{\Sigma} = \frac{1}{N}\sum_{j = 1}^{N} q(v_{j})^{-1} \phi(\cdot,v_{j})\otimes_{L^{2}}\phi(\cdot,v_{j}).
\end{equation*}
Recall the expectation representation \eqref{eq: kernel integral operator as expectation section2} for $\Sigma$, which indicates $\Sigma = \mb{E}(\hat{\Sigma})$. Similarly, denote the empirical operator $\tilde{\Sigma}=N^{-1}\tilde{\Phi}\tilde{\Phi}^{*}=\Sigma^{-\frac{\theta}{2}}\hat{\Sigma}\Sigma^{-\frac{\theta}{2}}$.
\begin{lemma}[Resolvent reduction for ridge error and coefficient norm] \label{Lemma: bound uniform error mcE with random operator norms}
The uniform approximation error $\mc{E}$ and the solution vector $\boldsymbol{\beta}^{*}$ satisfy
\begin{equation*}
\begin{aligned}
& \mc{E}(N,s,\theta,p) \le \lambda\left\|\Sigma^{\frac{\theta-p}{2}}(\tilde{\Sigma}+\lambda I)^{-1} \Sigma^{\frac{s-\theta}{2}} \right\| , \\
& |\boldsymbol{\beta}^{*}| \le N^{-\frac{1}{2}} \left\|(\tilde{\Sigma}+\lambda I)^{-\frac{1}{2}}\Sigma^{\frac{s-\theta}{2}}\right\| .
\end{aligned}
\end{equation*}
\end{lemma}

\begin{proof}[Proof of Lemma \ref{Lemma: bound uniform error mcE with random operator norms}]

With the solution \eqref{solution of normal equations, interpolation}, we have
\[
\begin{aligned}
\Phi\boldsymbol{\beta}^* &=\Sigma^{\frac{\theta}{2}}\tilde{\Sigma}
(\tilde{\Sigma}+\lambda I)^{-1}\Sigma^{-\frac{\theta}{2}}u .
\end{aligned}
\]
Since \(u\in\mathcal H^s\) and \(s\ge\theta\), we have \(u=\Sigma^{\frac{\theta}{2}}\Sigma^{-\frac{\theta}{2}}u\). Hence, \[
\begin{aligned}
\Phi\boldsymbol{\beta}^*-u
&=
\Sigma^{\frac{\theta}{2}}
\left[\tilde{\Sigma}(\tilde{\Sigma}+\lambda I)^{-1}-I\right]
\Sigma^{-\frac{\theta}{2}}u \\
&=
-\lambda\Sigma^{\frac{\theta}{2}}
(\tilde{\Sigma}+\lambda I)^{-1}\Sigma^{-\frac{\theta}{2}}u .
\end{aligned}
\]
Taking the \(\mc{H}^{p}\)-norm gives
\begin{equation*}
\begin{aligned}
\left\|u-\Phi\boldsymbol{\beta}^{*}\right\|_{\mc{H}^{p}} & \le \lambda\left\|\Sigma^{\frac{\theta-p}{2}}(\tilde{\Sigma}+\lambda I)^{-1} \Sigma^{\frac{s-\theta}{2}} \right\| \left\|\Sigma^{-\frac{s}{2}} u\right\|_{L^{2}} . \\
\end{aligned}
\end{equation*}
This gives the first estimate since \(\left\|\Sigma^{-\frac{s}{2}}u\right\|_{L^{2}} =\|u\|_{\mc{H}^{s}}\le1\).

For the coefficient vector, we compute \begin{equation*}
\begin{aligned}
\left|\boldsymbol{\beta}^{*}\right|^{2} & = N^{-1} \left\langle\tilde{\Sigma}(\tilde{\Sigma}+\lambda I)^{-1}\Sigma^{-\frac{\theta}{2}} u,(\tilde{\Sigma}+\lambda I)^{-1}\Sigma^{-\frac{\theta}{2}} u\right\rangle \\
& \le N^{-1} \left(\left\|\tilde{\Sigma}^{\frac{1}{2}}(\tilde{\Sigma}+\lambda I)^{-\frac{1}{2}}\right\| \left\|(\tilde{\Sigma}+\lambda I)^{-\frac{1}{2}}\Sigma^{\frac{s-\theta}{2}}\right\| \left\|\Sigma^{-\frac{s}{2}} u\right\|_{L^{2}} \right)^{2} .\\
\end{aligned}
\end{equation*}
The second estimate follows from $\left\|\tilde{\Sigma}^{\frac{1}{2}}(\tilde{\Sigma}+\lambda I)^{-\frac{1}{2}}\right\|\le1$ and \(\|u\|_{\mc{H}^{s}}\le1\), which completes the proof.
\end{proof}

Next, we study the concentration properties of the empirical operators and prepare tools for the subsequent control of the operator norms in Lemma \ref{Lemma: bound uniform error mcE with random operator norms}.

\begin{lemma}[Preconditioned empirical concentration and resolvent comparison] \label{Lemma: empirical operator concentration interpolation}
Let $N\in\mb{N}_{+}$, $0<\gamma\le1-\theta\le1$, and $0<\delta<1$. If the admissibility condition \eqref{ineq: N > d_max(lambda) ln(d(lambda)/delta), interpolation} holds, then
\begin{equation} \label{ineq: norm estimate for r(Sigma)(hatSigma-Sigma)r(Sigma)}
\left\|r(\Sigma)(\hat{\Sigma}-\Sigma)r(\Sigma)\right\|\le\frac{15}{16}.
\end{equation}
with probability at least $1-\delta$. On the same event, for every $a\in[0,1/2]$,
\[
\left\|(\Sigma^{1-\theta}+\lambda I)^{a}(\tilde{\Sigma}+\lambda I)^{-a}\right\|\le16^{a}.
\]
\end{lemma}

To prove Lemma~\ref{Lemma: empirical operator concentration interpolation}, we use the Cordes inequality~\cite{Cordes1987} and Minsker's Bernstein inequality for self-adjoint operators~\cite{Minsker2017extensions}. Cordes' inequality yields the fractional resolvent comparison, whereas Minsker's inequality provides an intrinsic-dimension tail bound for $r(\Sigma)(\hat{\Sigma}-\Sigma)r(\Sigma)$, with the ambient dimension replaced by the effective rank of the variance operator.

\begin{lemma}[Cordes inequality~\cite{Cordes1987}] \label{Lemma: Cordes inequality}
Given two bounded, self-adjoint and positive operators $A$ and $B$, we have $\left\|A^{r} B^{r}\right\| \le \|A B\|^{r}$ for all $r \in[0,1]$.
\end{lemma}

\begin{proposition}[{Bernstein inequality for self-adjoint operators~\cite[Theorem 3.1 and Section 3.2]{Minsker2017extensions}}] \label{Proposition: Bernstein’s ineq for self-adjoint operators, Minsker2017}
Let $\mc H$ be a separable Hilbert space, and let $\{X_{i}\}_{i=1}^{n}$ be independent self-adjoint Hilbert--Schmidt random operators on $\mc H$ such that $\mb{E} X_{i}=0$ for $1\le i\le n$ and $\left\|\sum_{i=1}^{n} \mb{E} X_{i}^{2}\right\| \le \sigma^{2}$. Assume that $\left\|X_{i}\right\| \le U$ almost surely for all $1 \le i \le n$ and some positive $U \in \mb{R}$. Then, for any $t \ge\frac{1}{6}\left(U+\sqrt{U^{2}+36 \sigma^{2}}\right)$,
$$\mb{P}\left(\left\|\sum_{i=1}^{n} X_{i}\right\|>t\right) \le 14 \frac{\operatorname{tr}\left(\sum_{i=1}^{n} \mb{E} X_{i}^{2}\right)}{\sigma^{2}}
\exp\left(-\frac{t^{2} / 2}{\sigma^{2}+t U / 3}\right)$$.
\end{proposition}

\begin{proof}[Proof of Lemma \ref{Lemma: empirical operator concentration interpolation}]
First, we write $r(\Sigma)(\hat{\Sigma}-\Sigma)r(\Sigma)$ as the sum of independent random self-adjoint operators. Denote $\phi_{j} = r(\Sigma)\phi(\cdot,v_{j})$ and
\begin{equation*}
\begin{aligned}
X_{j} & =\frac{1}{Nq(v_{j})}\phi_{j} \otimes \phi_{j} - \frac{1}{N} \Sigma r^{2}(\Sigma) .
\end{aligned}
\end{equation*}
Then, $r(\Sigma)(\hat{\Sigma}-\Sigma)r(\Sigma) = \sum_{j=1}^{N} X_{j}$, $\mb{E} X_{j}=0$ and
\begin{equation*}
\begin{aligned}
\|X_{j}\| & \le \frac{1}{N} \max\left(\operatorname*{\tau\text{-}ess\,sup}_{v\in\mc V} \frac{\ell_{\lambda}(v;\theta,\gamma)}{q(v)}, \|\Sigma r^{2}(\Sigma)\| \right)  = \frac{d_{\max}(q,\lambda)}{N},
\end{aligned}
\end{equation*}
where we used $\|\Sigma r^{2}(\Sigma)\| \le \operatorname{tr}\left(\Sigma r^{2}(\Sigma)\right)\le d_{\max}(q,\lambda)$ in the equality. To control $\mb{E} X_{j}^{2}$, we observe
\begin{equation*}
\begin{aligned}
\mb{E} X_{j}^{2} & \preceq \frac{1}{N^{2}}\mb{E} \left[\frac{1}{q(v_{j})^{2}} \|\phi_{j}\|_{L^{2}}^{2}\phi_{j} \otimes \phi_{j}\right] \\
& \preceq \frac{d_{\max}(q,\lambda)}{N^{2}}\mb{E} \left[\frac{1}{q(v_{j})}\phi_{j} \otimes \phi_{j}\right]  = \frac{d_{\max}(q,\lambda)}{N^{2}}\Sigma r^{2}(\Sigma) .
\end{aligned}
\end{equation*}
Using $0\preceq \Sigma r^{2}(\Sigma) \preceq I$, we obtain
\begin{equation*}
\begin{aligned}
& \left\|\sum_{j=1}^{N} \mb{E} X_{j}^{2}\right\| \le  \frac{d_{\max}(q,\lambda)}{N} \|\Sigma r^{2}(\Sigma)\|\le  \frac{d_{\max}(q,\lambda)}{N} , \\
& \operatorname{tr}\left(\sum_{j=1}^{N} \mb{E} X_{j}^{2}\right) \le  \frac{d_{\max}(q,\lambda)}{N} \operatorname{tr}(\Sigma r^{2}(\Sigma)) \le  \frac{d_{\max}(q,\lambda)d(\lambda)}{N} .
\end{aligned}
\end{equation*}
To apply Proposition \ref{Proposition: Bernstein’s ineq for self-adjoint operators, Minsker2017}, we set $t = 15/16$ and check
\begin{equation*}
\begin{aligned}
\frac{1}{6}\left(\frac{d_{\max}(q,\lambda)}{N} + \sqrt{\frac{d_{\max}^{2}(q,\lambda)}{N^{2}} + 36 \frac{d_{\max}(q,\lambda)}{N}}\right) < \frac{2}{3} \le t,
\end{aligned}
\end{equation*}
where we used $d_{\max}(q,\lambda)/N\le 1/3$. Then, by Proposition \ref{Proposition: Bernstein’s ineq for self-adjoint operators, Minsker2017},
\begin{equation*}
\begin{aligned}
\mb{P}\left(\left\|r(\Sigma)(\hat{\Sigma}-\Sigma)r(\Sigma)\right\|>\frac{15}{16}\right) \le 14 d(\lambda) \exp\left(-\frac{75 N}{224 d_{\max}(q,\lambda)}\right) \le \delta,
\end{aligned}
\end{equation*}
which gives the estimate (\ref{ineq: norm estimate for r(Sigma)(hatSigma-Sigma)r(Sigma)}). For the second estimate, since (\ref{ineq: norm estimate for r(Sigma)(hatSigma-Sigma)r(Sigma)}) implies $\Sigma - \hat{\Sigma} \preceq \frac{15}{16} r(\Sigma)^{-2}$, we have
\begin{equation*}
\begin{aligned}
\hat{\Sigma} + \lambda \Sigma^{\theta} & \succeq \frac{1}{16}\left(\Sigma+\lambda \Sigma^{\theta}\right)+\frac{15}{16}\left[\Sigma + \lambda \Sigma^{\theta} - r(\Sigma)^{-2}\right] \\
& \succeq \frac{1}{16}\left(\Sigma+\lambda \Sigma^{\theta}\right) ,
\end{aligned}
\end{equation*}
where the last line follows from the fact that, for all $x\ge0$,
\begin{equation*}
\begin{aligned}
\left(x^{1-\theta} + \lambda\right)^{\frac{\gamma}{1-\theta}} x^{1-\gamma} \le \left(x^{1-\theta} + \lambda\right) x^{\theta} , \end{aligned}
\end{equation*}
and $x + \lambda x^{\theta} - r(x)^{-2} \ge 0$. Hence, $\tilde{\Sigma} + \lambda I\succeq (\Sigma^{1-\theta}+\lambda I)/16$ and
\begin{equation*}
\begin{aligned}
\left\|(\Sigma^{1-\theta} + \lambda I)^{\frac{1}{2}}(\tilde{\Sigma} + \lambda I)^{-\frac{1}{2}}\right\|^{2} & = \left\|(\tilde{\Sigma} + \lambda I)^{-\frac{1}{2}}(\Sigma^{1-\theta} + \lambda I)(\tilde{\Sigma} + \lambda I)^{-\frac{1}{2}}\right\|  \le 16 .
\end{aligned}
\end{equation*}
Therefore, we have proved the second estimate for $a=1/2$.
Applying Lemma \ref{Lemma: Cordes inequality} with $A=(\Sigma^{1-\theta}+\lambda I)^{\frac{1}{2}}$, $B=(\tilde{\Sigma}+\lambda I)^{-\frac{1}{2}}$ and $r=2a$, we obtain
\[
\left\|(\Sigma^{1-\theta}+\lambda I)^{a}(\tilde{\Sigma}+\lambda I)^{-a}\right\| \le \left\|(\Sigma^{1-\theta}+\lambda I)^{\frac{1}{2}}(\tilde{\Sigma}+\lambda
I)^{-\frac{1}{2}}\right\|^{2a} \le 16^{a}
\]
for every $a\in[0,1/2]$, which proves the second estimate and completes the proof.
\end{proof}

Now, we are ready to prove Theorem \ref{Theorem: abstract error estimate for regression in interpolation spaces}.

\begin{proof}[Proof of Theorem \ref{Theorem: abstract error estimate for regression in interpolation spaces}]
The concentration estimate is applied below at an arbitrary admissible penalty level $\eta>0$ satisfying \eqref{ineq: N > d_max(lambda) ln(d(lambda)/delta), interpolation} with $\lambda=\eta$. The quantity $\varsigma_N$ is used only as the associated critical penalty level. When the infimum defining $\varsigma_N$ is not attained, estimates involving $\varsigma_N$ are obtained by applying the argument at admissible levels above $\varsigma_N$ and passing to the limit.

(1) By Lemma \ref{Lemma: bound uniform error mcE with random operator norms}, since $\tilde{\Sigma}$ is a positive self-adjoint operator and $s-\bar{p}\le 1-(2\theta-1)$, we decompose
\begin{align}
\mc{E}(N,s,\theta,p) & \le \lambda \left\|\Sigma^{\frac{\bar{p}-p}{2}}\right\| \left\|\Sigma^{\frac{\theta-\bar{p}}{2}}(\tilde{\Sigma}+\lambda I)^{-\frac{\theta-\bar{p}}{2(1-\theta)}}\right\| \left\| (\tilde{\Sigma}+\lambda I)^{-1+\frac{s-\bar{p}}{2(1-\theta)}} \right\| \left\|(\tilde{\Sigma}+\lambda I)^{-\frac{s-\theta}{2(1-\theta)}}\Sigma^{\frac{s-\theta}{2}}\right\| \notag\\
& \le \lambda^{\frac{s-\bar{p}}{2(1-\theta)}}\left\|\Sigma\right\|^{\frac{\bar{p}-p}{2}} \left\|\Sigma^{\frac{\theta-\bar{p}}{2}}(\tilde{\Sigma}+\lambda I)^{-\frac{\theta-\bar{p}}{2(1-\theta)}}\right\| \left\|(\tilde{\Sigma}+\lambda I)^{-\frac{s-\theta}{2(1-\theta)}}\Sigma^{\frac{s-\theta}{2}} \right\| .  \label{ineq: first decomp for mcE when theta>p}
\end{align}
Note that $\bar{p}\le\theta\le s$. Since $\theta<1$, the choice $\gamma=1-\theta$ is positive. Thus, for any admissible penalty level $\eta$ for $(N,\delta,q,\theta,1-\theta)$, Lemma \ref{Lemma: empirical operator concentration interpolation} with $\gamma=1-\theta$ and $a=1/2$ gives, with probability at least $1-\delta$,
\begin{equation*}
\begin{aligned}
\left\|\Sigma^{\frac{1-\theta}{2}}(\tilde{\Sigma}+\lambda I)^{-\frac{1}{2}}\right\|
& \le \left\|\Sigma^{\frac{1-\theta}{2}}(\Sigma^{1-\theta}+\eta I)^{-\frac{1}{2}}\right\| \left\|(\Sigma^{1-\theta}+\eta I)^{\frac{1}{2}} (\tilde{\Sigma}+\eta I)^{-\frac{1}{2}}\right\| \left\|(\tilde{\Sigma}+\eta I)^{\frac{1}{2}}(\tilde{\Sigma}+\lambda I)^{-\frac{1}{2}}\right\| \\
& \le 4 \max(1,\sqrt{\eta/\lambda}).
\end{aligned}
\end{equation*}
Similarly, we obtain $\left\|(\tilde{\Sigma}+\lambda I)^{-\frac{1}{2}}\Sigma^{\frac{1-\theta}{2}}\right\| \le 4 \max(1,\sqrt{\eta/\lambda})$. It then follows from Lemma \ref{Lemma: Cordes inequality} that
\begin{equation*}
\begin{aligned}
& \left\|\Sigma^{\frac{\theta-\bar{p}}{2}}(\tilde{\Sigma}+\lambda I)^{-\frac{\theta-\bar{p}}{2(1-\theta)}}\right\| \le (16 \max(1,\eta/\lambda))^{\frac{\theta-\bar{p}}{2(1-\theta)}}, \\
& \left\|(\tilde{\Sigma}+\lambda I)^{-\frac{s-\theta}{2(1-\theta)}}\Sigma^{\frac{s-\theta}{2}} \right\| \le (16 \max(1,\eta/\lambda))^{\frac{s-\theta}{2(1-\theta)}} .
\end{aligned}
\end{equation*}
Substituting the above bounds into (\ref{ineq: first decomp for mcE when theta>p}) yields the estimate for $\mc{E}$ in (1) with $\varsigma_{N}$ replaced by $\eta$. For every $\varepsilon>0$, choose an admissible $\eta\le\varsigma_N+\varepsilon$. Since the right-hand side is continuous and nondecreasing in $\eta$, the same estimate holds with $\varsigma_N+\varepsilon$ in place of $\eta$ with probability at least $1-\delta$. Letting $\varepsilon\downarrow0$ and using continuity of probability from above gives the stated estimate. Additionally, by Lemma \ref{Lemma: bound uniform error mcE with random operator norms} and the second inequality above,
\begin{equation*}
\begin{aligned}
|\boldsymbol{\beta}^{*}| & \le N^{-\frac{1}{2}} \left\|(\tilde{\Sigma}+\lambda I)^{\frac{s-1}{2(1-\theta)}}\right\| \left\|(\tilde{\Sigma}+\lambda I)^{-\frac{s-\theta}{2(1-\theta)}}\Sigma^{\frac{s-\theta}{2}}\right\| \\
& \le N^{-\frac{1}{2}} \lambda^{\frac{s-1}{2(1-\theta)}}  (16 \max(1,\eta/\lambda))^{\frac{s-\theta}{2(1-\theta)}} ,
\end{aligned}
\end{equation*}
The same limiting argument with $\eta\le\varsigma_N+\varepsilon$ proves the estimate for $|\boldsymbol{\beta}^{*}|$. (2) By Lemma \ref{Lemma: bound uniform error mcE with random operator norms} and $A^{-1}-B^{-1}=A^{-1}(B-A)B^{-1}$, we obtain
\begin{align}
\mc{E}(N,s,\theta,p) & \le \lambda\left\|\Sigma^{\frac{\theta-p}{2}}(\Sigma^{1-\theta}+\lambda I)^{-1} \Sigma^{\frac{s-\theta}{2}} \right\| + \lambda\left\|\Sigma^{\frac{\theta-p}{2}}\left[(\Sigma^{1-\theta}+\lambda I)^{-1} -(\tilde{\Sigma}+\lambda I)^{-1}\right] \Sigma^{\frac{s-\theta}{2}} \right\| \notag \\
& \le \lambda\left\|\Sigma^{\frac{s-p}{2}}(\Sigma^{1-\theta}+\lambda I)^{-1}\right\| + \lambda\left\|\Sigma^{\frac{\theta-p}{2}}(\Sigma^{1-\theta}+\lambda I)^{-1}(\tilde{\Sigma}-\Sigma^{1-\theta})(\tilde{\Sigma}+\lambda I)^{-1} \Sigma^{\frac{s-\theta}{2}} \right\| , \label{ineq: first decomp for mcE when theta<p}
\end{align}
where the first term is bounded by $$\lambda\left\|\Sigma^{\frac{s-p}{2}}(\Sigma^{1-\theta}+\lambda I)^{-\frac{s-p}{2(1-\theta)}}\right\|
\left\|(\Sigma^{1-\theta}+\lambda I)^{-1+\frac{s-p}{2(1-\theta)}}\right\| \le\lambda^{\frac{s-p}{2(1-\theta)}} . $$ Since $\theta\le p<1$, we have
$0<1-p\le1-\theta$, so the concentration estimate applies with $\gamma=1-p$. Let $r(x)$ be defined in (\ref{def: r(x) used in concentration}) with this choice of $\gamma$. Then, the second term in (\ref{ineq: first decomp for mcE when theta<p}) is bounded by
\begin{equation*}
\begin{aligned}
&  \lambda\left\|(\Sigma^{1-\theta}+\lambda I)^{\frac{2\theta-p-1}{2(1-\theta)}} \right\| \left\|r(\Sigma)(\hat{\Sigma}-\Sigma)r(\Sigma) \right\|  \left\| \Sigma^{\frac{p-\theta}{2}}(\Sigma^{1-\theta}+\lambda I)^{\frac{1-p}{2(1-\theta)}} (\tilde{\Sigma}+\lambda I)^{-1} \Sigma^{\frac{s-\theta}{2}} \right\| \\
& \le \frac{15}{16}\lambda^{\frac{1-p}{2(1-\theta)}}   \left\|\Sigma^{\frac{p-\theta}{2}} (\Sigma^{1-\theta}+\lambda I)^{\frac{\theta-p}{2(1-\theta)}}\right\| \left\|(\Sigma^{1-\theta}+\lambda I)^{\frac{1}{2}} (\tilde{\Sigma}+\lambda I)^{-1} \Sigma^{\frac{s-\theta}{2}} \right\| \\
& \le \frac{15}{16}\lambda^{\frac{1-p}{2(1-\theta)}}   \left\|(\Sigma^{1-\theta}+\lambda I)^{\frac{1}{2}} (\tilde{\Sigma}+\lambda I)^{-\frac{1}{2}}\right\| \left\| (\tilde{\Sigma}+\lambda I)^{\frac{s-1}{2(1-\theta)}} \right\| \left\| (\tilde{\Sigma}+\lambda I)^{\frac{\theta-s}{2(1-\theta)}}(\Sigma^{1-\theta}+\lambda I)^{\frac{s-\theta}{2(1-\theta)}} \right\| \\
& \le 15\lambda^{\frac{s-p}{2(1-\theta)}} ,
\end{aligned}
\end{equation*}
where we used the first claim of Lemma \ref{Lemma: empirical operator concentration interpolation} in the first inequality and the second claim of Lemma \ref{Lemma: empirical operator concentration interpolation} in the last inequality. The bound for $|\boldsymbol{\beta}^{*}|$ follows in the same way as in (1), using the admissibility of $\lambda$ with $\gamma=1-p$. This completes the proof.
\end{proof}

\section{Proofs for uniform approximation by random Fourier features}
\label{Section: proofs uniform approximation random Fourier features}

This appendix supplies the Fourier-specific ingredients used to derive the uniform approximation results in Section~\ref{Section: Estimate of d(lambda) and the convergence rate, interpolation}. The argument has three layers. We first show that the relevant kernel integral operators are injective, so the positive spectral subspace is all of $L^2(\Omega)$. Subsection~\ref{Section: effective dimension estimates for translation invariant kernels} then bounds the effective dimension in the four frequency-decay regimes, while Subsection~\ref{Subsection: Fourier descriptions and interpolation embeddings} identifies the target and error spaces with the appropriate kernel interpolation spaces. These ingredients are assembled in Subsection~\ref{Section: proof uniform approximation regularity balls} to prove the regularity-adapted approximation theorem. The remaining subsections adapt the same strategy to growing uniform frequency supports and prove Proposition~\ref{Proposition: leverage approximation growing bandwidth target} and Theorem~\ref{Theorem: optimal growing bandwidth leverage rates}.

\paragraph{Injectivity of the kernel integral operator.}
The following proof justifies working on the full space $L^2(\Omega)$ for all frequency measures used in Section~\ref{Section: Estimate of d(lambda) and the convergence rate, interpolation}.

\begin{proof}[Proof of Lemma \ref{Lemma: injectivity of translation-invariant kernel integral operators}]
Since $\tau_s$ is the symmetrized measure of $\tau$, the kernels defined by \eqref{eq: translation invariant kernel for random Fourier feature} using $\tau$ and $\tau_s$ coincide. Consequently, we have
$$ \begin{aligned}
k(x,y) & =\int_{\mb{R}^{d}}e^{\mathrm{i} w^{\top}(x-y)}\,\mr{d}\tau_s(w).
\end{aligned} $$
For any $f\in L^{2}(\Omega)$, let $F:=f\mathbf{1}_{\Omega}$ denote its zero extension outside $\Omega$. The boundedness of $\Omega$ ensures that $F\in L^{1}(\mb{R}^{d})\cap L^{2}(\mb{R}^{d})$. Applying Fubini's theorem yields
\[
\begin{aligned}
\langle \Sigma f,f\rangle_{L^{2}(\Omega)}
&=
\int_{\Omega}\int_{\Omega}
k(x,y)f(y)\overline{f(x)}\,\mr{d}y\,\mr{d}x  \\
&=
\int_{\mb{R}^{d}}
\left(
\int_{\Omega} f(y)e^{-i w^{\top}y}\,\mr{d}y
\right)
\overline{
\left(
\int_{\Omega} f(x)e^{-i w^{\top}x}\,\mr{d}x
\right)}
\,\mr{d}\tau_s(w) \\
&=
\int_{\mb{R}^{d}}
\left|\int_{\Omega} f(x)e^{-i w^{\top}x}\,\mr{d}x\right|^{2}
\,\mr{d}\tau_s(w).
\end{aligned}
\]
Now assume $\Sigma f=0$. Then $\langle \Sigma f,f\rangle_{L^{2}(\Omega)}=0$, and hence
\[
\int_{\Omega} f(x)e^{-i w^{\top}x}\,\mr{d}x=0, \quad \text{for } \tau_s\text{-a.e. } w .
\]
The function $w\mapsto \int_{\Omega} f(x)e^{-i w^{\top}x}\,\mr{d}x$ is continuous on $\mb{R}^{d}$. Therefore it vanishes on $\operatorname{supp}\tau_s$. Since $\operatorname{supp}\tau_s$ contains a nonempty open set, the Fourier transform of $F$ vanishes on a nonempty open subset of $\mb{R}^{d}$.

Since \(F\) has compact support, the Paley--Wiener--Schwartz theorem~\cite[Theorem 7.3.1]{Hormander2003} implies that its Fourier transform extends to an entire function on \(\mb{C}^{d}\). By the identity theorem for real-analytic functions, the Fourier transform of $F$ vanishes identically on $\mb{R}^{d}$. The injectivity of the Fourier transform then implies $F=0$ almost everywhere. Hence $f=0$ in $L^{2}(\Omega)$. Thus $\ker\Sigma=\{0\}$, which completes the proof.
\end{proof}

With the spectral subspace identified, we next estimate its effective dimension and then relate the resulting interpolation scale to the target and error spaces.

\subsection{Effective-dimension estimates for translation-invariant kernels}
\label{Section: effective dimension estimates for translation invariant kernels}

This subsection proves the spectral estimates used in Proposition~\ref{Proposition: upper bound for d(lambda;theta,gamma)} and hence in Theorem~\ref{Theorem: interpolation improved convergence rate of RFM}. We first state a Fourier-multiplier comparison lemma and the resulting four-regime effective-dimension bound.
We then prove the comparison lemma and apply it, together with the eigenvalue asymptotics of \cite{widom1963asymptotic1,widom1964asymptotic2}, in each decay regime.

We use the complete elliptic integral of the first kind
\[
\boldsymbol{K}(r):=\int_{0}^{\pi/2}\left(1-r^{2}\sin^{2}\theta\right)^{-1/2}\mr{d}\theta, \qquad 0\le r<1,
\]
and set $\Upsilon(x):=x^{x}$ for $x>0$.

\paragraph{Comparison and effective-dimension statements.}
To formulate the operator comparisons used below, consider an integral operator
\begin{equation} \label{eq: form of integral operators in Widom1963}
Tf(x) = \int_{\mb{R}^{d}} 1_{\Omega}(x) k(x-y) 1_{\Omega}(y) f(y) \mr{d} y ,
\end{equation}
where $k\in L^{1}(\mb{R}^{d})$ has a nonnegative Fourier transform
\[
K(\xi)=\int_{\mb{R}^{d}} k(x)e^{-\imath \xi^{\top}x}\mr{d}x.
\]
The following monotonicity lemma allows us to compare the eigenvalues of translation-invariant integral operators through their Fourier multipliers.

\begin{lemma} \label{Lemma: monotonicity of certain integral operators}
Let $\{T_j\}_{j=1}^{2}$ be integral operators of the form \eqref{eq: form of integral operators in Widom1963} with kernels $\{k_j\}_{j=1}^{2}$. Assume that
\[
k_j(x)=(2\pi)^{-d}\int_{\mb{R}^{d}}K_j(\xi)e^{\imath \xi^{\top}x}\mr{d}\xi,
\]
and that $0\le K_2(\xi)\le K_1(\xi)$ almost everywhere. Then $0\preceq T_2\preceq T_1$.
\end{lemma}

\begin{proposition} \label{Proposition: upper bound for d(lambda;theta,gamma)}
Let $\Sigma$ be the integral operator on $L^{2}(\Omega)$ associated with the translation-invariant kernel in \eqref{eq: translation invariant kernel for random Fourier feature}. Let $0<\lambda$, $0<\gamma\le1-\theta\le1$, and
\[
d(\lambda;\theta,\gamma) :=\operatorname{tr}\left[\left(\Sigma^{1-\theta}+\lambda I\right)^{-\frac{\gamma}{1-\theta}}\Sigma^{\gamma}\right].
\]
(1) Let $s\gamma>d/2$ and suppose that, for some $0<c_{\tau}\le C_{\tau}<\infty$,
\[
c_{\tau}(1+|w|^{2})^{-s} \le \frac{\mr{d}\tau}{\mr{d}w}(w) \le C_{\tau}(1+|w|^{2})^{-s} \quad\text{for a.e. }w\in\mb{R}^{d}.
\]
Then, there exist constants $c_{1}>0$ and $M_{0}>0$, depending on $d,s,\theta,\gamma,\Omega,c_{\tau}$, and $C_{\tau}$, such that for $\lambda=c_{1}(|\Omega|/M)^{2s(1-\theta)/d}$, $d(\lambda)\le M$ whenever $M\ge M_{0}$. \\ (2) Let $R,\kappa>0$, $s \ge 1$, $\Omega = (-R, R)^{d}$ and $\mr{d}\tau \propto e^{- 2\kappa |w|^{1/s}} \mr{d} w$. Then, for any
\begin{equation} \label{eq: def of kappa_s in subexponential and exponential convergence rate, interpolation}
0 < \kappa_{s} <
\begin{cases}
\frac{(1-\theta)\kappa}{2^{2/(sd)}sd^{2-1/(2s)}R^{1/s}},     & \text{ for } s>1 , \\
\frac{(1-\theta) \boldsymbol{K}(\operatorname{sech}  \pi R \sqrt{d}\kappa^{-1})}{2^{1/d}d \boldsymbol{K}(\tanh  \pi R \sqrt{d}\kappa^{-1})},   & \text{ for } s=1 , \\
\end{cases}
\end{equation}
there exists a constant $c_{2}$ depending on $\kappa,d,s,\gamma,R$ such that, for
\[
\lambda=c_{2}\exp\bigl(-\kappa_sM^{1/(sd)}\bigr),
\]
one has $d(\lambda)\le M$ for all sufficiently large $M$. \\ (3) Let $R,\kappa>0$, $s>1$, $\Omega=(-R,R)^d$, and $\mr d\tau\propto e^{-2\kappa|w|^s}\mr dw$. For any $0<\bar\kappa(s)<2-2/s$, there exists a constant $c_3$ depending on $\kappa,\bar\kappa,d,s,\gamma,R$ such that, for
\[
\lambda=c_3\Upsilon\left[\frac{(M/2)^{1/d}}
{e(\bar\kappa^{-1}\gamma^{-1}+3/2)}\right]^{-(1-\theta)\bar\kappa d},
\]
one has $d(\lambda)\le M$ for all sufficiently large $M$. \\ (4) Let $R,S>0$, $\Omega=(-R,R)^d$, and $\mr d\tau\propto1_{(-S,S)^d}(w)\mr dw$. For any $0<\bar\kappa<2$, there exists a constant $C_{\operatorname{Ban}}$ depending only on $SR$ such that, for
\[
\lambda=(\pi/S)^{(1-\theta)d}C_{\operatorname{Ban}}^{1-\theta} \Upsilon\left[\frac{(M/2)^{1/d}}
{e(\bar\kappa^{-1}\gamma^{-1}+3/2)}\right]^{-(1-\theta)\bar\kappa d},
\]
one has $d(\lambda)\le M$ for all sufficiently large $M$.
\end{proposition}

\paragraph{Proofs of the spectral estimates.}
If $\{\lambda_j\}_{j\ge1}$ are the positive eigenvalues of $\Sigma$, then
\[
d(\lambda;\theta,\gamma) =\operatorname{tr}\left[\left(\Sigma^{1-\theta}+\lambda I\right)^{-\frac{\gamma}{1-\theta}}\Sigma^\gamma\right]
=\sum_{j=1}^{\infty}\frac{\lambda_j^\gamma}{\left(\lambda_j^{1-\theta}+\lambda\right)^{\frac{\gamma}{1-\theta}}}.
\]
The eigenvalue decay estimates below are based on \cite{widom1963asymptotic1,widom1964asymptotic2}.

\begin{proof}[Proof of Lemma \ref{Lemma: monotonicity of certain integral operators}]
For every $f\in L^{2}(\mb{R}^{d})$, a direct calculation gives
\[
\begin{aligned}
\langle T_jf,f\rangle_{L^{2}(\mb{R}^{d})}
&=\int_{\mb{R}^{d}}\int_{\mb{R}^{d}}1_{\Omega}(x)k_j(x-y)1_{\Omega}(y)f(y)\overline{f(x)}\mr{d}y\mr{d}x\\
&=(2\pi)^{-d}\int_{\mb{R}^{d}}K_j(\xi)
\left|\int_{\mb{R}^{d}}1_{\Omega}(x)f(x)e^{-\imath \xi^{\top}x}\mr{d}x\right|^{2}\mr{d}\xi.
\end{aligned}
\]
Hence $0\le\langle T_2f,f\rangle_{L^{2}(\mb{R}^{d})}\le\langle T_1f,f\rangle_{L^{2}(\mb{R}^{d})}$ for every $f\in L^{2}(\mb{R}^{d})$, which proves the claim and completes the proof.
\end{proof}

\begin{proof}[Proof of Proposition \ref{Proposition: upper bound for d(lambda;theta,gamma)}]
The inequality $(a^{t}+b^{t})/2\le(a+b)^{t}\le a^{t}+b^{t}$ for $a,b>0$ and $0\le t\le1$ gives
\begin{equation} \label{ineq: d(lambda;theta,gamma) bounded by eigenvalues of Sigma}
\sum_{j=1}^{\infty}\frac{\lambda_{j}^{\gamma}}{\lambda_{j}^{\gamma} + \lambda^{\frac{\gamma}{1-\theta}}}
\le d(\lambda;\theta,\gamma)
\le \sum_{j=1}^{\infty}\frac{2\lambda_{j}^{\gamma}}{\lambda_{j}^{\gamma} + \lambda^{\frac{\gamma}{1-\theta}}} .
\end{equation}
It remains to estimate the eigenvalues and the resulting series in the four cases of the proposition.

\emph{Case (1): polynomial decay.} Let $\Sigma_s$ be the integral operator corresponding to the Fourier multiplier $(1+|w|^2)^{-s}$. The upper comparison assumption on $\mr d\tau$ and Lemma~\ref{Lemma: monotonicity of certain integral operators} give
\[
0\preceq\Sigma\preceq C_\tau\Sigma_s, \qquad \lambda_j(\Sigma)\le C_\tau\lambda_j(\Sigma_s).
\]
By \cite[Theorem II]{widom1963asymptotic1}, the nonincreasing function equimeasurable with $1_{\Omega}(x)(1+|w|^{2})^{-s}$ on $\mb{R}^{d}\times\mb{R}^{d}$ is
\begin{equation*}
\phi_{0}(t) = \left[\left(\frac{t}{|\Omega||B_{1}(0)|}\right)^{2/d} +1\right]^{-s} ,
\end{equation*}
and hence $\lambda_j(\Sigma_s)\lesssim_{d,s,\Omega}j^{-2s/d}$ for all sufficiently large $j$. Enlarging the constant to cover the finitely many remaining eigenvalues yields
\[
\lambda_j(\Sigma)\lesssim_{d,s,\Omega,C_\tau}j^{-2s/d} \qquad (j\in\mb N_+).
\]
Substituting this estimate into \eqref{ineq: d(lambda;theta,gamma) bounded by eigenvalues of Sigma} and applying Lemma \ref{Lemma: control sum series, Sobolev} yield
\[
d(\lambda) \lesssim_{d,s,\theta,\gamma,\Omega,C_\tau} \lambda^{-\frac{d}{2s(1-\theta)}},
\]
which implies the desired bound after choosing $c_1$ and $M_0$ sufficiently large in terms of the parameters specified in the proposition.

\emph{Case (2): subexponential and exponential decay.} For any $R > 0$, the eigenvalues of the operator in \eqref{eq: form of integral operators in Widom1963} are unchanged if $1_{\Omega}(\cdot)$ is replaced by $1_{\Omega}(R\cdot)$ and $k(\cdot)$ by $R^{d} k(R\cdot)$. We only need to consider the operator in \eqref{eq: form of integral operators in Widom1963} with $\Omega = (-1,1)^{d}$ and kernel
\begin{equation*}
\tilde{k}(x) = C_{\kappa, d, s}^{-1} \int_{\mb{R}^{d}}  e^{-2\kappa R^{-1/s} |w|^{1/s}+\imath w^{\top}x}  \mr{d} w .
\end{equation*}
It follows from H\"older's inequality that
\[
d^{1/(2s)-1}\sum_{p=1}^{d}|w_{p}|^{1/s} \le |w|^{1/s}.
\]
To control $\lambda_{j}(\Sigma)$, Lemma \ref{Lemma: monotonicity of certain integral operators} reduces the problem to the operator $T_{\operatorname{Gev}}$ of the form \eqref{eq: form of integral operators in Widom1963} with $\Omega = (-1,1)^{d}$ and kernel
\begin{equation*}
\begin{aligned}
k_{\operatorname{Gev}}(x)
& = C_{\kappa, d, s}^{-1}   \prod_{p=1}^{d} \int_{\mb{R}}   \exp\left(-\kappa R^{-1/s} d^{1/(2s)-1}|w_{p}|^{1/s} + \imath w_{p}x_{p}\right)  \mr{d} w_{p}.
\end{aligned}
\end{equation*}
Denote by $T_{\operatorname{Gev}}^{(p)}$ the one-dimensional integral operators of the form \eqref{eq: form of integral operators in Widom1963} on $(-1,1)$ with kernel
\[
k_{\operatorname{Gev}}^{(p)}(x_{p}) := \int_{\mb{R}} \exp\left(-\kappa R^{-1/s} d^{1/(2s)-1}|w_{p}|^{1/s} + \imath w_{p}x_{p}\right) \mr{d} w_{p}.
\]
Then, the eigenvalues of $T_{\operatorname{Gev}}$ are \cite{Berlinet2004RKHS}
\[
C_{\kappa, d, s}^{-1} \prod_{p=1}^{d} \lambda_{j_{p}}(T_{\operatorname{Gev}}^{(p)}), \qquad j_{1},j_{2},\ldots,j_{d}\in \mb{N}_{+}.
\]
For $s>1$, \cite[Theorem I]{widom1964asymptotic2} gives
\[
\lim_{j\to \infty} \frac{\lambda_{j}(T_{\operatorname{Gev}}^{(p)})}{ 2\pi \exp\left(-\kappa R^{-1/s} d^{1/(2s)-1}[\pi j/2 +o(j)]^{1/s} \right) } = 1 .
\]
Thus, for
\[
\tilde{\kappa} < \frac{\kappa}{d}\left(\frac{\pi\sqrt{d}}{2R}\right)^{1/s},
\]
there exists a constant $C_{\operatorname{Gev}}$ such that $\lambda_{j}(T_{\operatorname{Gev}}^{(p)}) \le C_{\operatorname{Gev}}^{1/d} e^{-\tilde{\kappa}j^{1/s}}$ for all $j$. For $s=1$, \cite[Theorem II]{widom1964asymptotic2} gives
\[
\lim_{j\to \infty} \frac{\ln\lambda_{j}(T_{\operatorname{Gev}}^{(p)})}{j \pi} = -\frac{\boldsymbol{K}(\operatorname{sech} \pi R
\sqrt{d}\kappa^{-1})}{\boldsymbol{K}(\tanh \pi R \sqrt{d}\kappa^{-1})}.
\]
Hence, for
\[
\tilde{\kappa} < \pi \frac{\boldsymbol{K}(\operatorname{sech} \pi R \sqrt{d}\kappa^{-1})}{\boldsymbol{K}(\tanh \pi R \sqrt{d}\kappa^{-1})},
\]
there exists a constant $C_{\operatorname{Gev}}$ such that $\lambda_{j}(T_{\operatorname{Gev}}^{(p)}) \le C_{\operatorname{Gev}}^{1/d} e^{-\tilde{\kappa}j}$ for all $j$. By \eqref{ineq: d(lambda;theta,gamma) bounded by eigenvalues of Sigma},
\begin{equation*}
\begin{aligned}
d(\lambda;\theta,\gamma) & \le \sum_{j \in \mb{N}_{+}^{d}}\frac{2C_{\kappa, d, s}^{-\gamma} C_{\operatorname{Gev}}^{\gamma}  e^{-\tilde{\kappa}\gamma\sum_{p=1}^{d} j_{p}^{1/s}}}{C_{\kappa, d, s}^{-\gamma} C_{\operatorname{Gev}}^{\gamma}  e^{-\tilde{\kappa}\gamma\sum_{p=1}^{d} j_{p}^{1/s}} + \lambda^{\frac{\gamma}{1-\theta}}} \\
& \le \sum_{j \in \mb{N}_{+}^{d}}\frac{2  e^{-\tilde{\kappa}\gamma|j|_{1}^{1/s}}}{e^{-\tilde{\kappa}\gamma|j|_{1}^{1/s}} + C_{\kappa, d, s}^{\gamma} C_{\operatorname{Gev}}^{-\gamma}\lambda^{\frac{\gamma}{1-\theta}}} ,
\end{aligned}
\end{equation*}
where we used $\sum_{p=1}^{d} j_{p}^{1/s} \ge |j|_{1}^{1/s}$. Lemma \ref{Lemma: control sum series, Gevrey classes s>1 and Analytic class2} yields
\[
d(\lambda;\theta,\gamma) \le 2(2+1_{\{d=1\}}) \left[sd \tilde{\kappa}^{-1}\ln(C_{\kappa, d, s}^{-1} C_{\operatorname{Gev}}\lambda^{-\frac{1}{1-\theta}})
\right]^{sd} .
\]
Consequently, $d(\lambda)\le M$ is guaranteed by
\[
\lambda = C_{\kappa, d, s}^{\theta-1} C_{\operatorname{Gev}}^{1-\theta} \exp\left( - \frac{(1-\theta) \tilde{\kappa} M^{1/(sd)}
}{\left[2(2+1_{\{d=1\}})\right]^{1/(sd)}sd} \right),
\]
which implies the desired bound in (2).

\emph{Case (3): super-exponential decay.} As in Case (2), scaling reduces the problem to \eqref{eq: form of integral operators in Widom1963} on $(-1,1)^d$ with kernel
\[
\tilde{k}(x) = C_{\kappa, d, s}^{-1} \int_{\mb{R}^{d}} e^{-2\kappa R^{-s} |w|^{s}+\imath w^{\top}x} \mr{d} w .
\]
H\"older's inequality gives
\[
\min(1, d^{s/2-1})\sum_{p=1}^{d}|w_{p}|^{s} \le |w|^{s}.
\]
Lemma \ref{Lemma: monotonicity of certain integral operators} reduces the estimate to $T_{\operatorname{Sup}}$ with kernel
\[
k_{\operatorname{Sup}}(x) = C_{\kappa, d, s}^{-1} \prod_{p=1}^{d} \int_{\mb{R}} \exp\left(-\kappa \min(1, d^{s/2-1}) R^{-s}|w_{p}|^{s} + i w_{p}x_{p}\right)
\mr{d} w_{p}.
\]
Let $T_{\operatorname{Sup}}^{(p)}$ be the corresponding one-dimensional operator. The eigenvalues of $T_{\operatorname{Sup}}$ are \cite{Berlinet2004RKHS}
\[
C_{\kappa, d, s}^{-1} \prod_{p=1}^{d} \lambda_{j_{p}}(T_{\operatorname{Sup}}^{(p)}), \qquad j_{1},j_{2},\ldots,j_{d}\in\mb{N}_{+}.
\]
By \cite[Corollary 1]{widom1964asymptotic2},
\[
\lim_{j\to \infty} \frac{\ln\lambda_{j}(T_{\operatorname{Sup}}^{(p)})}{j \ln j} = - (2-2 / s).
\]
Thus, for any $0<\bar{\kappa}(s)<2-2/s$, there exists a constant $C_{\operatorname{Sup}}=C_{\operatorname{Sup}}(\kappa,\bar\kappa,d,s,R)$ such that
\[
\lambda_{j}(T_{\operatorname{Sup}}^{(p)}) \le C_{\operatorname{Sup}}^{1/d} e^{-\bar{\kappa} j \ln j}
\]
for all $j\in\mb{N}_{+}$. It follows from \eqref{ineq: d(lambda;theta,gamma) bounded by eigenvalues of Sigma} that
\begin{equation*}
\begin{aligned}
d(\lambda;\theta,\gamma) & \le \sum_{j\in \mb{N}_{+}^{d}} \frac{2C_{\kappa, d, s}^{-\gamma} C_{\operatorname{Sup}}^{\gamma}  e^{-\bar{\kappa}\gamma \sum_{p=1}^{d} j_{p} \ln j_{p}}}{C_{\kappa, d, s}^{-\gamma} C_{\operatorname{Sup}}^{\gamma}  e^{-\bar{\kappa}\gamma \sum_{p=1}^{d} j_{p} \ln j_{p}} + \lambda^{\frac{\gamma}{1-\theta}}} \\
& \le \sum_{j\in \mb{N}_{+}^{d}} \frac{2 e^{-\bar{\kappa}\gamma \sum_{p=1}^{d} j_{p} \ln j_{p}}}{e^{-\bar{\kappa}\gamma \sum_{p=1}^{d} j_{p} \ln j_{p}} + C_{\kappa, d, s}^{\gamma} C_{\operatorname{Sup}}^{-\gamma} \lambda^{\frac{\gamma}{1-\theta}}} .
\end{aligned}
\end{equation*}
By Lemma \ref{Lemma: control sum series, Fourier transform decays superexponentially2}, let $\tilde{M} \ge d$ satisfy
\[
C_{\kappa, d, s}^{\gamma} C_{\operatorname{Sup}}^{-\gamma} \lambda^{\frac{\gamma}{1-\theta}} = \left(\tilde{M}/d\right)^{-\bar{\kappa} \gamma\tilde{M}}.
\]
Then
\[
d(\lambda) \le 2 (\bar{\kappa}^{-1}\gamma^{-1}+3/2)^{d}\left(e\tilde{M}/d\right)^{d}.
\]
Consequently,
\begin{equation} \label{eq: the choice of lambda in Fourier transform decays superexponentially, interpolation}
\lambda = C_{\kappa, d, s}^{\theta-1} C_{\operatorname{Sup}}^{1-\theta}  \Upsilon\left[e^{-1} (\bar{\kappa}^{-1}\gamma^{-1}+3/2)^{-1}(M/2)^{1/d}\right]^{-(1-\theta)\bar{\kappa} d}
\end{equation}
guarantees $d(\lambda) \le M$.

\emph{Case (4): bandlimited features.} Scaling reduces the problem to \eqref{eq: form of integral operators in Widom1963} on $(-1,1)^d$ with kernel
\[
\tilde{k}_{\eta}(x) = (2S)^{-d} \int_{(-SR,SR)^{d}} e^{\imath w^{\top}x} \mr{d} w .
\]
The comparison operator $T_{\operatorname{Ban}}$ has the same kernel. As in \cite{widom1964asymptotic2}, let $\lambda_j(\gamma)$ be the eigenvalues of the one-dimensional operators on $(-1,1)$ with kernel
\[
(2\pi)^{-1} \int_{(-\gamma,\gamma)} e^{i w x}\mr{d}w = \frac{\sin(\gamma x)}{\pi x}.
\]
The eigenvalues of $T_{\operatorname{Ban}}$ are \cite{Berlinet2004RKHS}
\[
\pi^{d} S^{-d} \prod_{p=1}^{d} \lambda_{j_{p}}(SR), \qquad j_{1},j_{2},\ldots,j_{d}\in\mb{N}_{+}.
\]
For fixed $SR$, \cite[Corollary 2]{widom1964asymptotic2} gives
\[
\lim_{j\to \infty} \frac{\ln\lambda_{j}(SR)}{j \ln j} = - 2.
\]
Thus, for every $0<\bar{\kappa}<2$, there exists $C_{\operatorname{Ban}}$ such that
\[
\lambda_{j}(SR) \le C_{\operatorname{Ban}}^{1/d} e^{-\bar{\kappa} j \ln j}
\]
for all $j\in\mb{N}_{+}$. By \eqref{ineq: d(lambda;theta,gamma) bounded by eigenvalues of Sigma},
\begin{equation*}
\begin{aligned}
d(\lambda;\theta,\gamma) & \le \sum_{j\in \mb{N}_{+}^{d}} \frac{2(\pi/S)^{d\gamma}  C_{\operatorname{Ban}}^{\gamma}  e^{-\bar{\kappa}\gamma \sum_{p=1}^{d} j_{p} \ln j_{p}}}{(\pi/S)^{d\gamma}  C_{\operatorname{Ban}}^{\gamma}  e^{-\bar{\kappa}\gamma \sum_{p=1}^{d} j_{p} \ln j_{p}} + \lambda^{\frac{\gamma}{1-\theta}}}  \\
& \le \sum_{j\in \mb{N}_{+}^{d}} \frac{2 e^{-\bar{\kappa}\gamma \sum_{p=1}^{d} j_{p} \ln j_{p}}}{ e^{-\bar{\kappa}\gamma \sum_{p=1}^{d} j_{p} \ln j_{p}} + (S/\pi)^{d\gamma}  C_{\operatorname{Ban}}^{-\gamma} \lambda^{\frac{\gamma}{1-\theta}}} .
\end{aligned}
\end{equation*}
Let $\tilde{M}\ge d$ satisfy
\[
(S/\pi)^{d\gamma} C_{\operatorname{Ban}}^{-\gamma} \lambda^{\frac{\gamma}{1-\theta}} =\left(\tilde{M}/d\right)^{-\bar{\kappa}\gamma\tilde{M}}.
\]
Lemma \ref{Lemma: control sum series, Fourier transform decays superexponentially2} gives
\[
d(\lambda)\le 2(\bar{\kappa}^{-1}\gamma^{-1}+3/2)^d(e\tilde M/d)^d.
\]
Therefore,
\[
\lambda = (\pi/S)^{(1-\theta)d} C_{\operatorname{Ban}}^{1-\theta} \Upsilon\left[e^{-1}
(\bar{\kappa}^{-1}\gamma^{-1}+3/2)^{-1}(M/2)^{1/d}\right]^{-(1-\theta)\bar{\kappa} d}
\]
guarantees $d(\lambda)\le M$, which completes the proof.
\end{proof}

The effective-dimension bounds determine admissible penalty scales. To apply the abstract estimate, it remains to place the target classes in the source interpolation spaces and to embed the error interpolation spaces into the norms appearing in Theorem~\ref{Theorem: interpolation improved convergence rate of RFM}.

\subsection{Fourier descriptions and interpolation-space embeddings}
\label{Subsection: Fourier descriptions and interpolation embeddings}

This subsection provides the two deterministic space identifications needed by the uniform approximation theorem. We first recall the whole-space Fourier description of translation-invariant RKHSs and the corresponding weighted $L^2$ interpolation identity. These facts yield the embedding of the weighted Fourier target classes into the kernel interpolation scale. We then identify or embed that scale into the Sobolev norms used to measure the approximation error. The subsection closes by connecting the weighted Fourier assumptions to the spatial Gevrey and analytic conditions in Assumption~\ref{Assumption: regularities}.

\paragraph{Whole-space Fourier descriptions.}

Let $w$ be a positive measurable weight function. Define weighted space $L^{2}_{w}(\mb{R}^{d})$ as the space of all measurable functions for which $\|f\|_{L^{2}_{w}(\mb{R}^{d})} = \|w f\|_{L^2(\mb{R}^{d})}<\infty$. The following proposition shows that the RKHS associated with certain translation invariant kernels on $\mb{R}^{d}$ may be characterized as weighted $L^2$ spaces in Fourier space.
\begin{proposition}[{\cite[Theorem 10.12]{Wendland2005Scattered}}] \label{Proposition: Fourier characterization for RKHS a.w. translation invariant kernels}
Suppose that $k \in C(\mb{R}^{d}) \cap L^{1}(\mb{R}^{d})$ is a real-valued positive definite function. Define $$\mc{G}:=\left\{f \in L^{2}(\mb{R}^{d}) \cap
C(\mb{R}^{d}): \wh{f} / \sqrt{\wh{k}} \in L^{2}(\mb{R}^{d})\right\}$$ and equip this space with the bilinear form $$\la f, g\ra_{\mc{G}}:=(2 \pi)^{-d /
2}\la\wh{f} / \sqrt{\wh{k}}, \wh{g} / \sqrt{\wh{k}}\ra_{L^{2}(\mb{R}^{d})}=(2 \pi)^{-d / 2} \int_{\mb{R}^{d}} \frac{\wh{f}(w) \overline{\wh{g}(w)}}{\wh{k}(w)} d
w$$. Then $\mc{G}$ is a real Hilbert space with inner product $\la\cdot, \cdot\ra_{\mc{G}}$ and reproducing kernel $k(\cdot-\cdot)$. Hence $\mc{G}$ is the RKHS
associated with $k$ on $\mb{R}^{d} $, i.e., $\mc{G} = \mc{H}_{k}(\mb{R}^{d})$, and both inner products coincide. In particular, every $f \in \mc{H}_{k}(\mb{R}^{d})$ can be recovered from its Fourier transform $\wh{f} \in L^{1}(\mb{R}^{d}) \cap L^{2}(\mb{R}^{d})$.
\end{proposition}

\begin{proposition}[Bandlimited RKHS characterization]
\label{Proposition: bandlimited RKHS characterization}
Let $Q_S=(-S,S)^d$, and let $k_S$ be defined by
\[
\widehat{k_S}=(2\pi)^{-d/2}b_S1_{Q_S}, \qquad S,b_S>0.
\]
Then
\[
\mc H_{k_S}(\mb R^d) =PW_S:=\{F\in L^2(\mb R^d): \operatorname{supp}\widehat F\subset\overline{Q_S}\}, \qquad \|F\|_{\mc H_{k_S}(\mb R^d)}
=b_S^{-1/2}\|F\|_{L^2(\mb R^d)}.
\]
\end{proposition}

\begin{proof}
Equip $PW_S$ with the inner product $\langle F,G\rangle_{PW_S}=b_S^{-1}\langle F,G\rangle_{L^2}$. Since $\widehat F\in L^1(Q_S)$, Fourier inversion and Plancherel's identity give, for every $F\in PW_S$,
\[
\langle F,k_S(\cdot-x)\rangle_{PW_S} =(2\pi)^{-d/2}\int_{Q_S}\widehat F(w)e^{\imath w^\top x}\mr dw =F(x).
\]
Thus $PW_S$ is an RKHS with reproducing kernel $k_S(\cdot-\cdot)$. The uniqueness of the RKHS proves the claim.
\end{proof}

We recall a result from \cite{Peetre1964interpolation} stating the interpolation relationship between weighted $L^2$ spaces.
\begin{proposition}
Let $w_{1}$, $w_{2}$ be positive measurable functions. Then, for $0<\theta<1$,
\[
(L^{2}_{w_{1}}(\mb{R}^{d}),L^{2}_{w_{2}}(\mb{R}^{d}))_{\theta,2} =L^{2}_{w}(\mb{R}^{d}), \qquad w=w_{1}^{1-\theta}w_{2}^{\theta}.
\]
\end{proposition}

We recall extension and restriction theorems for RKHSs associated with the same kernel.
\begin{proposition}[{\cite[Theorem 10.46 and Theorem 10.47]{Wendland2005Scattered}}] \label{Proposition: extension and restriction theorems for RKHSs associated with the same kernel}
Let $\Omega_{1}\subseteq\Omega_{2}\subseteq\mb{R}^{d}$ and $k$ be a positive definite kernel on $\Omega_{2}$. Then, \\ (1) Each function  $f \in \mc{H}_{k}(\Omega_{1})$  has a natural extension to a function  $E f \in \mc{H}_{k}(\Omega_{2})$. Furthermore,  $\|E f\|_{\mc{H}_{k}(\Omega_{2})}=\|f\|_{\mc{H}_{k}(\Omega_{1})}$. \\ (2) The restriction $f|_{\Omega_{1}}$ of any function $f \in \mc{H}_{k}(\Omega_{2})$ is contained in $\mc{H}_{k}(\Omega_{1})$ with $\|f|_{\Omega_{1}}\|_{\mc{H}_{k}(\Omega_{1})} \le \|f\|_{\mc{H}_{k}(\Omega_{2})}$.
\end{proposition}

\paragraph{Target classes in the kernel interpolation scale.}
\begin{proposition}[Fourier source classes in the kernel interpolation scale]
\label{Proposition: Fourier source class embeds into kernel interpolation space}
Let $R>0$, $x_0\in\mb R^d$, $Q_R=x_0+(-R,R)^d$, $a>0$, $0<\kappa\le\bar\kappa$, and $\sigma=\kappa/\bar\kappa$. Let $k_{\bar\kappa,a}$ be the kernel induced by
\[
\mr d\tau_{\bar\kappa,a}(w)\propto e^{-2\bar\kappa|w|^a}\mr dw.
\]
Then
\[
\mc F_{\kappa,a}(Q_R) \hookrightarrow \mc H_{k_{\bar\kappa,a}}^\sigma(Q_R), \qquad \|u\|_{\mc H_{k_{\bar\kappa,a}}^\sigma(Q_R)} \lesssim_{\kappa,\bar\kappa,d,a}
\|u\|_{\kappa,a}.
\]
\end{proposition}

\begin{proof}
Let $\mc R_{Q_R}$ denote restriction from $\mb R^d$ to $Q_R$. By Proposition \ref{Proposition: Fourier characterization for RKHS a.w. translation invariant kernels}, the whole-space RKHS of $k_{\bar\kappa,a}$ has Fourier norm equivalent to
\[
\left\|e^{\bar\kappa|\cdot|^a}\widehat U\right\|_{L^2(\mb R^d)}.
\]
If $0<\sigma<1$, the weighted $L^2$ interpolation identity above gives
\[
\bigl(L^2(\mb R^d),\mc H_{k_{\bar\kappa,a}}(\mb R^d)\bigr)_{\sigma,2} = \left\{U: e^{\kappa|\cdot|^a}\widehat U\in L^2(\mb R^d)\right\},
\]
with equivalent norms. The map $\mc R_{Q_R}$ is contractive from $L^2(\mb R^d)$ to $L^2(Q_R)$ and, by Proposition \ref{Proposition: extension and restriction theorems for RKHSs associated with the same kernel}, from $\mc H_{k_{\bar\kappa,a}}(\mb R^d)$ to $\mc H_{k_{\bar\kappa,a}}(Q_R)$. Interpolating these two restriction estimates and using Proposition \ref{Proposition: interpolation between RKHS and L2} yields
\[
\|\mc R_{Q_R}U\|_{\mc H_{k_{\bar\kappa,a}}^\sigma(Q_R)} \lesssim_{\kappa,\bar\kappa,d,a} \left\|e^{\kappa|\cdot|^a}\widehat U\right\|_{L^2(\mb R^d)}.
\]
Taking the infimum over all admissible whole-space extensions $U$ of $u$ proves the claim when $0<\sigma<1$. If $\sigma=1$, then $\kappa=\bar\kappa$ and $\mc H_{k_{\bar\kappa,a}}^\sigma =\mc H_{k_{\bar\kappa,a}}$. The Fourier characterization and the contractive restriction in Proposition \ref{Proposition: extension and restriction theorems for RKHSs associated with the same kernel} give
\[
\|\mc R_{Q_R}U\|_{\mc H_{k_{\bar\kappa,a}}^1(Q_R)} \lesssim_{\bar\kappa,d,a} \left\|e^{\bar\kappa|\cdot|^a}\widehat U\right\|_{L^2(\mb R^d)}.
\]
Taking the infimum over admissible extensions proves the endpoint and completes the proof.
\end{proof}

\paragraph{Kernel interpolation spaces and error norms.}
The following proposition relates kernel interpolation spaces to Sobolev spaces on bounded domains.
\begin{proposition} \label{Proposition: relation between fractional Sobolev and interpolation space of RKHS on bounded domains}
(1) Suppose $k\in L^{1}(\mb{R}^{d})$ and $\wh{k}\simeq (1+ |\cdot|^{2})^{-s}$ with $s>d/2$. Then, for $0\le\nu\le s$, $\mc{H}_{k}^{\nu/s}(\Omega)= H^{\nu}(\Omega)$ with $\|\cdot\|_{\mc{H}_{k}^{\nu/s}(\Omega)}\simeq_{d,s,\Omega}\|\cdot\|_{H^{\nu}(\Omega)}$. \\ (2) Suppose $k\in L^{1}(\mb{R}^{d})$ and $\wh{k}\simeq e^{-2\kappa|\cdot|^{s}}$ with $\kappa,s>0$. Let $t\ge0$, $1\le p\le\infty$, and $0<\theta\le1$. Fix $\ell\ge t$ if $p=2$, and fix any $\ell>t+d/2$ otherwise. Then every $f\in\mc{H}_{k}^{\theta}(\Omega)$ is smooth and
\begin{equation*}
\|f\|_{W^{t,p}(\Omega)}
\lesssim_{d,\kappa,s,t,p,\ell,\theta,\Omega}
\|f\|_{\mc{H}_{k}^{\theta}(\Omega)}.
\end{equation*}
(3) Let $k=k_S$ be the bandlimited kernel in Proposition \ref{Proposition: bandlimited RKHS characterization}. Let $t,p,\theta$, and $\ell$ be as in (2). Then every $f\in\mc{H}_{k}^{\theta}(\Omega)$ is smooth and
\begin{equation*}
\|f\|_{W^{t,p}(\Omega)}
\lesssim_{d,t,p,\ell,\theta,\Omega}
b_S^{\theta/2}(1+dS^2)^{\ell/2}
\|f\|_{\mc{H}_{k}^{\theta}(\Omega)}.
\end{equation*}
\end{proposition}
\begin{proof}
(1) The case $\nu=0$ is immediate. The case $\nu=s$ follows as in \cite[Corollary 10.48]{Wendland2005Scattered}. We extend the argument to fractional Sobolev spaces. By Proposition \ref{Proposition: extension and restriction theorems for RKHSs associated with the same kernel} and Proposition \ref{Proposition: Fourier characterization for RKHS a.w. translation invariant kernels}, every $f\in \mc{H}_{k}(\Omega)$ has an extension $E f \in \mc{H}_{k}(\mb{R}^{d}) = H^{s}(\mb{R}^{d})$ and
\begin{equation} \label{ineq: mcH_k(Omega) continuously embed in H^s(Omega)}
\|f\|_{H^{s}(\Omega)} \le \|Ef\|_{H^{s}(\mb{R}^{d})} \lesssim_{d,s} \|Ef\|_{\mc{H}_{k}(\mb{R}^{d})} \lesssim_{d,s} \|f\|_{\mc{H}_{k}(\Omega)} .
\end{equation}
On the other hand, by \cite{Vyacheslav1999ExtensionsBesov}, every $f\in H^{s}(\Omega)$ has an extension $\tilde{E} f \in H^{s}(\mb{R}^{d}) = \mc{H}_{k}(\mb{R}^{d}) $ satisfying $\|\tilde{E} f\|_{H^{s}(\mb{R}^{d})} \lesssim_{d,s,\Omega} \|f\|_{H^{s}(\Omega)}$. Thus, $\tilde{E} f\in \mc{H}_{k}(\mb{R}^{d})$ and
\begin{equation} \label{ineq: H^s(Omega) continuously embed in mcH_k(Omega)}
\|f\|_{\mc{H}_{k}(\Omega)} \le \|\tilde{E}f\|_{\mc{H}_{k}(\mb{R}^{d})} \lesssim_{d,s}\|\tilde{E}f\|_{H^{s}(\mb{R}^{d})}\lesssim_{d,s,\Omega} \|f\|_{H^{s}(\Omega)} .
\end{equation}
For $0<\nu<s$, Proposition \ref{Proposition: interpolation between RKHS and L2} gives
\[
\mc{H}^{\nu/s}_{k}(\Omega) = (L^{2}(\Omega),\mc{H}_{k}(\Omega))_{\nu/s,2}.
\]
Moreover, $H^{\nu}(\Omega)=(L^{2}(\Omega),H^{s}(\Omega))_{\nu/s,2}$ with equivalent norms. The real interpolation theorem for bounded linear operators, together with \eqref{ineq: mcH_k(Omega) continuously embed in H^s(Omega)} and \eqref{ineq: H^s(Omega) continuously embed in mcH_k(Omega)}, yields the desired estimate.

(2) Fix $t,p,\theta$, and $\ell$ as in the statement. Propositions \ref{Proposition: extension and restriction theorems for RKHSs associated with the same kernel} and \ref{Proposition: Fourier characterization for RKHS a.w. translation invariant kernels} give, for every $f\in\mc H_k(\Omega)$,
\[
\|f\|_{H^{\ell/\theta}(\Omega)} \lesssim_{d,\kappa,s,\ell,\theta} \|f\|_{\mc H_k(\Omega)},
\]
because the exponential Fourier weight dominates $(1+|\cdot|^2)^{\ell/(2\theta)}$. For $0<\theta<1$, Proposition \ref{Proposition: interpolation between RKHS and L2}, the identity $H^\ell=(L^2,H^{\ell/\theta})_{\theta,2}$ on $\Omega$, and the real interpolation theorem imply
\[
\|f\|_{H^\ell(\Omega)} \lesssim_{d,\kappa,s,\ell,\theta,\Omega} \|f\|_{\mc H_k^\theta(\Omega)}.
\]
The same estimate follows directly from the endpoint bound when $\theta=1$. Finally, $H^\ell(\Omega)\hookrightarrow W^{t,p}(\Omega)$ by the choice of $\ell$.
Since $t$ is arbitrary, every $f\in\mc H_k^\theta(\Omega)$ is smooth. This proves (2). \\ (3) The argument is the same, but Proposition \ref{Proposition: bandlimited RKHS characterization} gives
\[
\|Ef\|_{L^2(\mb R^d)} =b_S^{1/2}\|Ef\|_{\mc H_k(\mb R^d)}
\]
for the minimum-norm bandlimited extension $Ef$. Hence
\[
\|f\|_{H^{\ell/\theta}(\Omega)} \le b_S^{1/2}(1+dS^2)^{\ell/(2\theta)} \|f\|_{\mc H_k(\Omega)}.
\]
Interpolating as in (2) yields
\[
\|f\|_{H^\ell(\Omega)} \lesssim_{d,\ell,\theta,\Omega} b_S^{\theta/2}(1+dS^2)^{\ell/2} \|f\|_{\mc H_k^\theta(\Omega)}.
\]
The Sobolev embedding $H^\ell(\Omega)\hookrightarrow W^{t,p}(\Omega)$ proves (3), including smoothness, and completes the proof.
\end{proof}

\paragraph{Spatial regularity and Fourier decay.}
We finally verify the implications used to interpret the Fourier source conditions in terms of standard smoothness classes.

\begin{proof}[Proof of Lemma \ref{Lemma: Fourier decay from Gevrey and analytic regularity}]
Suppose first that Assumption \ref{Assumption: regularities}(b) holds, and let $\widetilde u\in G^s(\Omega')$ be an extension of $u$. Choose $\chi\in G_0^s(\Omega')$ such that $\chi=1$ on a neighborhood of $\overline\Omega$, and extend $U:=\chi\widetilde u$ by zero to $\mb R^d$. Then $U\in G_0^s(\mb R^d)$ and $U|_\Omega=u$. By the Fourier characterization of compactly supported Gevrey functions \cite[Theorem 1.6.1]{rodino1993linear}, there exist $C,c>0$ such that
\[
|\widehat U(w)|\le C\exp(-c|w|^{1/s}),\qquad w\in\mb R^d.
\]
Consequently, for every $0<\kappa<c$, $e^{\kappa|\cdot|^{1/s}}\widehat U\in L^2(\mb R^d)$, and hence $\|u\|_{\kappa,1/s}<\infty$.

If Assumption \ref{Assumption: regularities}(c) holds, the Paley--Wiener type result \cite[Theorem IX.13]{Reed1978Methods} gives
\[
e^{\kappa|\cdot|}\widehat u\in L^2(\mb R^d) \qquad\text{for every }0<\kappa<\rho.
\]
Taking $U=u$ in the definition of $\|u\|_{\kappa,1}$ proves $\|u\|_{\kappa,1}<\infty$ and completes the proof.
\end{proof}

\subsection{Proof of the uniform approximation theorem}
\label{Section: proof uniform approximation regularity balls}

\begin{proof}[Proof of Theorem \ref{Theorem: interpolation improved convergence rate of RFM}]
We first prove a probabilistic implication which will be used repeatedly. Fix $\varrho\in\{\operatorname{ph},\operatorname{cx}\}$ and $0\le\theta<1$, set $\gamma=1-\theta$, and assume that the chosen value of $\lambda$ satisfies
\begin{equation} \label{ineq: d lambda smaller than M in interpolation improved theorem}
d(\lambda;\theta,\gamma)\le M.
\end{equation}
Both feature representations induce the same kernel integral operator $\Sigma$ and hence the same effective dimension $d(\lambda;\theta,\gamma)$. By Definition~\ref{Definition: Leverage score sampling} and \eqref{eq: leverage score sampling density}, the optimal density $q^*_{\lambda,\varrho}=q^*_{\lambda,\varrho}(\cdot;\theta,\gamma)$ associated with the chosen representation satisfies $d_{\max}(q^*_{\lambda,\varrho},\lambda;\theta,\gamma) =d(\lambda;\theta,\gamma)\le M$. Since $M\ge e\delta/14$, one has $14M/\delta\ge e$ and therefore $\ln(14M/\delta)\ge1$. Moreover, \eqref{ineq: d lambda smaller than M in interpolation improved theorem} implies
\[
\ln\frac{14d(\lambda;\theta,\gamma)}{\delta} \le \ln\frac{14M}{\delta}.
\]
Consequently,
\[
\begin{aligned}
N&\ge 3M\ln(14M/\delta)\\
&\ge 3d_{\max}(q^*_{\lambda,\varrho},\lambda;\theta,\gamma)
\max\left\{\ln\frac{14d(\lambda;\theta,\gamma)}{\delta},1\right\}.
\end{aligned}
\]
Thus the sufficient sampling condition \eqref{ineq: N > d_max(lambda) ln(d(lambda)/delta), interpolation} holds at this value of $\lambda$. By the definition of the critical penalty $\varsigma_N(\delta,q^*_{\lambda,\varrho},\theta,\gamma)$, this gives
\begin{equation} \label{ineq: varsigma no larger than lambda interpolation improved theorem}
\varsigma_N(\delta,q^*_{\lambda,\varrho},\theta,1-\theta)\le \lambda.
\end{equation}
The event supplied by Theorem \ref{Theorem: abstract error estimate for regression in interpolation spaces} has probability at least $1-\delta$, depends only on the sampled features, and is uniform over the unit ball of $\mc H_k^{\sigma}$. On this event, let $f\in\mc H_k^{\sigma}$, $\theta\le\sigma\le1$, and let $\zeta$ satisfy $0\le\zeta\le\theta$ and $2\theta-1\le\zeta$. Then $\bar p=\max\{\zeta,2\theta-1\}=\zeta$ in Theorem \ref{Theorem: abstract error estimate for regression in interpolation spaces}(1). Combining that theorem with \eqref{ineq: varsigma no larger than lambda interpolation improved theorem} and using homogeneity gives, simultaneously for all $f\in\mc H_k^{\sigma}$,
\begin{equation} \label{ineq: abstract consequence interpolation improved theorem}
\begin{aligned}
\|f-\Phi_{\varrho}\boldsymbol\beta^*\|_{\mc H_k^{\zeta}}
&\le 16\lambda^{\frac{\sigma-\zeta}{2(1-\theta)}}\|f\|_{\mc H_k^{\sigma}},\\
|\boldsymbol\beta^*|
&\le 16N^{-1/2}\lambda^{\frac{\sigma-1}{2(1-\theta)}}\|f\|_{\mc H_k^{\sigma}}.
\end{aligned}
\end{equation}
The RF function produced by $\boldsymbol\beta^*$ has the representation-independent form
\[
\Phi_{\varrho}\boldsymbol\beta^*(x) =\sum_{j=1}^{N}\beta_j^*q^*_{\lambda,\varrho}(v_j^{\varrho})^{-1/2} \phi_{\varrho}(x,v_j^{\varrho}).
\]
Set $\alpha_j^{\varrho}=q^*_{\lambda,\varrho}(v_j^{\varrho})^{-1/2}\beta_j^*$, $1\le j\le N$. Therefore
\begin{equation} \label{eq: beta alpha relation interpolation improved theorem}
|\alpha^{\varrho}|_{\ell^2(q^*_{\lambda,\varrho})}^{2}
=\sum_{j=1}^{N}q^*_{\lambda,\varrho}(v_j^{\varrho})|\alpha_j^{\varrho}|^2
=\sum_{j=1}^{N}|\beta_j^*|^2
=|\boldsymbol\beta^*|^2.
\end{equation}
For the complex-exponential representation and a real-valued target, taking the real part gives \eqref{Eq: linear combination of random Fourier features, cosine-sine}. Since the real-part map is contractive and $|\alpha_j^{\operatorname{cx}}|^2=|a_j|^2+|b_j|^2$, both estimates in \eqref{ineq: abstract consequence interpolation improved theorem} and the coefficient identity above remain valid for this real cosine--sine realization.

We now prove the four cases. For a fixed target $u$ and representation $\varrho$, abbreviate $u_N^{\varrho}$ and $\alpha^{\varrho}(u)$ by $u_N$ and $\alpha$, respectively.

\emph{Proof of (1).} For $\mr d\tau\simeq(1+|w|^2)^{-\bar s}\mr dw$, the kernel in \eqref{eq: translation invariant kernel for random Fourier feature} satisfies $\widehat k\eqsim(1+|\cdot|^2)^{-\bar s}$. Put
\[
\theta=\frac{\nu}{\bar s},\qquad \gamma=1-\frac{\nu}{\bar s},\qquad \sigma=\frac{s}{\bar s},\qquad \zeta=\frac{t}{\bar s}.
\]
The assumptions imply $0\le\theta\le\sigma\le1$, $0\le\zeta\le\theta$, $2\theta-1\le\zeta$, and $\bar s(1-\theta)=\bar s-\nu>d/2$. Proposition \ref{Proposition: upper bound for d(lambda;theta,gamma)}(1) therefore gives \eqref{ineq: d lambda smaller than M in interpolation improved theorem} for $\lambda=c_1M^{-2(\bar s-\nu)/d}$, after the fixed factor involving $|\Omega|$ is absorbed into $c_1$.

By Proposition \ref{Proposition: relation between fractional Sobolev and interpolation space of RKHS on bounded domains}(1),
\[
\|f\|_{\mc H_k^{r/\bar s}(\Omega)} \simeq_{d,\bar s,r,\Omega}\|f\|_{H^r(\Omega)}, \qquad 0\le r\le\bar s.
\]
Applying \eqref{ineq: abstract consequence interpolation improved theorem} and this equivalence at $r=s$ and $r=t$ yields
\[
\|u-u_N\|_{H^t(\Omega)} \lesssim_{d,\bar s,s,\nu,t,\Omega} M^{-(s-t)/d}\|u\|_{H^s(\Omega)}.
\]
Similarly, \eqref{eq: beta alpha relation interpolation improved theorem} and the coefficient estimate in \eqref{ineq: abstract consequence interpolation improved theorem} give
\[
|\alpha|_{\ell^2(q^*_{\lambda,\varrho})} \lesssim_{d,\bar s,s,\nu,\Omega} N^{-1/2}M^{(\bar s-s)/d}\|u\|_{H^s(\Omega)}.
\]
Taking the supremum over $\mb B_{H^s(\Omega)}$ proves (1).

\emph{Common preparation for (2)--(4).} Since $\Omega\subset\widetilde\Omega$,
\begin{equation}
\label{eq: restriction from outer cube to Lipschitz domain}
\|g\|_{W^{t,p}(\Omega)}\le\|g\|_{W^{t,p}(\widetilde\Omega)},
\qquad t\ge0,\quad 1\le p\le\infty.
\end{equation}
Translation invariance shows that the kernel integral operator on $\widetilde\Omega$ is unitarily equivalent to the corresponding operator on $(-R,R)^d$. Hence Proposition \ref{Proposition: upper bound for d(lambda;theta,gamma)} applies on $\widetilde\Omega$. For every $0<\zeta\le1$, Proposition \ref{Proposition: relation between fractional Sobolev and interpolation space of RKHS on bounded domains}(2) gives, for the exponential kernels in (2) and (3),
\begin{equation}
\label{eq: exponential embedding interpolation improved theorem}
\|g\|_{W^{t,p}(\widetilde\Omega)}
\lesssim_{d,t,p,\zeta,R}\|g\|_{\mc H_k^\zeta(\widetilde\Omega)},
\qquad t\ge0,\quad 1\le p\le\infty.
\end{equation}
For the bandlimited kernel in (4), part (3) of the same proposition gives
\begin{equation}
\label{eq: bandlimited embedding interpolation improved theorem}
\|g\|_{W^{t,p}(\widetilde\Omega)}
\lesssim_{d,t,p,\zeta,R}
\left(\frac{\pi}{S}\right)^{d\zeta/2}
(1+dS^2)^{(2t+d+1)/4}
\|g\|_{\mc H_k^\zeta(\widetilde\Omega)}.
\end{equation}
The sampling event, the coefficients, and $u_N$ in \eqref{ineq: abstract consequence interpolation improved theorem} are independent of $t$ and $p$. Since \eqref{eq: restriction from outer cube to Lipschitz domain}-- \eqref{eq: bandlimited embedding interpolation improved theorem} are deterministic, the estimates below hold simultaneously for all $t\ge0$ and $1\le p\le\infty$.

\emph{Proof of (2).} Put $\sigma=\kappa/\bar\kappa$ and set
\[
\theta=\frac\sigma2,\qquad \gamma=1-\frac\sigma2,\qquad \zeta=\frac\sigma4.
\]
Then $0<\zeta\le\theta<\sigma\le1$ and $2\theta-1<\zeta$. Choose an admissible whole-space extension $U$ of $u$ such that
\[
\|e^{\kappa|\cdot|^{1/s}}\widehat U\|_{L^2(\mb R^d)} \le2\|u\|_{\kappa,1/s},
\]
and put $f:=U|_{\widetilde\Omega}$. Proposition \ref{Proposition: Fourier source class embeds into kernel interpolation space} gives
\begin{equation}
\label{eq: source stretched exponential interpolation improved theorem}
\|f\|_{\mc H_k^\sigma(\widetilde\Omega)}
\lesssim_{\kappa,\bar\kappa,d,s}\|u\|_{\kappa,1/s}.
\end{equation}
Let $b_2$ be the right-hand side of \eqref{eq: def of kappa_s in subexponential and exponential convergence rate, interpolation} with $\theta=\sigma/2$ and decay parameter $\bar\kappa$, and set $a_\lambda:=b_2/2$. Proposition \ref{Proposition: upper bound for d(lambda;theta,gamma)}(2) then gives \eqref{ineq: d lambda smaller than M in interpolation improved theorem} for $\lambda=c_\lambda\exp(-a_\lambda M^{1/(sd)})$ and all sufficiently large $M$.

Applying \eqref{ineq: abstract consequence interpolation improved theorem} to $f$, followed by \eqref{eq: restriction from outer cube to Lipschitz domain}, \eqref{eq: exponential embedding interpolation improved theorem}, and \eqref{eq: source stretched exponential interpolation improved theorem}, we obtain
\[
\|u-u_N\|_{W^{t,p}(\Omega)} \lesssim_{\kappa,\bar\kappa,d,s,t,p,R} \|u\|_{\kappa,1/s} \exp(-a_{\mathrm e}M^{1/(sd)}),
\]
where $a_{\mathrm e}:=(\sigma-\zeta)a_\lambda/[2(1-\theta)]>0$. The coefficient estimate and \eqref{eq: beta alpha relation interpolation improved theorem} similarly give
\[
|\alpha|_{\ell^2(q^*_{\lambda,\varrho})} \lesssim_{\kappa,\bar\kappa,d,s,R} \|u\|_{\kappa,1/s}N^{-1/2} \exp(a_{\mathrm c}M^{1/(sd)}),
\]
where $a_{\mathrm c}:=(1-\sigma)a_\lambda/[2(1-\theta)]\ge0$. Taking the supremum over $\mb B_{\mc F_{\kappa,1/s}(\Omega)}$ proves (2).

\emph{Proof of (3).} Keep the indices $\sigma,\theta,\gamma,\zeta$ from (2). Choose an admissible extension $U$ satisfying
\[
\|e^{\kappa|\cdot|^s}\widehat U\|_{L^2(\mb R^d)} \le2\|u\|_{\kappa,s},
\]
and set $f:=U|_{\widetilde\Omega}$. Proposition \ref{Proposition: Fourier source class embeds into kernel interpolation space} implies
\begin{equation}
\label{eq: source superexp interpolation improved theorem}
\|f\|_{\mc H_k^\sigma(\widetilde\Omega)}
\lesssim_{\kappa,\bar\kappa,d,s}\|u\|_{\kappa,s}.
\end{equation}

Fix $\xi_s:=1-1/s\in(0,2-2/s)$ and put
\[
A_M:=\frac{(M/2)^{1/d}}
{e(\xi_s^{-1}\gamma^{-1}+3/2)}.
\]
Proposition \ref{Proposition: upper bound for d(lambda;theta,gamma)}(3) provides $c_3>0$ such that
\[
\lambda_0:=c_3\Upsilon(A_M)^{-(1-\theta)\xi_s d} \quad\text{satisfies}\quad d(\lambda_0;\theta,\gamma)\le M
\]
for all sufficiently large $M$. Set
\[
b_3:=\frac{2^{-1/d}}
{e\bigl(\xi_s^{-1}\gamma^{-1}+3/2\bigr)},
\qquad a_\lambda:=\frac12\gamma\xi_s b_3.
\]
Then $A_M=b_3M^{1/d}$ and
\[
(1-\theta)\xi_s d A_M\ln A_M =\gamma\xi_s b_3M^{1/d}\bigl(\ln M+d\ln b_3\bigr) \ge a_\lambda M^{1/d}\ln M
\]
for all sufficiently large $M$. Consequently,
\[
\lambda:=c_3\exp(-a_\lambda M^{1/d}\ln M)\ge\lambda_0
\]
for all sufficiently large $M$. The effective dimension is nonincreasing in $\lambda$, so \eqref{ineq: d lambda smaller than M in interpolation improved theorem} holds for this value of $\lambda$.

The same restriction and embedding argument as in (2) yields
\[
\|u-u_N\|_{W^{t,p}(\Omega)} \lesssim_{\kappa,\bar\kappa,d,s,t,p,R} \|u\|_{\kappa,s}\exp(-a_{\mathrm e}M^{1/d}\ln M),
\]
where $a_{\mathrm e}:=(\sigma-\zeta)a_\lambda/[2(1-\theta)]>0$. Moreover,
\[
|\alpha|_{\ell^2(q^*_{\lambda,\varrho})} \lesssim_{\kappa,\bar\kappa,d,s,R} \|u\|_{\kappa,s}N^{-1/2} \exp(a_{\mathrm c}M^{1/d}\ln M),
\]
where $a_{\mathrm c}:=(1-\sigma)a_\lambda/[2(1-\theta)]\ge0$. Taking the supremum over $\mb B_{\mc F_{\kappa,s}(\Omega)}$ proves (3).

\emph{Proof of (4).} Choose $U\in L^2(\mb R^d)$ such that $U|_\Omega=u$, $\operatorname{supp}\widehat U\subset[-S,S]^d$, and $\|U\|_{L^2(\mb R^d)}\le2\|u\|_{\mc B_S(\Omega)}$. Put $f:=U|_{\widetilde\Omega}$. The Fourier characterization of the bandlimited RKHS gives
\begin{equation}
\label{eq: source bandlimited interpolation improved theorem}
\|f\|_{\mc H_k(\widetilde\Omega)}
\lesssim_{d,S}\|U\|_{L^2(\mb R^d)}.
\end{equation}
Set $\theta=\gamma=1/2$, $\sigma=1$, and $\zeta=1/4$. Fix $\xi=1$ and put
\[
A_M:=\frac{(M/2)^{1/d}}{e(\xi^{-1}\gamma^{-1}+3/2)}.
\]
Proposition \ref{Proposition: upper bound for d(lambda;theta,gamma)}(4) provides a constant $c_4>0$ such that $\lambda_0:=c_4\Upsilon(A_M)^{-(1-\theta)\xi d}$ satisfies $d(\lambda_0;\theta,\gamma)\le M$ for all sufficiently large $M$. Here
\[
A_M=b_4M^{1/d}, \qquad b_4:=\frac{2^{-1/d}}{e(2+3/2)} =\frac{2^{1-1/d}}{7e}.
\]
Set $a_\lambda:=b_4/4=2^{1-1/d}/(28e)$. Since
\[
(1-\theta)\xi d A_M\ln A_M =\frac{b_4}{2}M^{1/d}\bigl(\ln M+d\ln b_4\bigr) \ge a_\lambda M^{1/d}\ln M
\]
for all sufficiently large $M$. Thus $\lambda:=c_4\exp(-a_\lambda M^{1/d}\ln M)\ge\lambda_0$, and \eqref{ineq: d lambda smaller than M in interpolation improved theorem} holds.

Using \eqref{ineq: abstract consequence interpolation improved theorem}, \eqref{eq: restriction from outer cube to Lipschitz domain}, \eqref{eq: bandlimited embedding interpolation improved theorem}, and \eqref{eq: source bandlimited interpolation improved theorem}, we obtain
\[
\|u-u_N\|_{W^{t,p}(\Omega)} \lesssim_{d,t,p,S,R}\|u\|_{\mc B_S(\Omega)} \exp(-a_{\mathrm e}M^{1/d}\ln M), \qquad a_{\mathrm e}:=\frac34a_\lambda.
\]
Since $\sigma=1$, the coefficient estimate contains no power of $\lambda$:
\[
|\alpha|_{\ell^2(q^*_{\lambda,\varrho})} \lesssim_{d,S,R}\|u\|_{\mc B_S(\Omega)}N^{-1/2}.
\]
Taking the supremum over $\mb B_{\mc B_S(\Omega)}$ proves (4) and completes the proof.
\end{proof}

\subsection{Effective dimension for growing uniform supports}
\label{Subsection: growing bandwidth effective dimension}

\begin{lemma}[A non-asymptotic effective-dimension bound for growing bandwidth]
\label{Lemma: growing bandwidth PSWF effective dimension}
Let $0<a<1$ and $R,S>0$, put $c=SR$, and define $\mr{d}\tau_S(w)=(2S)^{-d}1_{Q_S}(w)\mr{d}w$. Let $\Sigma_S$ be the integral operator on $L^2(Q_R)$ generated by the kernel associated with $\tau_S$. There is a constant $C_{a,d}\ge1$, independent of $S$, $R$, and $J$, with the following property: for every integer $J\ge \max\{2c,2\}$, the parameter $\lambda_{J,S}$ defined by \eqref{eq: lambda growing bandwidth PSWF} satisfies $d(\lambda_{J,S};1-a,a)\le 2J^d$.
\end{lemma}

\begin{proof}
Let $\mu_n(c)$, $n\ge0$, be the decreasing eigenvalues of the one-dimensional time-frequency concentration operator
\[
(Q_cf)(x)=\int_{-1}^1 \frac{\sin(c(x-y))}{\pi(x-y)}f(y)\mr{d}y.
\]
If $\lambda_n(c)$ is the corresponding eigenvalue of the finite Fourier transform, then $\mu_n(c)=\frac{c}{2\pi}|\lambda_n(c)|^2$. This relation is recorded in \cite[Equation (9)]{Osipov2013CertainInequalities}. The explicit prolate spheroidal wave-function estimate in \cite[Theorem 4]{Osipov2012ExplicitUpper}, also stated in \cite[Theorem 3.20]{OsipovRokhlinXiao2013PSWF}, is
\[
|\lambda_n(c)| \le \frac{\sqrt{\pi}\,c^n(n!)^2}
{(2n)!\,\Gamma(n+3/2)}.
\]
For $n\ge1$,
\[
\Gamma(n+3/2) =\frac{\sqrt\pi}{2}\prod_{k=1}^n\left(k+\frac12\right) \ge\frac{\sqrt\pi}{2}n!, \qquad (2n)!=n!\prod_{k=1}^n(n+k)\ge(n!)n^n.
\]
Consequently,
\begin{equation}
\label{eq: elementary explicit PSWF tail}
|\lambda_n(c)|\le2\left(\frac cn\right)^n,
\qquad
\mu_n(c)\le\frac{2c}{\pi}\left(\frac cn\right)^{2n}.
\end{equation}

Let $q=c/J\le1/2$. From \eqref{eq: elementary explicit PSWF tail}, for $n\ge J$, $\mu_n(c)^a\le C_a(1+c)^a q^{2an}$. Therefore
\begin{equation}
\label{eq: one dimensional PSWF tail sum}
T_J:=\sum_{n=J}^{\infty}\mu_n(c)^a
\le C_a(1+c)^a q^{2aJ}.
\end{equation}
Since $0<\mu_n(c)<1$ and the second term below is controlled by \eqref{eq: one dimensional PSWF tail sum},
\begin{equation}
\label{eq: one dimensional PSWF full sum}
A_c:=\sum_{n=0}^{\infty}\mu_n(c)^a
\le J+T_J\le C_aJ.
\end{equation}

By the product-kernel construction \cite[Theorem~13]{Berlinet2004RKHS}, $\Sigma_S$ is the $d$-fold tensor product of the corresponding one-dimensional integral operator. Hence, its eigenvalues are
\[
\Lambda_{\boldsymbol n} =\left(\frac{\pi}{S}\right)^d \prod_{\ell=1}^d\mu_{n_\ell}(c), \qquad \boldsymbol n\in\mb N_0^d.
\]
Let $\mathcal I_J=\{0,\ldots,J-1\}^d$. A union bound together with \eqref{eq: one dimensional PSWF tail sum} and \eqref{eq: one dimensional PSWF full sum} gives
\[
\sum_{\boldsymbol n\notin\mathcal I_J} \prod_{\ell=1}^d\mu_{n_\ell}(c)^a \le dT_JA_c^{d-1} \le C'_{a,d}J^{d-1}(1+c)^aq^{2aJ}.
\]
Since $x/(x+\lambda)\le\min\{1,x/\lambda\}$,
\[
\begin{aligned}
d(\lambda_{J,S};1-a,a)
&=\sum_{\boldsymbol n\in\mb N_0^d}
\frac{\Lambda_{\boldsymbol n}^{a}}
{\Lambda_{\boldsymbol n}^{a}+\lambda_{J,S}}\\
&\le J^d+
\left(\frac{\pi}{S}\right)^{da}\lambda_{J,S}^{-1}
\sum_{\boldsymbol n\notin\mathcal I_J}
\prod_{\ell=1}^d\mu_{n_\ell}(c)^a
\le2J^d.
\end{aligned}
\]
Choosing $C_{a,d}\ge C'_{a,d}$ in \eqref{eq: lambda growing bandwidth PSWF} completes the proof.
\end{proof}

\subsection{Proof of the growing-bandwidth leverage estimate}
\label{Subsection: proof growing bandwidth leverage approximation}

\begin{proof}[Proof of Proposition~\ref{Proposition: leverage approximation growing bandwidth target}]
Applying \eqref{eq: RKHS feature-space characterization} with $\tau=\tau_S$ gives
\[
\|U|_{Q_R}\|_{\mc H_{k_S}(Q_R)} \le (S/\pi)^{d/2}\|U\|_{L^2(\mb R^d)}.
\]

Since $\widehat k_S=(2\pi)^{-d/2}(\pi/S)^d1_{(-S,S)^d}$, Proposition \ref{Proposition: relation between fractional Sobolev and interpolation space of RKHS on bounded domains}(3), applied on $Q_R$ with $b_S=(\pi/S)^d$ and then restricted to $\Omega$, using the Sobolev order $t+(d+1)/2$, gives
\begin{equation}
\label{eq: explicit bandlimited interpolation embedding growing bandwidth}
\|f|_\Omega\|_{W^{t,p}(\Omega)}
\le C_{d,t,p,a,R}
(\pi/S)^{d(1-a)/2}
(1+dS^2)^{(2t+d+1)/4}
\|f\|_{\mc H_{k_S}^{1-a}(Q_R)}.
\end{equation}

Lemma~\ref{Lemma: growing bandwidth PSWF effective dimension} gives $d(\lambda_{J,S};1-a,a)\le2J^d$. Leverage-score sampling gives $d_{\max}(q^*_{\lambda_{J,S},\varrho}(\,\cdot\,;1-a,a),\allowbreak \lambda_{J,S};1-a,a)\allowbreak=d(\lambda_{J,S};1-a,a)$. Hence \eqref{eq: sample condition growing bandwidth leverage} makes $\lambda_{J,S}$ admissible in the sense of \eqref{ineq: N > d_max(lambda) ln(d(lambda)/delta), interpolation}; hence the associated critical penalty satisfies $\varsigma_N\le\lambda_{J,S}$. Applying Theorem \ref{Theorem: abstract error estimate for regression in interpolation spaces}(1) with source index $1$ and regression and error indices $1-a$, and then using homogeneity, gives
\[
\|U-u_{N,S}^{\varrho}\|_{\mc H_{k_S}^{1-a}(Q_R)} \le16\lambda_{J,S}^{1/2} \|U|_{Q_R}\|_{\mc H_{k_S}(Q_R)}.
\]
Combining this estimate with \eqref{eq: explicit bandlimited interpolation embedding growing bandwidth} and using $a+(1-a)=1$ cancels all powers $(\pi/S)^{da/2}$, $(\pi/S)^{d(1-a)/2}$, and $(S/\pi)^{d/2}$. This proves \eqref{eq: finite bandwidth leverage error explicit}. For each fixed $\varrho$, the sampling event and $u_{N,S}^{\varrho}$ are independent of $t$ and $p$, so the estimate holds simultaneously for all the stated error norms. The coefficient bound follows from the coefficient estimate in the same theorem and $|\alpha^{\varrho}|_{\ell^2(q^*_{\lambda_{J,S},\varrho} (\,\cdot\,;1-a,a))} =|\boldsymbol\beta^*|$; the power of $\lambda_{J,S}$ is zero because the source index is $s=1$. This completes the proof.
\end{proof}

\subsection{Proof of the growing uniform reference measure theorem}
\label{Subsection: proof growing uniform reference measures}

\begin{proof}[Proof of Theorem~\ref{Theorem: optimal growing bandwidth leverage rates}]
\emph{Representation-wise sampling event.} Fix $\varrho\in\{\operatorname{ph},\operatorname{cx}\}$ throughout the proof. Fix $a=1/2$ in Lemma \ref{Lemma: growing bandwidth PSWF effective dimension} and Proposition \ref{Proposition: leverage approximation growing bandwidth target}. For each case, set $S=S_J$. In (1) and (2), $S_JR/J\le1/4$. In (3),
\[
\frac{S_JR}{J} \le\frac14J^{-(s-1)/s}(\log J)^{1/s}\longrightarrow0.
\]
Thus $J\ge\max\{2S_JR,2\}$ for all sufficiently large $J$, and Proposition \ref{Proposition: leverage approximation growing bandwidth target} applies. Let $\mc A_J^{\varrho}$ be its sampling event. By the uniformity in Theorem \ref{Theorem: abstract error estimate for regression in interpolation spaces} and homogeneity, $\mc A_J^{\varrho}$ depends only on the sampled features, and \eqref{eq: finite bandwidth leverage error explicit} holds on $\mc A_J^{\varrho}$ for every $V\in L^2(\mb R^d)$ satisfying $\operatorname{supp}\widehat V\subset\overline{Q_{S_J}}$. In each case, $u_{N,J}^{\varrho}$ is obtained by applying that proposition to the bandlimited target constructed below.

\emph{Sobolev targets.} Apply \cite[Theorem D.1]{MingYu2026Spectral} on $\Omega$ with $\varepsilon=S_J^{-1}$. Using the fixed construction in its proof, choose a single approximant $V_J$, independent of $t$ and $p$. For $S_J\ge1$, that theorem gives $\operatorname{supp}\widehat V_J\subset B_{S_J}\subset\overline{Q_{S_J}}$ and, with $\beta_p:=d(1/p-1/2)_+$,
\begin{equation}
\label{eq: Wsp bandlimited smoothing growing leverage}
\begin{aligned}
\|u-V_J\|_{W^{t,p}(\Omega)}
&\le C S_J^{-(s-t)}\|u\|_{W^{s,p}(\Omega)},\\
\|V_J\|_{L^2(\mb R^d)}
&\le C S_J^{\beta_p}\|u\|_{W^{s,p}(\Omega)},
\end{aligned}
\end{equation}
for all $1\le p\le\infty$ such that $u\in W^{s,p}(\Omega)$ and $0\le t\le s$. Since $S_J=J/(4R_*)$ and $S_JR/J\le1/4$, Proposition \ref{Proposition: leverage approximation growing bandwidth target} and \eqref{eq: Wsp bandlimited smoothing growing leverage} imply
\[
\begin{aligned}
\|u-u_{N,J}^{\varrho}\|_{W^{t,p}(\Omega)}
&\le \|u-V_J\|_{W^{t,p}(\Omega)}
   +\|V_J-u_{N,J}^{\varrho}\|_{W^{t,p}(\Omega)}\\
&\le C\|u\|_{W^{s,p}(\Omega)}
\left[J^{-(s-t)}+J^{t+d/2+\beta_p+a/2}4^{-aJ}\right].
\end{aligned}
\]
The quantity $J^{s+d/2+\beta_p+a/2}4^{-aJ}$ is bounded for $J\ge1$; hence the second term is bounded by a constant times $J^{-(s-t)}$. Since $u\in\mb B_{W^{s,p}(\Omega)}$, this proves (1).

\emph{Fourier-regular targets.} Put $q=1/s$ in (2) and $q=s$ in (3). By the definition of the quotient norm, choose an extension $U\in L^2(\mb R^d)$, independent of $t,p$, such that
\[
U|_\Omega=u, \qquad \|e^{\kappa_0|\cdot|^q}\widehat U\|_2 \le2\|u\|_{\kappa_0,q}.
\]
For $S>0$, define $\widehat U_S=1_{Q_S}\widehat U$. With $\ell=t+(d+1)/2$ and $\kappa=\kappa_0/2$, Sobolev embedding, Plancherel's identity, and the definition of $U_S$ give
\begin{equation}
\label{eq: weighted Fourier tail growing leverage}
\begin{aligned}
\|u-U_S\|_{W^{t,p}(\Omega)}
&\le C\|U-U_S\|_{H^\ell(\mb R^d)}\\
&\le C\sup_{w\notin Q_S}
 (1+|w|^2)^{\ell/2}e^{-\kappa_0|w|^q}
 \|e^{\kappa_0|\cdot|^q}\widehat U\|_2\\
&\le C\|u\|_{\kappa_0,q}e^{-\kappa S^q}.
\end{aligned}
\end{equation}
The last inequality follows because $(1+|w|^2)^{\ell/2}e^{-(\kappa_0-\kappa)|w|^q}$ is bounded and $|w|\ge S$ outside $Q_S$. Moreover,
\[
\|U_S\|_{L^2(\mb R^d)} \le\|U\|_{L^2(\mb R^d)} \le2\|u\|_{\kappa_0,q}.
\]
We apply Proposition \ref{Proposition: leverage approximation growing bandwidth target} to $U_{S_J}$ and use
\[
\|u-u_{N,J}^{\varrho}\|_{W^{t,p}(\Omega)} \le\|u-U_{S_J}\|_{W^{t,p}(\Omega)} +\|U_{S_J}-u_{N,J}^{\varrho}\|_{W^{t,p}(\Omega)}.
\]

In (2), $S_J=J/(4R_*)$. Equation \eqref{eq: weighted Fourier tail growing leverage} and \eqref{eq: finite bandwidth leverage error explicit} give, respectively,
\[
C\|u\|_{\kappa_0,1/s}e^{-cJ^{1/s}} \quad\text{and}\quad C\|u\|_{\kappa_0,1/s}J^A4^{-aJ},
\]
for some $A=A(d,t,p,a)\ge0$. Since $J^A4^{-aJ}\le C e^{-cJ^{1/s}}$ for $s\ge1$ and $u\in\mb B_{\mc F_{\kappa_0,1/s}(\Omega)}$, assertion (2) follows, including the endpoint $s=1$.

In (3), $S_J=(J\log J)^{1/s}/(4R_*)$, and \eqref{eq: weighted Fourier tail growing leverage} is bounded by $C\|u\|_{\kappa_0,s}e^{-cJ\log J}$. Furthermore,
\[
\log\left(\frac{J}{S_JR}\right) \ge \log4+\frac{s-1}{s}\log J-\frac1s\log\log J \ge\frac{s-1}{2s}\log J
\]
for all sufficiently large $J$. Therefore,
\[
\left(\frac{S_JR}{J}\right)^{aJ} \le \exp\left(-\frac{a(s-1)}{2s}J\log J\right).
\]
The other factors in \eqref{eq: finite bandwidth leverage error explicit} grow at most algebraically in $J$ and $\log J$ and are absorbed by reducing the exponential constant. Since $u\in\mb B_{\mc F_{\kappa_0,s}(\Omega)}$, this proves (3). The extension $U$, its truncations, $\mc A_J^{\varrho}$, and $u_{N,J}^{\varrho}$ are independent of $t,p$, which proves the stated simultaneity in (2) and (3).

\emph{Conversion to $N$.}
To justify the stated sample-size choice, set $L=N/\log(N/\delta)$ and choose $J=\lfloor c_0L^{1/d}\rfloor$, where $c_0^d\le1/12$. For all sufficiently large $N$, one has
\[
\log(28J^d/\delta)\le2\log(N/\delta),
\]
because $J^d\le N$ and $N/\delta\ge28$ in this regime, and hence
\[
6J^d\log(28J^d/\delta) \le12c_0^dN\le N.
\]
Thus \eqref{eq: sample condition growing bandwidth leverage} holds. Also, once $c_0L^{1/d}\ge2$, $(c_0/2)L^{1/d}\le J\le c_0L^{1/d}$. Hence $J\asymp L^{1/d}$ and $\log J\asymp\log L$, and substitution in the three $J$-rates gives \eqref{eq: optimal rates in sample number N}, which completes the proof.
\end{proof}

\section{Technical Lemmas} \label{}

We prove the following technical lemma to control the summation series.
\begin{lemma} \label{Lemma: control sum series, Sobolev}
For arbitrary $C, \lambda>0$, $s>1$, it holds
\begin{equation*}
\sum_{j=1}^{\infty} \frac{C j^{-s}}{C j^{-s} + \lambda}
\le \frac{2s-1}{s-1}(C / \lambda)^{1/s}.
\end{equation*}
\end{lemma}

\begin{proof}
Note that $$\frac{C j^{-s}}{C j^{-s} + \lambda} \le \min(1, Cj^{-s}/\lambda)$$. For $\lambda\le C$, a direct calculation yields
\begin{equation*}
\begin{aligned}
\sum_{j=1}^{\infty} \frac{C j^{-s}}{C j^{-s} + \lambda} & \le \left\lceil (C / \lambda)^{1/s}\right\rceil + \frac{C}{\lambda}\int_{(C / \lambda)^{1/s}}^{\infty} t^{-s} d t \\
& = \left\lceil (C / \lambda)^{1/s}\right\rceil + \frac{1}{s-1}(C / \lambda)^{1/s} \\
& \le \frac{2s-1}{s-1}(C / \lambda)^{1/s} .
\end{aligned}
\end{equation*}
For $\lambda> C$,
\begin{equation*}
\begin{aligned}
\sum_{j=1}^{\infty} \frac{C j^{-s}}{C j^{-s} + \lambda} & \le  \frac{C}{\lambda}\sum_{j=1}^{\infty} j^{-s} \le \frac{sC}{(s-1)\lambda} ,
\end{aligned}
\end{equation*}
which completes the proof.
\end{proof}

We prove the following technical lemma to control the summation series.

\begin{lemma} \label{Lemma: control sum series, Gevrey classes s>1 and Analytic class2}
For any $\tilde{\kappa}>0$, $s\ge 1$ and $0< \lambda\le e^{-\max(\tilde{\kappa},1)}$, it holds
\begin{equation*}
\begin{aligned}
\sum_{j \in \mb{N}_{+}^{d}}\frac{e^{-\tilde{\kappa}|j|_{1}^{1/s}}}{ e^{-\tilde{\kappa}|j|_{1}^{1/s}} + \lambda}
& \le (2+1_{\{d=1\}}) \left[sd \tilde{\kappa}^{-1}\ln(1 / \lambda) \right]^{sd} .
\end{aligned}
\end{equation*}
\end{lemma}

\begin{proof}
First, we denote $r=\tilde{\kappa}^{-s}\ln^{s} (1 / \lambda)\ge1$ and divide the summation into two parts
\begin{equation*}
\begin{aligned}
\operatorname{I} = \sum_{|j|_{1} \le r+1} \frac{e^{-\tilde{\kappa}|j|_{1}^{1/s}}}{ e^{-\tilde{\kappa}|j|_{1}^{1/s}} + \lambda}, \qquad
\operatorname{II} = \sum_{|j|_{1} > r+1} \frac{e^{-\tilde{\kappa}|j|_{1}^{1/s}}}{ e^{-\tilde{\kappa}|j|_{1}^{1/s}} + \lambda}.
\end{aligned}
\end{equation*}
Since each summand is less than $1$, we have \begin{equation} \label{ineq: bound for operatornameI}
\operatorname{I} \le\begin{pmatrix}
 \lfloor r\rfloor+1\\ d
\end{pmatrix} \le (1+1_{\{d\le2\}}) \Gamma(d+1)^{-1} r^{d}.
\end{equation}
To bound $\operatorname{II}$, note that the number of $j\in \mb{N}_{+}^{d}$ for which $|j|_{1} =T$ is $$\begin{pmatrix}
 T-1\\ d-1
\end{pmatrix} \le \Gamma(d)^{-1}(T-1)^{d-1} .$$
Then, with the above estimate, $\operatorname{II}$ is bounded by
\begin{equation*}
\begin{aligned}
\lambda^{-1} \sum_{T = \lfloor r\rfloor+2}^{\infty}  \sum_{|j|_{1} = T} e^{-\tilde{\kappa}|j|_{1}^{1/s}}
\le \lambda^{-1}\Gamma(d)^{-1} \sum_{T = \lfloor r\rfloor+2}^{\infty} T^{d-1}e^{-\tilde{\kappa}T^{1/s}} .
\end{aligned}
\end{equation*}
Denote $h(T) := T^{d-1}e^{-\tilde{\kappa}T^{1/s}}$ for $T>0$. We shall use the integral of $h$ to control the above summation. For $d=1$, since $h(T)$ is a positive decreasing function of $T$,
\begin{equation} \label{ineq: sum h(T) bounded by integral of h, simple}
\begin{aligned}
\sum_{T = \lfloor r\rfloor+2}^{\infty} h(T) & \le \int_{r}^{\infty} h(T) \mr{d}T .
\end{aligned}
\end{equation}
For $d\ge2$, as $T$ increases, $h(T)$ increases and then decreases, attaining its maximum at $T^{*}=[s(d-1)/\tilde{\kappa}]^{s}$. If $\tilde{\kappa}\ge s(d-1)$, then $T^{*}\le1$ and (\ref{ineq: sum h(T) bounded by integral of h, simple}) holds. Otherwise, $\tilde{\kappa}<s(d-1)$, $h(T)\le \int_{T}^{T+1}h(t)\mr{d}t$ for $T\le T^{*}-1$, and $h(T)\le \int_{T-1}^{T}h(t)\mr{d}t$ for $T\ge T^{*}+1$. There are one or two integers in $(T^{*}-1,T^{*}+1)$. If there is one, its function
value does not exceed $h(T^{*})$. If there are two, denoted by $T_{1}<T_{2}$, then $$\min(h(T_{1}),h(T_{2})) \le \int_{T_{1}}^{T_{2}}h(t)\mr{d}t, \qquad
\max(h(T_{1}),h(T_{2}))\le h(T^{*})$$. Therefore,
\begin{equation} \label{ineq: sum h(T) bounded by integral of h, plus max value h(T*)}
\begin{aligned}
\sum_{T = \lfloor r\rfloor+2}^{\infty} h(T) & \le \int_{r}^{\infty} h(T) \mr{d}T + h(T^{*}) 1_{\{\lfloor r\rfloor+1\le T^{*}\}}.
\end{aligned}
\end{equation}
The second term on the right-hand side above vanishes unless $r<T^{*}$. To control it with the first term, we employ the concavity of $h$. In fact,
\begin{equation*}
\begin{aligned}
h^{\prime\prime}(T) = \left[\left(\frac{\tilde{\kappa}}{s}T^{1/s}\right)^{2} - \left(2d-3+\frac{1}{s}\right)\frac{\tilde{\kappa}}{s}T^{1/s} + (d-1)(d-2)\right] T^{d-3}e^{-\tilde{\kappa}T^{1/s}} ,
\end{aligned}
\end{equation*}
and $h(T)$ is concave on $[T^{*},T_{c}]$, where $T_{c}$ satisfies
\begin{equation*}
\begin{aligned}
\frac{\tilde{\kappa}}{s}T_{c}^{1/s} & = \frac{1}{2}\left[\left(2d-3+\frac{1}{s}\right) + \sqrt{\left(2d-3+\frac{1}{s}\right)^{2}-4(d-1)(d-2)}\right] \\
& \ge d-1+\frac{1}{s} .
\end{aligned}
\end{equation*}
Denote $T_{a} = (s/\tilde{\kappa})^{s}(d-1+1/s)^{s}$ with $T_{a}\in(T^{*},T_{c}]$. Since $s\ge1$ and $\tilde{\kappa}< s(d-1)$, the length of $[T^{*},T_{a}]$ is at least
\begin{equation*}
\begin{aligned}
\left(\frac{s}{\tilde{\kappa}}\right)^{s}\left[\left(d-1+\frac{1}{s}\right)^{s}-(d-1)^{s}\right] & \ge \left(\frac{s}{\tilde{\kappa}}\right)^{s}(d-1)^{s-1} \ge \frac{s}{\tilde{\kappa}} ,
\end{aligned}
\end{equation*}
and $h(T_{a})\ge h(T^{*})/e$. Then, the concavity of $h$ on $[T^{*},T_{c}]$ ensures
\begin{equation*}
\begin{aligned}
\frac{1}{2}(T_{a}-T^{*})\left(h(T^{*})+h(T_{a})\right) & \le \int_{T^{*}}^{T_{a}} h(T) \mr{d}T ,
\end{aligned}
\end{equation*}
which gives that for $r\le T^{*}$,
\begin{equation} \label{ineq: control max value h(T*) with integral of h}
\begin{aligned}
h(T^{*}) & \le \frac{2e\tilde{\kappa}}{(e+1)s}\int_{T^{*}}^{T_{a}} h(T) \mr{d}T \\
& \le \frac{2e\tilde{\kappa}}{(e+1)s}\int_{r}^{\infty} h(T) \mr{d}T.
\end{aligned}
\end{equation}

It remains to estimate the integral of $h$. For $a\ge1$, the substitution $t=u+a$ gives
\[
\begin{aligned}
\int_{a}^{\infty} t^{b-1} e^{-t} \mr{d} t & \le e^{-a} \int_{0}^{\infty}\left[a^{b-1} + (b-1)a^{b-2}u\right] e^{-u} \mr{d} u \\
& \le b e^{-a} a^{b-1}, &&1\le b\le2,\\[2pt]
\int_{a}^{\infty} t^{b-1} e^{-t} \mr{d} t & \le 2^{b-2} e^{-a} \int_{0}^{\infty}\left(u^{b-1}+a^{b-1}\right) e^{-u} \mr{d} u \\
& \le 2^{b-2} e^{-a} (\Gamma(b)+1)a^{b-1}, &&b\ge2.
\end{aligned}
\]
With the above facts and the variable substitution $t = \tilde{\kappa}T^{1/s}$, we obtain
\begin{equation} \label{ineq: control the integral of h(T)}
\begin{aligned}
\int_{r}^{\infty} h(T) \mr{d}T & \le s\tilde{\kappa}^{-sd}\int_{\tilde{\kappa}r^{1/s}}^{\infty} t^{sd-1}e^{-t} \mr{d}t, \\
& \le [s 1_{\{sd<2\}} + 2^{sd-2}(\Gamma(sd)+1)1_{\{sd\ge2\}}] s\lambda \tilde{\kappa}^{-1}  r^{d-1/s}  ,
\end{aligned}
\end{equation}
where $\tilde{\kappa} r^{1/ s}= \ln (1 / \lambda) \ge 1$. Plugging this into (\ref{ineq: sum h(T) bounded by integral of h, simple}) yields for $d=1$,
\begin{equation*}
\begin{aligned}
\operatorname{II} & \le [s 1_{\{s<2\}} + 2^{s-2}(\Gamma(s)+1)1_{\{s\ge2\}}] s r ,
\end{aligned}
\end{equation*}
and then a combination with (\ref{ineq: bound for operatornameI}) gives
\begin{equation*}
\begin{aligned}
\sum_{j \in \mb{N}_{+}^{d}}\frac{e^{-\tilde{\kappa}|j|_{1}^{1/s}}}{ e^{-\tilde{\kappa}|j|_{1}^{1/s}} + \lambda}
& \le [2 + s^{2} 1_{\{s<2\}} + 2^{s-2}s(\Gamma(s)+1)1_{\{s\ge2\}}]r \le 3s^{s} r ,
\end{aligned}
\end{equation*}
where we used elementary facts $2 + s^{2}\le3s^{s}$ and for $s\ge2$,
\begin{equation*}
\begin{aligned}
2^{s-2}s(\Gamma(s)+1)+2 & \le 2^{s-3/2}\sqrt{\pi}s^{s+1/2}e^{-s+1/24}+ (2^{s-2}s+2) \le 3s^s
\end{aligned}
\end{equation*}
where the last inequality follows from (\ref{ineq: bound for Gamma function from Artin1964gamma}). Similarly, for $d\ge2$, a combination of (\ref{ineq: sum h(T) bounded by integral of h, plus max value h(T*)}), (\ref{ineq: control max value h(T*) with integral of h}) and (\ref{ineq: control the integral of h(T)}) yields
\begin{equation*}
\begin{aligned}
\sum_{T = \lfloor r\rfloor+2}^{\infty} h(T)
& \le \left[1+\frac{2e\tilde{\kappa}}{(e+1)s}\right]\int_{r}^{\infty} h(T) \mr{d}T \\
& \le \left[1+\frac{2e\tilde{\kappa}}{(e+1)s}\right] 2^{sd-2}(\Gamma(sd)+1) s\lambda \tilde{\kappa}^{-1}  r^{d-1/s} \\
& \le 2^{sd-2}\left(s +\frac{2e}{e+1}\right) (\Gamma(sd)+1)\lambda r^{d} \\
& \le 2^{sd-2}\left(s +\frac{3}{2}\right) (\Gamma(sd)+1)\lambda r^{d},
\end{aligned}
\end{equation*}
and then with (\ref{ineq: bound for operatornameI}), we get
\begin{equation*}
\begin{aligned}
& \sum_{j \in \mb{N}_{+}^{d}}\frac{e^{-\tilde{\kappa}|j|_{1}^{1/s}}}{ e^{-\tilde{\kappa}|j|_{1}^{1/s}} + \lambda} \le \left[2^{sd-2}\left(s +\frac{3}{2}\right) \frac{\Gamma(sd)+1}{\Gamma(d)} + \frac{1}{\Gamma(d)} \right] r^{d}.
\end{aligned}
\end{equation*}
Note that $\Gamma(sd)+1\ge2$. To simplify the prefactor before $r^{d}$, we employ the fact
\begin{equation} \label{ineq: bound for Gamma function from Artin1964gamma}
\sqrt{2 \pi} x^{x-1/2} e^{-x}\le \Gamma(x) \le \sqrt{2 \pi} x^{x-1/2} e^{-x+1 /(12 x)}
\end{equation}
from \cite[Chapter 3]{Artin1964gamma} and calculate
\begin{equation*}
\begin{aligned}
2^{sd-1}(s +2) \frac{\Gamma(sd)}{\Gamma(d)}
& \le 2^{sd-1}(s +2) \frac{(sd)^{sd-1/2} e^{-sd+1 /24}}{d^{d-1/2} e^{-d}} \\
& \le  \frac{e^{1 /24}(s +2)}{2\sqrt{s}} \left(\frac{2sd}{e}\right)^{sd}\left(\frac{e}{d}\right)^{d} \\
& \le  2\sqrt{s} (2/e)^{2(s-1)} (sd)^{sd} \le 2  (sd)^{sd}.
\end{aligned}
\end{equation*}
Recalling that $r=\tilde{\kappa}^{-s}\ln^{s}(1/\lambda)$ gives the claimed estimate and completes the proof.
\end{proof}

We prove the following technical lemma to control the summation series.
\begin{lemma} \label{Lemma: control sum series, Fourier transform decays superexponentially2}
Let $0<\tilde{\kappa}<2$ and $0< \lambda\le 1$. Choose $M \ge d$ such that $\lambda=\left(M/d\right)^{-\tilde{\kappa} M}$. Then,
\[
\sum_{j \in \mb{N}_{+}^{d}}\frac{e^{-\tilde{\kappa} \sum_{p=1}^{d} j_{p} \ln j_{p}}}{ e^{-\tilde{\kappa} \sum_{p=1}^{d} j_{p} \ln j_{p}} + \lambda} \le
\left(\frac{1}{\tilde{\kappa}}+\frac{3}{2}\right)^{d} \left(\frac{eM}{d}\right)^{d}.
\]
\end{lemma}

\begin{proof}
Since $x \ln x$ is convex, Jensen's inequality gives $$\frac{1}{d} \sum_{p=1}^{d} j_{p} \ln j_{p} \ge \frac{|j|_{1}}{d} \ln \frac{|j|_{1}}{d}$$, which implies
\begin{equation*}
\begin{aligned}
\sum_{j \in \mb{N}_{+}^{d}}\frac{e^{-\tilde{\kappa} \sum_{p=1}^{d} j_{p} \ln j_{p}}}{ e^{-\tilde{\kappa} \sum_{p=1}^{d} j_{p} \ln j_{p}} + \lambda}
& \le  \sum_{j \in \mb{N}_{+}^{d}}\frac{e^{-\tilde{\kappa} |j|_{1} \ln(|j|_{1}/d) }}{ e^{-\tilde{\kappa} |j|_{1} \ln(|j|_{1}/d)} + \lambda} . \\
\end{aligned}
\end{equation*}
Similar as in the proof of Lemma \ref{Lemma: control sum series, Gevrey classes s>1 and Analytic class2}, we divide the summation into two parts
\begin{equation*}
\begin{aligned}
\operatorname{I} & := \sum_{|j|_{1} \le M}\frac{e^{-\tilde{\kappa} |j|_{1} \ln(|j|_{1}/d) }}{ e^{-\tilde{\kappa} |j|_{1} \ln(|j|_{1}/d)} + \lambda} \le
\begin{pmatrix}
 \lfloor M\rfloor\\ d
\end{pmatrix} , \\\operatorname{II} & := \sum_{|j|_{1} > M}\frac{e^{-\tilde{\kappa} |j|_{1} \ln(|j|_{1}/d) }}{ e^{-\tilde{\kappa} |j|_{1} \ln(|j|_{1}/d)} + \lambda}
\le \frac{1}{\lambda} \sum_{T=\lfloor M\rfloor+1}^{\infty} \begin{pmatrix}
T-1\\ d-1
\end{pmatrix} \left(\frac{T}{d}\right)^{-\tilde{\kappa} T}.
\end{aligned}
\end{equation*}
Using $\Gamma(d+1) = d\Gamma(d)$ and (\ref{ineq: bound for Gamma function from Artin1964gamma}), we bound $\operatorname{I}$ by $M$ for $d=1$ and for $d\ge2$,
\begin{equation} \label{ineq: control for operatornameI in pf of Lemma: control sum series, Fourier transform decays superexponentially2}
\begin{aligned}
\operatorname{I} & \le \frac{\left(M-(d-1)/2\right)^{d}}{\sqrt{2\pi}d^{d+1/2}e^{-d}} \le  \frac{1}{\sqrt{2\pi d}}\left(\frac{e M}{d}\right)^{d}.
\end{aligned}
\end{equation}
Since $\left(M/d\right)^{\tilde{\kappa} M}$ is an increasing function of $M$ when $M\ge d$, we bound $\operatorname{II}$ for $d=1$ by
\begin{equation*}
\begin{aligned}
\operatorname{II} & \le \sum_{T=0}^{\infty} \left(T+\lfloor M\rfloor+1\right)^{-\tilde{\kappa} (T+\lfloor M\rfloor+1)} M^{\tilde{\kappa} M} \\
& \le \sum_{T=0}^{\infty} \left(T+\lfloor M\rfloor+1\right)^{-\tilde{\kappa}T} \le \left(1-2^{-\tilde{\kappa}}\right)^{-1} .\end{aligned}
\end{equation*}
Then, applying Lemma \ref{Auxiliary technical Lemma to pf Lemma: control sum series, Fourier transform decays superexponentially}, we get $\operatorname{II}\le 2 \max(\tilde{\kappa}^{-1},2)$ and
\begin{equation*}
\begin{aligned}
\sum_{j \in \mb{N}_{+}^{d}}\frac{e^{-\tilde{\kappa} \sum_{p=1}^{d} j_{p} \ln j_{p}}}{ e^{-\tilde{\kappa} \sum_{p=1}^{d} j_{p} \ln j_{p}} + \lambda} & \le \left[1 + 2 \max(\tilde{\kappa}^{-1},2)\right] M ,
\end{aligned}
\end{equation*}
which implies the desired bound for $d=1$. For $d\ge2$, we employ (\ref{ineq: bound for Gamma function from Artin1964gamma}) to lower-bound the factorial in the binomial coefficient and obtain
\begin{equation*}
\begin{aligned}
\operatorname{II} & \le \frac{1}{\lambda} \sum_{T=0}^{\infty} \frac{\left(T+\lfloor M\rfloor\right)^{d-1}}{\sqrt{2\pi}d^{d-1/2}e^{-d}} \left(\frac{T+\lfloor M\rfloor+1}{d}\right)^{-\tilde{\kappa} (T+\lfloor M\rfloor+1)}\\
& \le \frac{e^{d}}{\sqrt{2\pi d}}\sum_{T=0}^{\infty}  \left(\frac{T+\lfloor M\rfloor+1}{d}\right)^{d-1 -\tilde{\kappa}T} .
\end{aligned}
\end{equation*}
According to $T<(d-1)/\tilde{\kappa}$ or $T\ge(d-1)/\tilde{\kappa}$, we decompose the above sum into two parts and control each with an arithmetic sequence
\begin{equation*}
\begin{aligned}
\operatorname{II} & \le \frac{e^{d}}{\sqrt{2\pi d}} \left[ \sum_{T<(d-1)/\tilde{\kappa}}\left(\frac{d-1}{\tilde{\kappa} d}+\frac{\lfloor M\rfloor+1}{d}\right)^{d-1-\tilde{\kappa} T} + \sum_{T\ge (d-1)/\tilde{\kappa}} \left(\frac{d-1}{\tilde{\kappa} d}+\frac{\lfloor M\rfloor+1}{d}\right)^{d-1-\tilde{\kappa} T}\right] \\
& \le \frac{e^{d}}{\sqrt{2\pi d}} \left(\frac{1}{\tilde{\kappa}}+\frac{3M}{2d}\right)^{d-1}\left[1-\left(\frac{1}{2\tilde{\kappa} }+1\right)^{-\tilde{\kappa}}\right]^{-1} ,
\end{aligned}
\end{equation*}
where we used $d\le M <\lfloor M\rfloor+1 \le 3M/2$ in the second inequality. Then, applying Lemma \ref{Auxiliary technical Lemma to pf Lemma: control sum series, Fourier transform decays superexponentially}, we get
\begin{equation*}
\begin{aligned}
\operatorname{II} & \le \frac{2e^{d}}{\sqrt{2\pi d}} \left(\frac{1}{\tilde{\kappa}}+\frac{3M}{2d}\right)^{d-1}\max(\tilde{\kappa}^{-1},2)
\le \frac{1}{\sqrt{\pi}} \left(\frac{1}{\tilde{\kappa}}+\frac{3}{2}\right)^{d} \left(\frac{e M}{d}\right)^{d}.
\end{aligned}
\end{equation*}
Since $d\ge2$ implies $(2\pi d)^{-1/2}+\pi^{-1/2}<1$, combining the above estimate with \eqref{ineq: control for operatornameI in pf of Lemma: control sum series, Fourier transform decays superexponentially2} completes the proof.
\end{proof}

\begin{lemma} \label{Auxiliary technical Lemma to pf Lemma: control sum series, Fourier transform decays superexponentially}
For all $\tilde{\kappa} >0$, it holds that $\left(1-2^{-\tilde{\kappa}}\right)^{-1} \le 2 \max(\tilde{\kappa}^{-1},2)$ and
\begin{equation}
\begin{aligned} \label{eq: [1-(2k/2k+1)^k]^-1<2max(k^-1,2)}
& \left[1-\left(\frac{1}{2\tilde{\kappa}}+1\right)^{-\tilde{\kappa}}\right]^{-1} \le 2\max(\tilde{\kappa}^{-1},2).
\end{aligned}
\end{equation}
\end{lemma}

\begin{proof}
Note that $\left(1-2^{-\tilde{\kappa}}\right)^{-1}$ is monotonically decreasing, which implies $\left(1-2^{-\tilde{\kappa}}\right)^{-1} \le 4$ for $\tilde{\kappa} \ge 1/2$. Then, the first inequality follows from $\tilde{\kappa} /2 \le 1-2^{-\tilde{\kappa}}$ for $\tilde{\kappa}<1/2$. For (\ref{eq: [1-(2k/2k+1)^k]^-1<2max(k^-1,2)}), the left-hand side is monotonically decreasing with respect to $\tilde{\kappa}$ because
\begin{equation*}
\begin{aligned}
& \left[\left(\frac{1}{2\tilde{\kappa}}+1\right)^{-\tilde{\kappa}}\right]^{\prime} = \left(\frac{2\tilde{\kappa}}{2\tilde{\kappa}+1}\right)^{\tilde{\kappa}} \left[\ln \frac{2\tilde{\kappa}}{2\tilde{\kappa}+1} + 1- \frac{2\tilde{\kappa}}{2\tilde{\kappa}+1} \right] \le 0.
\end{aligned}
\end{equation*}
Thus, for all $\tilde{\kappa} \ge 1/2$, the left-hand side is no more than $\left(1-2^{-1/2}\right)^{-1} < 4$.
Then, it remains to prove
\begin{equation*}
\begin{aligned}
\tilde{\kappa}^{-1} \left[1-\left(\frac{1}{2\tilde{\kappa}}+1\right)^{-\tilde{\kappa}}\right] \ge 2\left(1-2^{-1/2}\right) \ge \frac{1}{2}, \quad \text{for } 0<\tilde{\kappa} \le 1/2.
\end{aligned}
\end{equation*}
The above inequality holds because the left-hand side is a decreasing function with derivative
\begin{equation*}
\begin{aligned}
\frac{1}{\tilde{\kappa}^{2}}\left(\frac{2\tilde{\kappa}}{2\tilde{\kappa}+1}\right)^{\tilde{\kappa}}\left[-\tilde{\kappa}\ln \left(\frac{2\tilde{\kappa}}{2\tilde{\kappa}+1}\right)-\left(\frac{2\tilde{\kappa}}{2\tilde{\kappa}+1}\right)^{-\tilde{\kappa}}+\frac{\tilde{\kappa}+1}{2\tilde{\kappa}+1}\right] \le -\frac{1}{\tilde{\kappa}(2\tilde{\kappa}+1)}\left(\frac{2\tilde{\kappa}}{2\tilde{\kappa}+1}\right)^{\tilde{\kappa}}<0 ,
\end{aligned}
\end{equation*}
which completes the proof.
\end{proof}

\section{Proofs for RFM solvers}
\label{Section: proofs RFM solvers}

\subsection{Proof of the strong-form error estimate}
\label{Section: proof strong-form RFM}

The proof of Theorem~\ref{Theorem: abstract error estimate for Strong-form RFM} uses the following stability and continuity estimates.

\begin{proposition}
\label{Proposition: control H^(s+l+/2) norm of error by loss}
Under the assumptions of Theorem~\ref{Theorem: abstract error estimate for Strong-form RFM}, there exists a constant $C>0$, independent of $u$ and $u_N$, such that
\[
\inf_{v\in V} \|u-u_N-v\|_{H^{s+l+1/2}(\Omega)} \le C\sqrt{\mathcal L(u_N)} \qquad\text{for all }u_N\in H^2(\Omega).
\]
\end{proposition}

\begin{proof}
By \cite[Chapter 2, Theorems 5.4 and 7.4]{LionsMagenes1972}, the boundary value operator $w\mapsto(Lw,Bw)$ satisfies, for every $0<r\le2$,
\begin{equation}
\label{ineq: Fredholm estimate for boundary value operator w mapsto (Lw,Bw)}
\inf_{z\in V}\|w-z\|_{H^r(\Omega)}
\le C_r\left(
\|Lw\|_{\Xi^{r-2}(\Omega)}
+\|Bw\|_{H^{r-l-1/2}(\partial\Omega)}
\right),
\end{equation}
where $C_r$ is independent of $w$ and $\Xi^{r-2}(\Omega)$ denotes the interior data space in the notation of Lions--Magenes. For $0<r\le2$, the continuous embedding $L^2(\Omega)\hookrightarrow\Xi^{r-2}(\Omega)$ follows from \cite[Chapter 2, equations (6.20) and (6.22)]{LionsMagenes1972}. Taking $r=s+l+1/2$ and $w=u-u_N$ in \eqref{ineq: Fredholm estimate for boundary value operator w mapsto (Lw,Bw)} gives the result and completes the proof.
\end{proof}

\begin{lemma}
\label{Lemma: Lv and Bv bounded by W^2,p and W^1,p norm of v}
The following estimates hold.
\begin{enumerate}
\item Under the assumptions of Theorem~\ref{Theorem: abstract error estimate for Strong-form RFM},
\[
\sqrt{\mathcal L(v)} \lesssim_{\Omega,L,B,\gamma} \|u-v\|_{H^2(\Omega)} \qquad\text{for all }v\in H^2(\Omega).
\]
\item Under Assumption~\ref{Assumption: Linf norm of the coefficients of linear differential operators L, B},
\[
\|Lv\|_{L^\infty(\Omega)} +\|Bv\|_{L^\infty(\partial\Omega)} \lesssim_{\Omega,d,\{\Lambda_i\}_{i=1}^5} \|v\|_{W^{2,\infty}(\Omega)} \qquad\text{for all }v\in
W^{2,\infty}(\Omega).
\]
\end{enumerate}
\end{lemma}

The first estimate follows from the definitions of $\mathcal L$, $L$, and $B$, together with the trace theorem. The second follows directly from the coefficient bounds in Assumption~\ref{Assumption: Linf norm of the coefficients of linear differential operators L, B}.

\begin{proof}[Proof of Theorem~\ref{Theorem: abstract error estimate for Strong-form RFM}]
For any $w\in V_{\mathrm{RF}}$, the minimality of $u_{\mathrm{RF}}$ gives $\mathcal L(u_{\mathrm{RF}})\le\mathcal L(w)$. Proposition~\ref{Proposition: control H^(s+l+/2) norm of error by loss} and Lemma~\ref{Lemma: Lv and Bv bounded by W^2,p and W^1,p norm of v}(1) therefore yield
\[
\inf_{v\in V} \|u-u_{\mathrm{RF}}-v\|_{H^{s+l+1/2}(\Omega)} \le C\sqrt{\mathcal L(w)} \le C\|u-w\|_{H^2(\Omega)}.
\]
Taking the infimum over $w\in V_{\mathrm{RF}}$ proves the theorem and completes the proof.
\end{proof}

\section{Proofs for singular-value estimates and condition-number lower bounds}
\label{Section: proofs for singular value estimates and condition number lower bounds}

This appendix collects the auxiliary estimates and the proofs of Theorems~\ref{Theorem: fast decay of singular values of RFMtx} and \ref{Theorem: exponential ill conditionality of random feature matrices}.

\subsection*{Auxiliary lemmas}

\begin{lemma} \label{Lemma: equivalent norm used in pf of exponential ill conditionality}
Let $s \in \mb{N}_{+}$, $1\le p\le\infty$ and $\Omega=(-1,1)^d$. Denote $$\varsigma_{1}(x) = \prod_{i=1}^{d}\frac{2}{\pi x_{i}}\sin\left(\frac{\pi
x_{i}}{2}\right), \quad \text{and }\quad \varsigma_{2}(x) = e^{-|x|^{2}/4}$$. Then, $\|v\|_{W^{s,p}(\Omega)} \lesssim_{d,s}
\|\varsigma_{j}v\|_{W^{s,p}(\Omega)}$ for all $v \in W^{s,p}(\Omega)$ and $j=1,2$.
\end{lemma}
\begin{proof}
Since $\sin(z)/z$ and $z/\sin(z)$ are analytic on $\{z\in\mb{C}:|z|\le 3\pi/4\}$, we have $\varsigma_{1}, 1/\varsigma_{1} \in C^{\infty}(\Omega)$. For $\varsigma_{2}$, it also holds that $e^{-|x|^{2}/4}, e^{|x|^{2}/4}\in C^{\infty}(\Omega)$. Denote $w(x) = \varsigma_{j}(x)v(x)$. Then $w \in W^{s,p}(\Omega)$
because $$\|w\|_{W^{s,p}(\Omega)} \lesssim_{d,s} \|\varsigma_{j}\|_{W^{s,\infty}(\Omega)} \|v\|_{W^{s,p}(\Omega)}$$. Since $\Omega$ is bounded, $1/\varsigma_{j}
\in W^{s,\infty}(\Omega)$ and $$\|v\|_{W^{s,p}(\Omega)} = \|w/\varsigma_{j}\|_{W^{s,p}(\Omega)} \lesssim_{d,s} \|1/\varsigma_{j}\|_{W^{s,\infty}(\Omega)}
\|w\|_{W^{s,p}(\Omega)}$$, which completes the proof.
\end{proof}

\begin{lemma}
\label{Lemma: Fourier decay of Gaussian weighted tanh features}
Let $\varsigma_{2}(x) = e^{-|x|^{2}/4}$, $S>0$ and $\kappa = \min(1,11\pi/(96\sqrt{d}S))$. Then,
\begin{equation*}
\begin{aligned}
& \|\varsigma_{2}\tanh(k^{\top}\cdot + v)\|_{\kappa,1} \lesssim_{d} 1 + S^{d}, \quad \text{ for $k \in [-S,S]^{d}$ and $v\in\mb{R}$}.
\end{aligned}
\end{equation*}
\end{lemma}

\begin{proof}
Denote $u:=\varsigma_{2}\tanh(k^{\top}\cdot+v)$. The case $k=0$ follows directly from the Gaussian Fourier transform, so we assume $k\ne0$ below. The function
$u$ has an analytic continuation to the set $\{\zeta\in \mb{C}^{d} :|k^{\top}\operatorname{Im} \zeta|<\pi/2\}$ as $$u(\lambda+i
\eta)=\left[1-\frac{2}{e^{2\left(k^{\top}\lambda+v+i k^{\top} \eta\right)}+1}\right] e^{-\left(|\lambda|^{2}-|\eta|^{2}+2 i \lambda^{\top} \eta\right) / 4}, $$
for $\lambda,\eta\in\mb{R}^{d}$ and $|k^{\top}\eta|<\pi/2$. For each such $\eta$, $u(\lambda+i \eta)$ is in $\ms{S}'(\mb{R}^{d})$ as a function of $\lambda$. By Cauchy's theorem, we can shift the region of integration in Fourier transform so that
\begin{equation} \label{eq: Fourier of varsigma_2 tanh(k cdot + v)}
\begin{aligned}
\wh{u}(\xi)=(2 \pi)^{-d / 2} \int_{\mathbb{R}^{d}} u(\lambda+i \eta) e^{-i \xi^{\top}(\lambda+i \eta)} d \lambda .
\end{aligned}
\end{equation}
For $|k^{\top}\eta|\le 11\pi/24$, note that
\begin{equation*}
\begin{aligned}
\left|\tanh(k^{\top}\lambda+v + i k^{\top}\eta)\right| & = \left|\frac{e^{4(k^{\top}\lambda+v)} - 1 + 2ie^{2(k^{\top}\lambda+v)} \sin(2 k^{\top}\eta)}{e^{4(k^{\top}\lambda+v)} + 1+ 2e^{2(k^{\top}\lambda+v)} \cos(2 k^{\top}\eta)}\right| \\
& \le \frac{e^{4(k^{\top}\lambda+v)} + 1}{e^{4(k^{\top}\lambda+v)} + 1+ 2e^{2(k^{\top}\lambda+v)} \cos(2 k^{\top}\eta)} \\
& \le 1 - \frac{ \cos(11\pi/12)}{1 + \cos(11\pi/12)} \le 30 .\\
\end{aligned}
\end{equation*}
Then, we take $\eta = -2\xi$ for $|\xi| \le \frac{11\pi}{48|k|}$ and $\eta = -\frac{11\pi\xi}{24|k| |\xi|}$ otherwise in (\ref{eq: Fourier of varsigma_2 tanh(k cdot + v)}), which gives
\begin{equation*}
\begin{aligned}
|\wh{u}(\xi)| & \le 30 \cdot 2^{d / 2} \left\{\begin{array}{ll}
e^{-|\xi|^{2}}, & |\xi| \le \frac{11\pi}{48|k|},\\
e^{\left(\frac{11\pi}{48|k|}\right)^{2} -\frac{11\pi|\xi|}{24|k|} }, & |\xi| \ge \frac{11\pi}{48|k|}.
\end{array}\right.
\end{aligned}
\end{equation*}
Since $\kappa = \min(1,11\pi/(96\sqrt{d}S))$, we have $\kappa\le 11\pi/(96|k|)$. By direct calculation,
\begin{equation*}
\begin{aligned}
\|u\|_{\kappa,1}^{2} & \lesssim_{d}  \int_{0}^{\frac{11\pi}{48|k|}} r^{d-1}e^{-2r^{2}+2\kappa r}\mr{d}r + \int_{\frac{11\pi}{48|k|}}^{\infty} r^{d-1} e^{2\left(\frac{11\pi}{48|k|}\right)^{2} - \frac{11\pi r}{12|k|} +2\kappa r}\mr{d}r \\
& \lesssim_{d}  \int_{0}^{\infty} e^{-2r^{2}+(2\kappa+1) r}\mr{d}r + e^{-\left(\frac{11\pi}{48|k|}\right)^{2}} \int_{0}^{\infty} \left(r+\frac{11\pi}{48|k|}\right)^{d-1} e^{- \frac{11\pi}{16|k|}r}\mr{d}r \\
& \lesssim_{d}  e^{(2\kappa+1)^{2}/8} + 1 + |k|^{d}, \\
\end{aligned}
\end{equation*}
which completes the proof.
\end{proof}

\begin{lemma}[Low-frequency Taylor approximation for Fourier features]
\label{Lemma: low frequency Taylor approximation for Fourier features}
For every $T\ge2$ and $0<|k|<(2\sqrt d)^{-1}$, the polynomial
\[
P_T(k,x):=\sum_{\ell=0}^{T}\frac{(\imath k^{\top}x)^\ell}{\ell!}
\]
satisfies
\[
\left\| e^{\imath k^{\top}(\cdot)}-P_T(k,\cdot) \right\|_{W^{2,\infty}([-1,1]^d)} \lesssim_d |k|^2\frac{2^{-(T-1)}}{(T-1)!}.
\]
The same estimate holds for the real and imaginary parts.
\end{lemma}

\begin{proof}
Let $\alpha$ be a multi-index with $r:=|\alpha|\le2$. Since $\partial^\alpha P_T(k,x)=(\imath k)^\alpha P_{T-r}(k,x)$, the exponential-series remainder gives
\[
\left| \partial^\alpha\bigl(e^{\imath k^{\top}x}-P_T(k,x)\bigr) \right| \le |k|^r e^{|k^{\top}x|} \frac{|k^{\top}x|^{T-r+1}}{(T-r+1)!}.
\]
For $x\in[-1,1]^d$ and $|k|<(2\sqrt d)^{-1}$, one has $|k^{\top}x|\le\sqrt d\,|k|<1/2$. Consequently,
\[
|k|^r e^{|k^{\top}x|} \frac{|k^{\top}x|^{T-r+1}}{(T-r+1)!} \le d^{(2-r)/2}e^{1/2}|k|^2 \frac{(\sqrt d\,|k|)^{T-1}}{(T-1)!} \lesssim_d |k|^2
\frac{2^{-(T-1)}}{(T-1)!}.
\]
Taking the maximum over $|\alpha|\le2$ proves the asserted bound. Taking real or imaginary parts cannot increase the norm, which completes the proof.
\end{proof}

\begin{lemma}[Low-frequency Taylor approximation for $\tanh$ features]
\label{Lemma: low frequency Taylor approximation for tanh features}
Set $\rho_{\mathrm T}:=4/(7\sqrt d)\in(0,1)$. For every $T\ge2$, $0<|k|<\rho_{\mathrm T}$, and $|v|\le dS$, the polynomial
\[
P_{T,k,v}(x):= \sum_{\ell=0}^{T}\frac{\tanh^{(\ell)}(v)}{\ell!}(k^{\top}x)^\ell
\]
satisfies
\begin{equation}
\label{ineq: low frequency tanh Taylor remainder}
\left\|\tanh(k^{\top}\cdot+v)-P_{T,k,v}\right\|_{W^{2,\infty}([-1,1]^d)}
\lesssim_{d,S}
|k|^2\left(\frac{|k|}{2\rho_{\mathrm T}}\right)^{T-1}.
\end{equation}
\end{lemma}

\begin{proof}
For $v\in[-dS,dS]$, the function $w\mapsto\tanh(v+w)$ is analytic in $|w|<\pi/2$. Cauchy's estimate on the circle $|w|=4/3$ gives
\[
\frac{|\tanh^{(\ell)}(v)|}{\ell!} \lesssim_{d,S}\left(\frac34\right)^\ell.
\]
Since $\sup_{\ell\ge0}(1+\ell)^2(6/7)^\ell<\infty$, we have
\[
\frac{|\tanh^{(\ell)}(v)|}{\ell!}\,\ell^r \lesssim_{d,S}\left(\frac78\right)^\ell, \qquad \ell\ge r,\quad 0\le r\le2.
\]
Let $\alpha$ be a multi-index with $r:=|\alpha|\le2$ and put $s:=7\sqrt d\,|k|/8=|k|/(2\rho_{\mathrm T})<1/2$. Termwise differentiation of the Taylor series and $|k^{\top}x|\le\sqrt d\,|k|$ yield, uniformly for $x\in[-1,1]^d$,
\begin{align*}
\left|\partial^\alpha
\bigl(\tanh(k^{\top}x+v)-P_{T,k,v}(x)\bigr)\right|
&\lesssim_{d,S}|k|^r\sum_{\ell=T+1}^{\infty}s^{\ell-r}\\
&\lesssim_{d,S}|k|^r s^{T+1-r}\\
&\lesssim_{d,S}|k|^2s^{T-1}.
\end{align*}
The implicit constant in the last step is uniform for $0\le r\le2$ because $s/|k|=(2\rho_{\mathrm T})^{-1}$. Taking the maximum over $|\alpha|\le2$ proves \eqref{ineq: low frequency tanh Taylor remainder} and completes the proof.
\end{proof}

\subsection*{Proof of Theorem~\ref{Theorem: fast decay of singular values of RFMtx}}

\begin{proof}
Put $K:=m-1$. The hypothesis on $m$ gives $K\ge3M\ln(15M)$. Note that $\sigma_{m}(\mf{\Psi}) \leq \|\mf{\Psi}\|_{F} /\sqrt{m}$ always holds for any matrix and $m$. We shall prove the upper bound on $\sigma_{m}$ for $m \geq 3$. By the minimax principle,
\begin{equation*}
\begin{aligned}
\sigma_{m} = \min _{\operatorname{dim} H=2N-m+1} \max _{\boldsymbol{\beta} \in H} \frac{|\mf{\Psi} \boldsymbol{\beta}|}{|\boldsymbol{\beta}|}  ,
\end{aligned}
\end{equation*}
where $H$ is a linear subspace of $\mb{R}^{2N}$. To control $\sigma_{m}$, it suffices to find a $(2N-m+1)$-dimensional linear subspace on which $|\mf{\Psi}\boldsymbol\beta|$ is small relative to $|\boldsymbol\beta|$. We reduce this problem to mutual approximation of the features $\psi_j$ in $W^{2,\infty}(\Omega)$, thereby minimizing the dependence on the collocation points, domain, and differential equation.

By Lemma \ref{Lemma: Lv and Bv bounded by W^2,p and W^1,p norm of v} and the trace theorem, we obtain
\begin{equation*}
\begin{aligned}
|\mf{\Psi} \beta|^{2} & \le n_{1}\left\|\sum_{j=1}^{2 N} \beta_{j}L\psi_{j}\right\|_{L^{\infty}(\Omega)}^{2} + n_{2} \left\|\sum_{j=1}^{2 N} \beta_{j}B\psi_{j}\right\|_{L^{\infty}(\partial\Omega)}^{2} \\
& \lesssim_{\Omega,d,\{\Lambda_{i}\}_{i=1}^{5}} n_{1} \left\|\sum_{j=1}^{2 N} \beta_{j}\psi_{j}\right\|_{W^{2, \infty}(\Omega)}^{2} +  n_{2}  \left\|\sum_{j=1}^{2 N} \beta_{j}\psi_{j}\right\|_{W^{1, \infty}(\partial\Omega)}^{2} \\
& \lesssim_{\Omega,d,\{\Lambda_{i}\}_{i=1}^{5}} n   \left\|\sum_{j=1}^{2 N} \beta_{j}\psi_{j}\right\|_{W^{2, \infty}(\Omega)}^{2} . \\
\end{aligned}
\end{equation*}
Denote $\widetilde{\Omega}=(-1,1)^d$. Then, it follows from $\Omega\subset\widetilde{\Omega}$ and Lemma \ref{Lemma: equivalent norm used in pf of exponential ill conditionality} that
\begin{equation*}
\begin{aligned}
|\mf{\Psi} \beta| & \lesssim_{\Omega,d,\{\Lambda_{i}\}_{i=1}^{5}} \sqrt{n}  \left\|\sum_{j=1}^{2 N} \beta_{j}\varsigma\psi_{j}\right\|_{W^{2, \infty}(\widetilde{\Omega})}  ,
\end{aligned}
\end{equation*}
where $\varsigma$ can be $1$, $\varsigma_{1}$ or $\varsigma_{2}$ as in Lemma \ref{Lemma: equivalent norm used in pf of exponential ill conditionality}.
For each $j$, suppose $\varphi_{j}$ is an approximation of $\varsigma\psi_{j}$ in some linear subspace of dimension at most $m-1$. Then there exists a $(2N-m+1)$-dimensional subspace $H \subset\mb{R}^{2N}$ such that
\begin{equation*}
\sum_{j=1}^{2 N} \beta_{j} \varphi_{j}(x) = 0, \quad \text{for any } \beta\in H.
\end{equation*}
Using the above relation, the triangle inequality, and the Cauchy--Schwarz inequality, we obtain
\begin{equation*}
\begin{aligned}
\left\|\sum_{j=1}^{2 N} \beta_{j}\varsigma\psi_{j} \right\|_{W^{2, \infty}(\widetilde{\Omega})} & = \left\|\sum_{j=1}^{2 N} \beta_{j}\left(\varsigma\psi_{j}- \varphi_{j}\right) \right\|_{W^{2, \infty}(\widetilde{\Omega})} \\
& \le |\beta|\left(\sum_{j=1}^{2 N}  \left\|\varsigma\psi_{j}- \varphi_{j} \right\|_{W^{2, \infty}(\widetilde{\Omega})}^{2} \right)^{1/2} ,
\end{aligned}
\end{equation*}
and therefore
\begin{equation}
\label{ineq: bound matrix beta by W2infty approximation errors}
\sigma_{m}
\lesssim_{\Omega,d,\{\Lambda_{i}\}_{i=1}^{5}}
\sqrt{n}\left(\sum_{j=1}^{2 N}
\left\|\varsigma\psi_{j}-\varphi_{j}\right\|_{W^{2,\infty}(\widetilde\Omega)}^{2}
\right)^{1/2}.
\end{equation}
So far, we have reduced bounding $\sigma_{m}$ to the approximation to $\varsigma\psi_{j}$ by $\varphi_{j}$. We shall choose appropriate $\varsigma$ and $\varphi_{j}$ in each case.

It remains to construct the approximants $\varphi_j$. Fix $\delta_0\in(14/15,1)$. Since $14/\delta_0<15$, the assumed lower bound on $K$ implies the sampling condition in Theorem~\ref{Theorem: interpolation improved convergence rate of RFM} with $N=K$ and $\delta=\delta_0$.

For Fourier features, take $\varsigma=\varsigma_1$ and set $S_{\mathrm F}:=S+\pi/2$. Direct calculation gives
\begin{equation*}
\begin{aligned}
& \left(\varsigma_{1} \cos(k^{\top} \cdot)\right)^{\wedge}(\xi) = \frac{(2/\pi)^{d/2}}{2}\left(1_{|\xi+k|_{\infty}< \frac{\pi}{2}} + 1_{|\xi-k|_{\infty}< \frac{\pi}{2}}\right),  \\
& \left(\varsigma_{1} \sin(k^{\top} \cdot)\right)^{\wedge}(\xi) = \frac{(2/\pi)^{d/2}}{2 i}\left(1_{|\xi+k|_{\infty}< \frac{\pi}{2}} - 1_{|\xi-k|_{\infty}< \frac{\pi}{2}}\right) .
\end{aligned}
\end{equation*}
Hence both functions are supported in $\overline{Q_{S_{\mathrm F}}}$ in frequency and have $L^2(\mb R^d)$ norms bounded by a constant depending only on $d$. Put
\[
a_{\lambda,\mathrm F}:=\frac{2^{1-1/d}}{28e}, \qquad a_{\mathrm F}:=\frac34a_{\lambda,\mathrm F} =\frac{3\cdot 2^{1-1/d}}{112e}.
\]
Apply Theorem~\ref{Theorem: interpolation improved convergence rate of RFM}(4) on $\widetilde\Omega$ using the randomly shifted cosine representation and taking $t=2$ and $p=\infty$. With $\lambda_{\mathrm F}=c_{\lambda,\mathrm F} \exp\bigl(-a_{\lambda,\mathrm F}M^{1/d}\ln M\bigr)$, the theorem supplies one set of $K$ sampled frequency--phase pairs and, simultaneously for all $2N$ target functions, approximants $\varphi_j$ in the resulting $K$-dimensional real trial space such that
\[
\|\varsigma_1\psi_j-\varphi_j\|_{W^{2,\infty}(\widetilde\Omega)} \lesssim_{d,S} \exp\bigl(-a_{\mathrm F}M^{1/d}\ln M\bigr), \qquad 1\le j\le2N.
\]
The event has probability at least $1-\delta_0>0$ and is uniform over the target ball. Since the features and targets are real, taking the real parts of the coefficients does not increase either the approximation error or the coefficient norm. We therefore fix one realization for which these estimates hold. Since $K=m-1$, substituting these estimates into \eqref{ineq: bound matrix beta by W2infty approximation errors} proves \eqref{ineq: Fourier fast decay of singular values}.

For $\tanh$ features, take $\varsigma=\varsigma_2$ and put
\[
\kappa_{\mathrm T}:=\min\left\{1, \frac{11\pi}{96\sqrt d\,S}\right\}.
\]
Lemma~\ref{Lemma: Fourier decay of Gaussian weighted tanh features} yields, uniformly for $k_j\in[-S,S]^d$ and $v_j\in\mb R$, $\|\varsigma_2\psi_j\|_{\kappa_{\mathrm T},1}\lesssim_d 1+S^d$. Apply Theorem~\ref{Theorem: interpolation improved convergence rate of RFM}(2) with the randomly shifted cosine representation, $s=1$, and $\kappa=\bar\kappa=\kappa_{\mathrm T}$, taking $t=2$ and $p=\infty$. Thus $\sigma=1$ and $\theta=1/2$. Since $\kappa_{\mathrm T}$ depends only on $d$ and $S$, and since the application domain $\widetilde\Omega=(-1,1)^d$ is fixed, the theorem yields a constant $a_{\mathrm T}=a_{\mathrm T}(d,S)>0$. The same uniform-event argument supplies approximants in a common $K$-dimensional real trial space satisfying
\[
\|\varsigma_2\psi_j-\varphi_j\|_{W^{2,\infty}(\widetilde\Omega)} \lesssim_{d,S}\exp\bigl(-a_{\mathrm T}M^{1/d}\bigr), \qquad 1\le j\le2N.
\]
Substitution into \eqref{ineq: bound matrix beta by W2infty approximation errors} proves \eqref{ineq: tanh fast decay of singular values} and completes the proof.
\end{proof}

\subsection*{Proof of Theorem~\ref{Theorem: exponential ill conditionality of random feature matrices}}

\begin{proof}
If $\sigma_{2N}=0$, the conclusion follows from the stated convention. Assume henceforth that $\sigma_{2N}>0$. Theorem \ref{Theorem: fast decay of singular values of RFMtx} applies with $m=2N$ because $2N-1\ge3M\ln(15M)$. Thus \eqref{ineq: Fourier fast decay of singular values} and \eqref{ineq: tanh fast decay of singular values}, with $m=2N$, apply in the Fourier and $\tanh$ cases, respectively.

We first prove part~(1). For $1\le i\le n_1$, ellipticity gives $-k_j^{\top}A(x_i)k_j\ge\underline a|k_j|^2$. Direct differentiation and cancellation of the cross terms in each cosine--sine pair give
\begin{equation}
\label{eq: Fourier pair identities}
\begin{aligned}
&|L\cos(k_j^{\top}\cdot)(x_i)|^2
+|L\sin(k_j^{\top}\cdot)(x_i)|^2
=\bigl(-k_j^{\top}A(x_i)k_j+c(x_i)\bigr)^2
+\bigl(b(x_i)^{\top}k_j\bigr)^2,
\qquad 1\le i\le n_1,\\
&|B\cos(k_j^{\top}\cdot)(x_i)|^2
+|B\sin(k_j^{\top}\cdot)(x_i)|^2
=g_1^2(x_i)\bigl(\mf n(x_i)^{\top}k_j\bigr)^2+g_2^2(x_i),
\quad n_1<i\le n.
\end{aligned}
\end{equation}
Set $\mc H:=\{j:|k_j|\ge(2\sqrt d)^{-1}\}$ and $\mc L:=\{j:|k_j|<(2\sqrt d)^{-1}\}$.

Suppose first that $|\mc H|\ge N/2$. Summing the two identities in \eqref{eq: Fourier pair identities} and using $c\ge0$ yields
\begin{align*}
\|\mf\Psi\|_{\mathrm F}^2
&\ge \frac{|\mc H|n_1\underline a^2}{16d^2}
+Nn_1C^2+Nn_2G^2\\
&\gtrsim_d
N\bigl(n_1\underline a^2+n_1C^2+n_2G^2\bigr).
\end{align*}
Since $\sigma_1\ge\|\mf\Psi\|_{\mathrm F}/\sqrt{2N}$,
\[
\sigma_1\gtrsim_d \sqrt{n_1}\,\underline a+\sqrt{n_1}C+\sqrt{n_2}G.
\]
Combining this estimate with \eqref{ineq: Fourier fast decay of singular values} proves part~(1) in this case.

It remains to consider $|\mc L|>N/2$. If $k_j=0$ for some $j\in\mc L$, the associated sine column is zero, contrary to $\sigma_{2N}>0$. Hence every low frequency is nonzero. Put
\[
\rho_*:=\max_{j\in\mc L}|k_j|\in\bigl(0,(2\sqrt d)^{-1}\bigr)
\]
and choose $j_*$ attaining the maximum. By \eqref{eq: Fourier pair identities},
\[
\|\mf\Psi_{\cdot,j_*}\|_2^2 +\|\mf\Psi_{\cdot,j_*+N}\|_2^2 \ge n_1\underline a^2\rho_*^4+n_1C^2+n_2G^2.
\]
Thus at least one of these columns has norm at least the square root of the right-hand side divided by $\sqrt2$, and therefore
\begin{equation}
\label{ineq: largest singular value low Fourier frequencies}
\sigma_1\gtrsim
\sqrt{n_1}\,\underline a\,\rho_*^2
+\sqrt{n_1}C+\sqrt{n_2}G.
\end{equation}

We next construct a vector that provides an upper bound for $\sigma_{2N}$. Let
\begin{equation}
\label{eq: polynomial degree for condition number lower bounds}
T:=\left\lfloor\frac{d}{4e}M^{1/d}\right\rfloor,
\qquad
D_T:=\dim\mc P_T(\mb R^d)=\binom{T+d}{d},
\end{equation}
where $\mc P_T(\mb R^d)$ is the space of real polynomials of total degree at most $T$. For all sufficiently large $M$, one has $d\le dM^{1/d}/(4e)$, and hence $D_T\le(e(T+d)/d)^d\le 2^{-d}M\le M<N<2|\mc L|$. The coefficient bounds and the trace estimate used in the proof of Theorem~\ref{Theorem: fast decay of singular values of RFMtx} give, for every $f\in W^{2,\infty}((-1,1)^d)$,
\begin{equation}
\label{ineq: collocation evaluation continuity for condition number lower bounds}
\left|\left(\widetilde Lf(x_i)\right)_{i=1}^n\right|
\lesssim_{\Omega,d,\{\Lambda_i\}_{i=1}^5}
\sqrt n\,\|f\|_{W^{2,\infty}((-1,1)^d)}.
\end{equation}
For $j\in\mc L$, set
\[
p_{j,c}:=|k_j|^{-2}\operatorname{Re}P_T(k_j,\cdot), \qquad p_{j,s}:=|k_j|^{-2}\operatorname{Im}P_T(k_j,\cdot).
\]
The $2|\mc L|$ vectors $(\widetilde Lp_{j,c}(x_i))_{i=1}^n$ and $(\widetilde Lp_{j,s}(x_i))_{i=1}^n$ span a subspace of $\mb R^n$ of dimension at most $D_T<2|\mc L|$. Hence there exists $\gamma=(\gamma_{j,c},\gamma_{j,s})_{j\in\mc L} \in\mb R^{2|\mc L|}$ with $|\gamma|=1$ such that
\begin{equation}
\label{eq: annihilation of normalized Fourier Taylor polynomials}
\sum_{j\in\mc L}
\left[
\gamma_{j,c}(\widetilde Lp_{j,c}(x_i))_{i=1}^n
+\gamma_{j,s}(\widetilde Lp_{j,s}(x_i))_{i=1}^n
\right]=0.
\end{equation}
Define $\beta\in\mb R^{2N}$ by $\beta_j:=\gamma_{j,c}/|k_j|^2$ and $\beta_{j+N}:=\gamma_{j,s}/|k_j|^2$ for $j\in\mc L$, and set all remaining components equal to zero. Then
\begin{equation}
\label{ineq: normalized Fourier coefficient lower bound}
|\beta|^2
=\sum_{j\in\mc L}
\frac{|\gamma_{j,c}|^2+|\gamma_{j,s}|^2}{|k_j|^4}
\ge\rho_*^{-4}.
\end{equation}
Using \eqref{eq: annihilation of normalized Fourier Taylor polynomials}, Lemma~\ref{Lemma: low frequency Taylor approximation for Fourier features}, \eqref{ineq: collocation evaluation continuity for condition number lower bounds}, and Cauchy--Schwarz, we obtain
\begin{align*}
|\mf\Psi\beta|
&=\left|\left(\widetilde L\!\left\{
\sum_{j\in\mc L}\left[
\gamma_{j,c}\frac{\cos(k_j^{\top}\cdot)-\operatorname{Re}P_T(k_j,\cdot)}
{|k_j|^2}
+\gamma_{j,s}\frac{\sin(k_j^{\top}\cdot)-\operatorname{Im}P_T(k_j,\cdot)}
{|k_j|^2}\right]
\right\}(x_i)\right)_{i=1}^n\right|\\
&\lesssim_{\Omega,d,\{\Lambda_i\}_{i=1}^5}
\sqrt n\,\frac{2^{-(T-1)}}{(T-1)!}
\sum_{j\in\mc L}(|\gamma_{j,c}|+|\gamma_{j,s}|)\\
&\lesssim_{\Omega,d,\{\Lambda_i\}_{i=1}^5}
\sqrt{nN}\,\frac{2^{-(T-1)}}{(T-1)!}.
\end{align*}
Together with \eqref{ineq: normalized Fourier coefficient lower bound}, this gives
\begin{equation}
\label{ineq: smallest singular value low Fourier frequencies}
\sigma_{2N}\le\frac{|\mf\Psi\beta|}{|\beta|}
\lesssim_{\Omega,d,\{\Lambda_i\}_{i=1}^5}
\rho_*^2\sqrt{nN}\,\frac{2^{-(T-1)}}{(T-1)!}.
\end{equation}
Combining \eqref{ineq: largest singular value low Fourier frequencies} and \eqref{ineq: smallest singular value low Fourier frequencies}, and using $\rho_*<1$, yields
\begin{equation}
\label{ineq: Fourier condition number before factorial estimate}
\frac{\sigma_1}{\sigma_{2N}}
\gtrsim_{\Omega,d,\{\Lambda_i\}_{i=1}^5}
\frac{\sqrt{n_1}\,\underline a+\sqrt{n_1}C+\sqrt{n_2}G}
{\sqrt{nN}}\,2^{T-1}(T-1)!.
\end{equation}

It remains to compare the factorial decay with the explicit rate constant $a_{\mathrm F}$. The elementary bound $(T-1)!\ge((T-1)/e)^{T-1}$ gives
\[
\ln\bigl(2^{T-1}(T-1)!\bigr) \ge (T-1)\ln\left(\frac{2(T-1)}e\right).
\]
For all sufficiently large $M$, one has
\[
T-1\ge\frac{3d}{16e}M^{1/d}, \qquad \ln\left(\frac{2(T-1)}e\right) \ge\frac{3}{4d}\ln M.
\]
It follows that
\[
\ln\bigl(2^{T-1}(T-1)!\bigr) \ge\frac{9}{64e}M^{1/d}\ln M.
\]
Since
\[
a_{\mathrm F}=\frac{3\cdot 2^{1-1/d}}{112e} \le\frac{3}{56e}<\frac{9}{64e},
\]
we have $2^{T-1}(T-1)!\ge \exp(a_{\mathrm F}M^{1/d}\ln M)$ for all sufficiently large $M$. Substitution into \eqref{ineq: Fourier condition number before factorial estimate} proves part~(1) in the low-frequency case.

We now prove part~(2). The argument first converts Assumption~\eqref{assump: uniform directional spread of interior collocation points} into a lower bound for individual columns. For any $k\ne0$ and $v\in\mb R$, the choice $\theta=k/|k|$ gives
\begin{equation}
\label{ineq: directional spread controls tanh preactivation}
\begin{aligned}
\max_{1\le i\le n_1}|k^{\top}x_i+v|
&\ge\frac12\left(
\max_{1\le i\le n_1}(k^{\top}x_i+v)
-\min_{1\le i\le n_1}(k^{\top}x_i+v)
\right)\\
&\ge r_0|k|.
\end{aligned}
\end{equation}
Since $k_j\in[-S,S]^d$, $|v_j|\le dS$, and $x_i\in(-1,1)^d$, all preactivations $z_{ij}:=k_j^{\top}x_i+v_j$ satisfy $|z_{ij}|\le2dS$. The function $|\tanh''z|/|z|$, initially defined for $z\ne0$, extends continuously to $z=0$ with value $2$ and is positive on $[-2dS,2dS]$. Consequently,
\[
\mu_{\mathrm T}:= \min_{|z|\le2dS}\frac{|\tanh''z|}{|z|}>0,
\]
where the quotient at the origin is understood by continuity.

Put $q_{ij}:=-k_j^{\top}A(x_i)k_j\ge\underline a|k_j|^2$. Because $b\equiv0$ and $\tanh''z=-2\tanh z\,\operatorname{sech}^2z$,
\[
L\psi_j(x_i) =\bigl(2q_{ij}\operatorname{sech}^2z_{ij}+c(x_i)\bigr) \tanh z_{ij}.
\]
The two terms on the right have the same sign. Hence
\[
|L\psi_j(x_i)| \ge q_{ij}|\tanh''z_{ij}| \ge\underline a|k_j|^2|\tanh''z_{ij}|.
\]
Combining this estimate with \eqref{ineq: directional spread controls tanh preactivation} shows that every column satisfies
\begin{equation}
\label{ineq: tanh column lower bound from directional spread}
\max_{1\le i\le n_1}|L\psi_j(x_i)|
\ge\mu_{\mathrm T}r_0\underline a|k_j|^3.
\end{equation}

Let $\rho_{\mathrm T}:=4/(7\sqrt d)$ be as in Lemma~\ref{Lemma: low frequency Taylor approximation for tanh features}, and set $\mc L_{\mathrm T}:=\{j:|k_j|<\rho_{\mathrm T}\}$ and $\mc H_{\mathrm T}:=\{j:|k_j|\ge\rho_{\mathrm T}\}$. By \eqref{ineq: directional spread controls tanh preactivation}, every $j\in\mc H_{\mathrm T}$ has an interior preactivation of magnitude at least $r_0\rho_{\mathrm T}$. We distinguish two exhaustive cases.

Suppose first that $|\mc L_{\mathrm T}|<N$. Then $|\mc H_{\mathrm T}|>N$, and \eqref{ineq: tanh column lower bound from directional spread} yields
\[
\|\mf\Psi\|_{\mathrm F}^2 \ge\sum_{j\in\mc H_{\mathrm T}} \max_{1\le i\le n_1}|L\psi_j(x_i)|^2 \ge N\mu_{\mathrm T}^2r_0^2\underline a^2\rho_{\mathrm T}^6.
\]
It follows that
\[
\sigma_1\ge\frac{\|\mf\Psi\|_{\mathrm F}}{\sqrt{2N}} \ge\frac{\mu_{\mathrm T}r_0\rho_{\mathrm T}^3}{\sqrt2}\,\underline a.
\]
Together with \eqref{ineq: tanh fast decay of singular values} and $\widetilde a_{\mathrm T}\le a_{\mathrm T}$, this proves \eqref{ineq: tanh condition number lower bound} in the high-frequency case.

It remains to consider $|\mc L_{\mathrm T}|\ge N$. Let $T$ and $D_T$ be as in \eqref{eq: polynomial degree for condition number lower bounds}. The estimate $D_T\le M$ proved above and the sampling condition imply $D_T<N-1$ for all sufficiently large $M$. If two indices in $\mc L_{\mathrm T}$ had zero frequency, their features would be constant and the associated columns would be linearly dependent. A zero frequency with $v_j=0$ would itself give a zero column. Both alternatives contradict $\sigma_{2N}>0$. Thus the set
\[
\mc L_{\mathrm T}^*:=\{j\in\mc L_{\mathrm T}:|k_j|>0\}
\]
has cardinality at least $N-1>D_T$.

For $j\in\mc L_{\mathrm T}^*$, let $P_j:=P_{T,k_j,v_j}$ be the polynomial from Lemma~\ref{Lemma: low frequency Taylor approximation for tanh features}. Since all $P_j$ belong to the $D_T$-dimensional space $\mc P_T(\mb R^d)$, there is a vector $\gamma=(\gamma_j)_{j\in\mc L_{\mathrm T}^*}$ such that
\[
|\gamma|=1, \qquad \sum_{j\in\mc L_{\mathrm T}^*}\gamma_jP_j=0.
\]
Extend $\gamma$ by zero to a vector in $\mb R^{2N}$ and define
\[
\rho_\gamma:=\max_{\gamma_j\ne0}|k_j|, \qquad \rho_*:=\max_{j\in\mc L_{\mathrm T}^*}|k_j|.
\]
Then $0<\rho_\gamma\le\rho_*<\rho_{\mathrm T}$. By \eqref{ineq: collocation evaluation continuity for condition number lower bounds}, Lemma~\ref{Lemma: low frequency Taylor approximation for tanh features}, and Cauchy--Schwarz,
\begin{align*}
\sigma_{2N}
&\le|\mf\Psi\gamma|
=\left|\left(\widetilde L\!\left\{
\sum_{j\in\mc L_{\mathrm T}^*}\gamma_j(\psi_j-P_j)
\right\}(x_i)\right)_{i=1}^n\right|\\
&\lesssim_{\Omega,d,\{\Lambda_i\}_{i=1}^5,S}
\sqrt n\sum_{j\in\mc L_{\mathrm T}^*}
|\gamma_j||k_j|^2
\left(\frac{|k_j|}{2\rho_{\mathrm T}}\right)^{T-1}\\
&\lesssim_{\Omega,d,\{\Lambda_i\}_{i=1}^5,S}
\sqrt{nN}\,\rho_\gamma^2
\left(\frac{\rho_\gamma}{2\rho_{\mathrm T}}\right)^{T-1}.
\end{align*}
On the other hand, \eqref{ineq: tanh column lower bound from directional spread} gives $\sigma_1\ge\mu_{\mathrm T}r_0\underline a\rho_*^3$. Since $\rho_*\ge\rho_\gamma$, we conclude that
\begin{align*}
\frac{\sigma_1}{\sigma_{2N}}
&\gtrsim_{\Omega,d,\{\Lambda_i\}_{i=1}^5,S,r_0}
\frac{\underline a}{\sqrt{nN}}
\frac{\rho_*^3}{\rho_\gamma^2}
\left(\frac{\rho_\gamma}{2\rho_{\mathrm T}}\right)^{-(T-1)}\\
&\ge
\frac{2\rho_{\mathrm T}\underline a}{\sqrt{nN}}
\left(\frac{\rho_\gamma}{2\rho_{\mathrm T}}\right)^{-(T-2)}
\ge
\frac{\rho_{\mathrm T}\underline a}{\sqrt{nN}}\,2^{T-1}.
\end{align*}
Finally, $T=\lfloor dM^{1/d}/(4e)\rfloor$ and $\widetilde a_{\mathrm T}\le d\ln2/(4e)$ imply
\[
2^{T-1}\ge\frac14 \exp\left(\frac{d\ln2}{4e}M^{1/d}\right) \ge\frac14\exp\bigl(\widetilde a_{\mathrm T}M^{1/d}\bigr).
\]
This proves \eqref{ineq: tanh condition number lower bound} in the low-frequency case and completes the proof.
\end{proof}

\end{document}